%% file: aistats2021.tex
\documentclass[twoside]{article}

\usepackage{silence}
\newcommand{\affiliationText}{
\scriptsize{$^*$ Equal contribution\\ $^1$ Samsung SAIT AI Lab, Montréal \\ $^2$ MILA and DIRO, Université de Montréal \\ $^3$ Princeton University and the Voleon Group \\ $^4$ Institute of Mathematics, EPFL}}

\usepackage[preprint]{aistats2021}
\input{utils/preamble}

\usepackage{natbib}
\begin{document}

\twocolumn[

\aistatstitle{Generalization of Quasi-Newton Methods:\\ Application to Robust Symmetric Multisecant Updates}

\aistatsauthor{ Damien Scieur$^{1,*}$ \And Lewis Liu$^{2,*}$ \And  Thomas Pumir$^3$ \And Nicolas Boumal$^4$ }

\runningauthor{Damien Scieur, Lewis Liu, Thomas Pumir and Nicolas Boumal}
% \runningtitle{Robust Symmetric Multisecant quasi-Newton Updates}
\runningtitle{Generalization of Quasi-Newton Methods: Application to Robust Symmetric Multisecant Updates}

\aistatsaddress{} ]

\begin{abstract}
\input{sections/abstract}
\end{abstract}

\input{sections/introduction}
\input{sections/notations.tex}

\input{sections/related_work.tex}

\input{sections/contributions.tex}

\input{sections/generic_qn.tex}

\input{sections/multisecant.tex}

\clearpage
{
    \small
    \bibliography{utils/bibliography_file.bib}
}
\clearpage

\appendix

\onecolumn

\input{sections/appendix/multisecant_algo}

\clearpage
\input{sections/appendix/projection_psd}

\clearpage
\input{sections/appendix/preconditioner}

\clearpage
\input{sections/appendix/generalized_qn_step}
\clearpage
\input{sections/appendix/convergence_analysis_new}

\clearpage
\input{sections/procrustes}
\clearpage
\input{sections/stability_analysis}

\input{sections/numexp}

\end{document}

%% file: utils/preamble.tex
\usepackage{amsmath}
\usepackage{amsthm}
\usepackage{amssymb}
\usepackage{empheq}
\usepackage{xcolor, color, colortbl}
\usepackage{framed}
\usepackage{pifont}
\usepackage{enumitem}
\usepackage{graphicx} % more modern
\usepackage{nicefrac}
\usepackage{booktabs}
\usepackage{multirow}
\usepackage{rotating}
\usepackage{nccmath}

\usepackage[tikz]{bclogo}
\usepackage{wasysym}

\definecolor{mydarkblue}{rgb}{0,0.08,0.45}
\usepackage[
    linkcolor=mydarkblue,
    citecolor=mydarkblue,
    filecolor=mydarkblue,
    urlcolor=mydarkblue,
    pdfview=FitH]{hyperref}

\usepackage{url}
\usepackage{relsize}

\def\xx{\textbf{x}}
\def\yy{\textbf{y}}
\def\zz{\textbf{z}}

\def\dd{\textbf{d}}
\def\XX{\textbf{X}}

\def\ZZ{\textbf{Z}}
\def\aa{\textbf{a}}
\def\bb{\textbf{b}}

\def\WW{\textbf{W}}

\def\II{\textbf{I}}
\def\yy{\textbf{y}}
\def\vv{\textbf{v}}

\def\ww{\textbf{w}}
\def\zz{\textbf{z}}
\def\SS{\textbf{S}}
\def\BB{\textbf{B}}
\def\AA{\textbf{A}}
\def\CC{\textbf{C}}
\def\MM{\textbf{M}}
\def\CC{\textbf{C}}
\def\DD{\textbf{D}}
\def\GG{\textbf{G}}
\def\PP{\textbf{P}}
\def\TT{\textbf{T}}
\def\VV{\textbf{V}}
\def\QQ{\textbf{Q}}
\def\UU{\textbf{U}}
\def\HH{\textbf{H}}
\def\EE{\textbf{E}}
\def\bSigma{{\boldsymbol \Sigma}}
\def\zzref{\textbf{z}_{\text{ref}}}
\def\yyref{\textbf{y}_{\text{ref}}}

\def\bChi{{\boldsymbol\chi}}

\def\balpha{{\boldsymbol \alpha}}

\def\calR{\mathcal{R}}
\def\calC{\mathcal{C}}
\def\calL{\mathcal{L}}

\def\HHref{{\HH}_{\text{ref}}}
\def\BBref{{\BB}_{\text{ref}}}
\def\MMref{{\MM}_{\text{ref}}}
\def\ZZref{{\ZZ}_{\text{ref}}}

\def\RR{{\mathbb R}}
\def\R{{\mathbb R}}

\renewcommand{\gg}{\boldsymbol{g}}
\renewcommand{\vec}{\mathbf{vec}}

\def\defas{\stackrel{\text{def}}{=}}

\DeclareMathOperator*{\Span}{\mathbf{span}}

\DeclareMathOperator*{\rank}{\mathbf{rank}}

\DeclareMathOperator*{\argmin}{{arg\,min}}

\newcommand{\norm}[1]{\|#1\|}

\definecolor{myblue}{HTML}{D2E4FC}
\definecolor{Gray}{gray}{0.92}

\newtheorem{proposition}{Proposition}
\newtheorem{lemma}{Lemma}
\newtheorem{theorem}{Theorem}

\usepackage{thmtools}
\usepackage{thm-restate}

\usepackage{wrapfig}

\newcommand{\DX}{\Delta\XX}
\newcommand{\DG}{\Delta\GG}

\usepackage{algorithmicx,algorithm}
\usepackage{algcompatible}

%% file: sections/abstract.tex
% Abstract
Quasi-Newton (qN) techniques approximate the Newton step by estimating the Hessian using the so-called secant equations. Some of these methods compute the Hessian using several secant equations but produce non-symmetric updates. Other quasi-Newton schemes, such as BFGS, enforce symmetry but cannot satisfy more than one secant equation. We propose a new type of quasi-Newton symmetric update using several secant equations in a least-squares sense. Our approach generalizes and unifies the design of quasi-Newton updates and satisfies provable robustness guarantees. 

%% file: sections/introduction.tex
\section{Introduction}

We consider second-order methods for unconstrained minimization of a smooth, possibly non-convex function $f\colon\R^{d} \rightarrow \R$.
Despite a locally quadratic convergence rate,
the well-known Newton method iteration %~\citep{Nesterov83Newton}, 
\begin{equation}\label{eq::newton}
\xx_{k+1} = \xx_{k} - \big[ \nabla^{2}f(\xx_{k}) \big]^{-1}\nabla f(\xx_{k})
\end{equation}
is not suitable for large-scale problems, in part because it requires solving a $d\times d$ linear system involving the Hessian at every iteration. To address this issue, quasi-Newton algorithms replace the update rule~\eqref{eq::newton} by
\begin{align}
\xx_{k+1} & = \xx_{k} - \BB_{k}^{-1}\nabla f(\xx_{k}) \quad \text{or}  \nonumber\\
\xx_{k+1} & = \xx_{k} - \HH_{k}^{\hphantom{-1}}\nabla f(\xx_{k}), \label{eq::qN}
\end{align}
where $\BB_{k} \approx \nabla^{2}f(\xx_{k})$ and $\HH_{k} \approx \big[ \nabla^{2}f(\xx_{k}) \big]^{-1}$ are approximations of the Hessian and its inverse (respectively) at $\xx_{k}$. Choosing the right approximation has drawn considerable attention in the optimization literature, notably the DFP update~\citep{Davidon59QNewton},
Broyden method~\citep{Broyden65QNewton}, SR1 update~\citep{Byrdetal96SR1} and the well-known BFGS method~\citep{Broyden70QNewton},~\citep{Fletcher70QNewton},~\citep{Goldfarb70QNewton}~\citep{Shanno70QNewton}.
In general, those methods estimate a matrix $\BB_{k}$ or $\HH_k$ satisfying the \textit{secant} equation
\begin{eqnarray}
\hphantom{\HH_{k}}\nabla f(\xx_{k}) - \nabla f(\xx_{k-1}) & = & \BB_k (\xx_{k} - \xx_{k-1}) \; \text{or}  \nonumber \\
\quad \HH_{k}(\nabla f(\xx_{k}) - \nabla f(\xx_{k-1})) & = & \hphantom{\BB_k} \xx_{k} - \xx_{k-1},\label{eq::secant}
\end{eqnarray}
then perform the quasi-Newton step \eqref{eq::qN}. It is also possible to satisfy \textit{several} secant equations. For instance, the multisecant Type-I and Type-II Broyden methods \citep{fang2009two} find a \textit{non-symmetric} matrix $\BB_k$ or $\HH_k$ satisfying a block of secants: for a memory size $m$ and for $i=k-m+1\ldots k$,
\begin{eqnarray}
\hphantom{\HH_{k}}\nabla f(\xx_{i}) - \nabla f(\xx_{i-1}) & = & \BB_k [\xx_{i} - \xx_{i-1}] \; \text{or}  \nonumber \\
\quad \HH_{k}[\nabla f(\xx_{i}) - \nabla f(\xx_{i-1})] & = & \hphantom{\BB_k} \xx_{i} - \xx_{i-1}.\nonumber
\end{eqnarray}

By contrast, other methods like BFGS and DFP enforce the symmetry of the update, but they satisfy only \textit{one} secant equation, in which case \citet{Powell1986HowBA} showed their high dependence in the step size. Indeed, while BFGS and DFP enjoy an optimal convergence rate on quadratics using exact line-search~\citep{nocedal2006optimization}, \citet{Powell1986HowBA} showed that with a \textit{unitary} step size, these updates converge particularly slowly on a simple quadratic function with just two variables. Moreover, it was also observed that BFGS updates are sensitive to gradient noise, and designing quasi-Newton methods for stochastic algorithms is still a challenge \citep{Byrd2016ASQ,bollapragada2018progressive,bollapragada2019exact,berahas2020investigation}.

Unfortunately, except for quadratic functions~\citep{schnabel1983quasi}, it is usually impossible to find a symmetric matrix that satisfies more than one secant equation.
\cite{Gower2016StochasticBB} adopted Hessian-vector products instead of the secant equations.
Moreover, line search has been shown to be computationally expensive. Finally, the stabilisation procedure for stochastic BFGS usually requires a growing batch size to reduce the gradient noise, making it unpractical in many applications. 

In this paper, we tackle those problems by proposing a symmetric multisecant update, that satisfies the secant equations in a least-squares sense. We show their optimality on quadratics \textit{with unitary stepsize}, and prove their robustness to gradient noise, making them good candidates in the context of stochastic optimization.

%% file: sections/notations.tex
\subsection{Notation} \label{sec:notations}

We use boldface small letters, like $\xx$, to refer to vectors and boldface capital letters, like $\AA$, for matrices. We use $d$ to refer to the \textit{dimension} of the problem, and $m$ for the \textit{memory} of the algorithm (we will see later that $m$ is the number of secant equations). For a function $f\colon\R^{d} \rightarrow \R$, its gradient and Hessian at $\xx$ are denoted by $\nabla f(\xx)$ and $\nabla^{2} f(\xx)$ respectively. Consistently with the notations in the literature, we use $\HH$ to denote an approximation of the \textit{inverse} of the Hessian, while we use $\BB$ to denote an approximation of the Hessian. We denote the usual \textit{Frobenius} norm as $\norm{\cdot}$. Moreover, for any square matrix $\AA \in \R^{d \times d}$ and any positive definite matrix $\WW \in \R^{d \times d}$, we define the norm $\norm{ \AA }_{\WW}$ as 
\begin{equation}\label{eq::weighted_frobenius}
    \textstyle \|\AA\|_\WW = \norm{\WW^{\frac{1}{2}}\AA\WW^{\frac{1}{2}}}.
\end{equation}

We often use the matrices $\XX,\, \GG \in\RR^{d\times m+1}$, that concatenates the iterates and their gradients as follow,
\[
    \XX = [\xx_i, \ldots, \xx_{i+m}], \;\; \GG = [\nabla f(\xx_i), \ldots, \nabla f(\xx_{i+m})].
\]
Also, we define $\CC$, and $\DX$ and $\DG$ as
\[
    \DX = \XX\CC, \qquad \DG = \GG\CC,
\]
where $\CC \in \RR^{m+1\times m}$ is a matrix of rank $m-1$ such that $\textbf{1}_{m+1}^T\CC = 0$, $\textbf{1}_{m+1}$ being a vector of size $m+1$ full of ones. Typically, $\CC$ is the column-difference matrix
\[
    \CC = \left[{\begin{smallmatrix}
    -1 & \hphantom{-}0 & \hphantom{-}0 & \ldots \\
    \hphantom{-}1 & -1 & \hphantom{-}0 & \ldots \\
    \hphantom{-}0 & \hphantom{-}1 & \hphantom{-}-1 & \ldots \\
    &  & \ddots& \ddots \\
    &  &  & \hphantom{-}1 & -1 \\
    &  &  & \hphantom{-}0  & \hphantom{-}1
    \end{smallmatrix}}\right] .% \Rightarrow \DX = \begin{bmatrix}\xx_{i+1}-\xx_{i}\\ \vdots\\ \xx_{i+m}-\xx_{i+m-1}\end{bmatrix},
\]
%and similarly for $\DG$.

%% file: sections/related_work.tex
\subsection{Related work}

The idea of updating an approximation of the Hessian or its inverse can be traced back to~\citet{Davidon59QNewton,Davidon59QNewtonSIAM} with the DFP update. Several updates, such as the Broyden method~\citep{Broyden65QNewton} or the BFGS method~\citep{Broyden70QNewton,Fletcher70QNewton,Goldfarb70QNewton,Shanno70QNewton} have been proposed since then. Notably,~\citet{Dembo82InexactQN},~\citet{Dembo1983TruncatedNewton} proposed to approximately invert the Hessian using a Conjugate Gradient method. Limited memory BFGS (L-BFGS)~\citep{LiuNocdeal1989LBFGS}, where a limited number of vectors are stored for the approximation of the Hessian, has proven to be a powerful type of quasi-Newton method. The use of multisecant equations has also been used in a different context by~\citet{Gower14ActionConstrainedQN} and~\citet{Hennig15Probabilistic}, and their connection with Anderson Acceleration \citep{anderson1965iterative} was studied by \citep{fang2009two}. This connection, combined with recent results on Anderson Acceleration \citep{toth2015convergence,walker2011anderson,rohwedder2011analysis,scieur2016regularized,scieur2018online}, especially in the stochastic \citep{scieur2017nonlinear} and non-smooth \citep{zhang2018globally} settings, may indicates that multisecant methods also enjoy some good theoretical properties.
To scale up second-order methods, recent works focus on stochastic quasi-Newton methods. The use of stochastic quasi-Newton updates has been investigated by~\citet{schraudolph07StoQN},~\citet{mokhtari15aOnlineBFGS},~\citet{Moritz16StoBFGS},~\citet{Byrd2016ASQ} and~\citet{Gower2016StochasticBB}, while approximating the Hessian through sampling methods has been proposed by \citet{ErdogduMontanari15HessianApprox},~\citet{Xu16SubSampledNewton} and~\citet{Agarwal17SecondOrder}, among others.

We now present two popular quasi-Newton updates: the BFGS method, and the multi-secant Broyden method. They will serve as a basis to motivate the needs of generalization of quasi-Newton updates.

\subsubsection{Single secant DFP/BFGS updates}

The BFGS update finds a symmetric matrix $\HH_{k}$ that satisfies the secant equation~\eqref{eq::secant}. Among the many possible solutions, it selects the one closest to $\HH_{k-1}$ in a weighted Frobenius norm \eqref{eq::weighted_frobenius}, specifically,
\begin{align}\label{eq::regularBFGS}
\begin{split}
\HH_{k} =  & ~ \underset{\HH=\HH^T}{\mathrm{argmin}}  \lVert \HH - \HH_{k-1} \rVert_{\WW}\\ 
\text{s.t. }& ~ \HH(\nabla f(\xx_{k}) - \nabla f(\xx_{k-1})) = \xx_{k} - \xx_{k-1}.
\end{split}
\end{align}
where $\WW$ is \emph{any} positive definite matrix such that $\WW(\nabla f(\xx_{k}) - \nabla f(\xx_{k-1})) = \xx_{k} - \xx_{k-1}$ \citep[\S8.1]{nocedal2006optimization} --- a similar claim holds for the update formula of $\BB_k$, known as DFP, whose update reads
\begin{align}\label{eq::regularDFP}
\begin{split}
    \BB_{k} =  & ~ \underset{\BB=\BB^T}{\mathrm{argmin}}  \lVert \BB - \BB_{k-1} \rVert_{\WW^{-1}}\\ 
    \text{s.t. }& ~ \BB(\xx_{k} - \xx_{k-1}) = \nabla f(\xx_{k}) - \nabla f(\xx_{k-1}).
\end{split}
\end{align}
The matrix is then inverted using the Woodbury matrix identity. In the two update rules, the matrices $\WW$ and $\WW^{-1}$ are used implicitly, i.e., we do not need to form $\WW$ to evaluate $\HH_k$ nor $\BB_k$.

Solving~\eqref{eq::regularBFGS} repeatedly, BFGS builds a sequence $\HH_1, \HH_2, \ldots$ of matrices such that each $\HH_{k}$ satisfies the $k$th secant equation. 
%Unfortunately, $\HH_k$ does not, in general, satisfy the previous $k-1$ secant equations. 
While it may satisfy the $k-1$ other secants approximately, the update rule offers no such guarantees. The same holds for the DFP update.

\subsubsection{Multi-secant Broyden updates}\label{multisecant_broyden}

In the case of Broyden updates, we seek a matrix $\BB$ for the type-I, or $\HH$ for the type-II, that satisfies the secant equations only, without any restriction on the symmetry of the estimate. The update of the standard Broyden method reads, for $i=k-m,\,\ldots,k$,
 \begin{align}
\begin{split}
\BB_{k} =  & ~ \underset{\BB}{\mathrm{argmin}}  \lVert \BB - \BB_{k-m} \rVert\\ 
\text{s.t. }& ~ \BB(\xx_{i} - \xx_{i-1}) = \nabla f(\xx_{i}) - \nabla f(\xx_{i-1}),\\
\HH_{k} =  & ~ \underset{\HH}{\mathrm{argmin}}  \lVert \HH - \HH_{k-m} \rVert\\ 
\text{s.t. }& ~ \HH(\nabla f(\xx_{i}) - \nabla f(\xx_{i-1})) = \xx_{i} - \xx_{i-1}. \label{eq::regular_broyden}
\end{split}
\end{align}
As for the DFP update, the matrix $\BB_k$ can also be inverted cheaply. In \citep{fang2009two}, the authors show how to extend this update to the case where we want to satisfy more than one secant equation. However, its solution is generally not symmetric.

%% file: sections/contributions.tex
\subsection{Contributions}

Quasi-Newton methods approximate the Hessian. The previous section shows  they do this in very different ways that seem incompatible given the work of \citet{schnabel1983quasi}.
Despite their differences, they share similarities, such as the idea of secant equations. This leads to the following questions:
\begin{quote}
    \emph{Is it possible to design a generalized framework for quasi-Newton updates encompassing Broyden's, DFP and BFGS schemes?}
\end{quote}
\begin{quote}
    \emph{Can Symmetric and Multisecant techniques be combined into a single update?}
\end{quote}

Our work proposes a positive answer to these questions trough the following contributions.
\begin{itemize}
    \item We propose a general framework that models and generalizes previous quasi-Newton updates.
    \item We derive new quasi-Newton update rules (Algorithm \ref{algo:symmetric_multisecant_qn}), which are \text{symmetric} and take into account \textit{several secant equations}. The bottleneck is an (economic size) Singular Value Decomposition (SVD), whose complexity is linear in the dimension of the problem, therefore comparable to other quasi-Newton methods.
    \item We show the optimality of the convergence rate of any multisecant quasi-Newton update built using our framework, on quadratic functions \textit{without line search}. This improves over the BFGS and DFP updates as they are inefficient with unitary step size on quadratics~\citep{Powell1986HowBA}, and suboptimal if exact line-search is not used.
    \item We introduce novel \textit{robust updates}, that provably reduce the sensitivity to the noise of our quasi-Newton schemes. This robustness property is a direct consequence of considering several secant equations at once.
\end{itemize}

\paragraph{Organization of the paper} In Section \ref{sec:generic_qn} we list the desirable properties of quasi-Newton schemes, and end with a generic quasi-Newton update. The choice of its parameters, like the loss/regularization functions, the preconditioner, the number of secants or the initialization leads to different, existing methods but also to potentially new ones. Then, Section~\ref{sec:robust_multisecant} proposes a novel quasi-Newton scheme (Algorithm \ref{algo:symmetric_multisecant_qn}) based on our framework, combining the ideas of DFP/BFGS and multisecant Broyden methods. This algorithm has the advantage of presenting a regularization term, which controls the stability of the update.

\begin{algorithm}[t]
  \caption{Type-I Symmetric Multisecant step}\label{algo:symmetric_multisecant_qn}
  \begin{algorithmic}[1]
    \STATEx (See Appendix \ref{sec:multisecant_algo} for the type-II version)
    \REQUIRE Function $f$ and gradient $\nabla f$, initial approximation of the Hessian $\BBref$, maximum memory $m$ (can be $\infty$), relative regularization parameter $\bar \lambda$.
    \STATE Compute $g_0 = \nabla f(x_0)$ and perform the initial step
    \[
        \xx_1 = \xx_0-\BB^{-1}_0 \gg_0\vspace{-2ex}
    \]
    \FOR{$t=1,2,\ldots$}
    	\STATE Form the matrices $\DX$ and $\DG$ (see Section \ref{sec:notations}) using the $m$ last pairs $(\xx_i,\nabla f(\xx_i))$.
    \STATE Compute the quasi-Newton direction $\dd$ as \vspace{-1ex}
    \[ 
        \dd_t = -\ZZ_\star^{-1} g_t, \vspace{-1ex}
    \]
    see \eqref{eq:zstar_inv} with $\AA = \DX$, $\DD=\DG$, $\qquad\qquad$ \mbox{$\ZZref = \BBref$}, $\lambda = \bar\lambda \|\AA\|$.
    \STATE Perform an approximate-line search\vspace{-1ex}
    \[
        \xx_{t+1} = \xx_t + h_t\dd_t, \quad h_t \approx \argmin_h f\big( \xx_t + h_t\dd_t \big).
    \vspace{-1ex}
    \]
    \ENDFOR
  \end{algorithmic}
\end{algorithm}

%% file: sections/generic_qn.tex
\section{Generalization of Quasi-Newton} \label{sec:generic_qn}

We have seen in the previous section two different quasi-Newton (qN) updates: one that focuses on the \textit{symmetry} of the estimate, the other on the number of satisfied \textit{secant equations}. In this section, we propose a unified framework to design existing and new qN schemes.

\subsection{Generalized (Multi-)Secant Equations} \label{sec:multisecant_eq}

The central part of qN methods is the secant equation. The idea follows from the linearization of the gradient of the objective function. Indeed, consider the function $f(\xx)$, assumed to be smooth, strongly convex and twice differentiable. The linearization of its gradient around the minimum $\xx_\star$ satisfies
\begin{equation}
    \nabla f(\xx) \approx \underbrace{\nabla f(\xx_{\star})}_{=0} + \nabla^2 f(\xx_{\star})(\xx-\xx_{\star}). \label{eq:linearization_gradient}
\end{equation}
After a ``Newton step'', we get
\[
    \xx - [\nabla^2f(\xx^{\star})]^{-1}\nabla f(\xx) \approx \xx_\star.
\]
Unfortunately, we do not have access to the matrix $\nabla^2f(\xx^{\star})$ as we do not know $\xx_{\star}$. Moreover, solving the linear system $[\nabla^2f(\xx^{\star})]^{-1}\nabla f(\xx)$ may be costly when $d$ is large.

To overcome such issues, consider a sequence $\{\xx_0,\ldots,\xx_{m}\}$ of points at which we have computed the gradients. Then, \eqref{eq:linearization_gradient} can be stated as
\[
    \GG = \nabla^2 f(\xx_{\star}) (\XX - \XX_{\star}),
\]
where $\XX_{\star} = \xx_{\star}\textbf{1}_{m+1}^T$, i.e., the matrix concatenating $m+1$ copies of the vector $\xx_{\star}$. Matrices $\XX$ and $\GG$ are defined in Section \ref{sec:notations}. 

Ideally, the estimate $\BB$ of the Hessian, or the estimate of its inverse $\HH$, has to satisfy the condition
\[
    \GG = \BB (\XX - \XX_{\star}) \;\; \text{ or } \;\; \HH \GG = (\XX - \XX_{\star}).
\]
 However, the dependency on $\xx_{\star}$ makes the problem of estimating $\BB$ or $\HH$ intractable. To remove this problematic dependency, consider a matrix $\CC\in\RR^{m+1\times m}$ of rank $m$ such that $\textbf{1}_{m+1}^T\CC = 0$ (see Section \ref{sec:notations} for an example). After multiplying by $\xx_{\star}$ on the right, we simplify $\XX_{\star}\CC=0$ and we obtain the \textit{multisecant equations}
\begin{equation} \label{eq:multisecant_equation}
    \DG = \BB \DX, \;\; \text{ or } \;\; \HH \DG = \DX,
\end{equation}
where $\DX$ and $\DG$ are defined in Section \ref{sec:notations}. In the specific case where we have only one secant equation, \eqref{eq:multisecant_equation} corresponds exactly to the standard secant equation in \eqref{eq::regularBFGS}. In the case where $\CC$ is the column-difference operator, we obtain the multisecant equations usually used in multisecant Broyden methods.

\subsection{Regularization and Constraints}

The matrices $\BB$ (Broyden Type-I and DFP updates) and $\HH$ (Broyden Type-II and BFGS) are selected so as to minimize the distances w.r.t.\ the reference matrices, called $\BBref$ and $\HHref$ respectively, as shown in~\eqref{eq::regular_broyden}. In the case where there is only a sequence of single secant equations, the reference matrix is taken as being the previous estimate, with an arbitrary initialization. In the case of a multisecant update, the reference matrix is arbitrary. Moreover, in the case of DFP and BFGS, we have in addition a \textit{symmetry} constraint, restraining even more the search space for the estimate of the Hessian. For simplicity, we will consider only the type-I update here, i.e., the estimate $\BB$. 
The formulation for estimate $\HH$ can be easily derived by swapping $\DG$ and $\DX$.

The intuition behind the regularization term is due to the number of degrees of freedom in the problem. The secant equation $\BB\DX = \DG$ defines the behavior of the operator $\BB$, mapping from $\Span\{\DX\}$ to $\Span\{\DG\}$. However, the dimension of these two spans is as most $m<d$. This means we have to define the behavior of $\BB$ \textit{outside} $\Span\{\DX\}$ and $\Span\{\DG\}$, i.e., from $\Span\{\DX\}^{\perp}$ to $\Span\{\DG\}^{\perp}$. 

Since $\BB$ outside the span is not driven by the secant equations, we have to define an  operator $\BBref$, characterizing the default behavior of $\BB$ outside the span of secant equations. This means that, in the case where $\BB$ satisfies exactly the secant equations, $\BB$ reads
\[
    \BB = [\DG\DX^{\dagger}] + \Theta(\II-\PP),
\]
where $\PP$ is the projector to the span of $\DX$, $\DX^{\dagger}$ is a pseudo-inverse of $\DX$, and  $\Theta$ depends on $\BBref$ and constraints (different $\Theta$ lead to different qN updates). In this way, $\BB$ satisfies the secant equation, since  multiplying $\BB$ by $\DX$ gives $\DG$, 
\begin{eqnarray*}
    \BB \DX & = &  \DG\DX^{\dagger}\DX  + \Theta(\II-\PP)\DX.
\end{eqnarray*}
We have $\PP\DX = \DX$, thus $(\II-\PP)\DX = 0$ (by construction of $\PP$). Moreover, $\DG\DX^{\dagger}\DX = \DG$ by definition of the pseudo-inverse.

The way $\BB$ behaves outside the span is thus driven by $\Theta$, which depends on the regularization, the initialization $\BBref$ and the constraints. To make a parallel with machine learning problems, $\Theta$ can be seen as the ``generalization'' (or ``out-of-sample'') term. We give example choices for $\Theta$ in Appendix \ref{sec:example_qn_methods}.

Consider the regularisation function $\calR(\cdot,\BBref)$, assumed to be strictly-convex, whose minimum is attained at $\BBref$, and the convex constraint set $\mathcal{C}$. We can write the qN update estimation problem as
\begin{equation}
    \min_{\BB \in \calC} \calR(\BB,\BBref) \quad \text{subject to } \BB\DX = \DG. \label{eq:exact_qn}
\end{equation}

This approach generalizes the way we define qN updates. Indeed, for instance, we recover DFP by setting $\calR = \|\BB-\BBref\|_{\WW^{-1}}$, $\calC = \mathbb{S}^{d\times d}$ (the set of symmetric matrices), $m=1$ and $\BBref = \BB_{k-1}$ in \eqref{eq:exact_qn}. We also recover the Type-I Broyden method by setting $\calR = \|\BB-\BBref\|$ and $\calC = \RR^{d\times d}$.

\subsection{Generalized QN Update}

A natural extension, given the updates of DFP/BFGS and multisecant Broyden, would be the symmetric multi-secant update. This update would read, for an arbitrary regularization function,
\begin{equation*}
    \min_{\BB \in \mathbb{S}^{d \times d}} \calR(\BB,\BBref) \quad \text{subject to } \BB\DX = \DG.
\end{equation*}
In the case where $m>1$, this multisecant technique seems promising as it combines the advantages of multisecant Broyden and symmetric updates.

% Unfortunately, the system of equations and the constraints in problem~\eqref{eq:exact_qn} are of the form $\BB \DX = \DG$. 
Assuming $\DX, \DG$ have full column rank, these equations always have a solution $\BB$. However, there exists a \textit{symmetric} solution \textit{if and only if} $\DX^T\DG$ is symmetric~\citep{schnabel1983quasi,Don87LME}.

When $\DX^T\DG$ is symmetric, \citet{schnabel1983quasi} derived a multisecant BFGS update rule. 
This assumption indeed holds for quadratic objectives, but not for general objective functions when $m \geq 2$, that is, when we consider more than one secant condition~\citep[Example 3.1]{schnabel1983quasi}. Hence, a naive extension of symmetric quasi-Newton update leads to infeasible problems.

To tackle the problem of infeasible updates, we can relax the constraint on the secant equations by a \textit{loss function} $\calL(\cdot, \DX, \DG)$. We finally end up with the \textit{generalized (type-I and type-II) qN update}
\begin{equation}
    \BB_k = \lim_{\lambda\rightarrow 0} \argmin_{\BB\in\calC} \calL(\BB,\DX,\DG) + \frac{\lambda}{2} \mathcal{\calR}(\BB,\BBref) \tag{GQN-I} \label{eq:gqn_t1}
\end{equation}
\begin{equation}
    \HH_k = \lim_{\lambda\rightarrow 0} \argmin_{\HH\in\calC} \calL(\HH,\DG,\DX) + \frac{\lambda}{2} \mathcal{\calR}(\HH,\HHref) \tag{GQN-II} \label{eq:gqn_t2}
\end{equation}
where we assume that $\calL$ and $\calR$ are strictly convex, and sufficiently simple to have an explicit formula for $\HH_k$. The limits here simply state that we first minimize the loss function, then with the remaining degrees of freedom we minimize the regularization term. In the case where the update \eqref{eq:exact_qn} is feasible, \eqref{eq:gqn_t1}/\eqref{eq:gqn_t2} and \eqref{eq:exact_qn} are equivalent.

\subsection{Preconditioning} \label{sec:preconditioner}

As shown for instance in DFP and BFGS, it is common to use a preconditioner to reduce the dependence of the update to the units of the Hessian. We give here the example for type-II update. The type-I follows immediately by considering $\WW^{-1}$ instead of $\WW$.

The idea of preconditioning is, instead of considering $\HH$, to set
\[
    \MM = \WW^{(1-\alpha)}\HH\WW^{\alpha},
\]
where $\WW$ ideally has the same units as the \textit{Hessian} of the function $f$. For example, in BFGS, $\WW$ is \textit{any} matrix such that $\WW\DX=\DG$, which always exists in the case where $\DX$ and $\DG$ are vectors. Ideally, the preconditioner cancels the units in the update rules, i.e., $\WW$ has to have the same units as the Hessian.

In the case where we consider a preconditioner,
\[
    \MM\WW^{-\alpha} \DX = \WW^{1-\alpha}\DG, \;\; \MMref = \WW^{\alpha-1}\HHref\WW^{-\alpha}.
\]
We now have the \textit{type-II Preconditioned Generalized Quasi-Newton} update 
\begin{align}
    \argmin_{\MM \in \tilde \calC} \;\;\;&  \calL(\MM,\WW^{-\alpha}\DX,\WW^{(1-\alpha)}\DG)  + \frac{\lambda}{2} \mathcal{\calR}(\MM,\MMref) \tag{PGQN-II} \label{eq:gqn_t2_precond}
\end{align}
where $\tilde \calC = \WW^{(1-\alpha)}\calC\WW^{\alpha}$, i.e., the image of the constraint after  application of the preconditioner. To retrieve the update $\HH$, it suffices to solve
\[
    \HH = \WW^{-(1-\alpha)}\MM\WW^{-\alpha}.
\]

\subsection{Rate of Convergence on Quadratics}

Our theorem below shows that generalized qN methods \eqref{eq:gqn_t1} and \eqref{eq:gqn_t2} are optimal on quadratics under mild assumptions, in the sense that their performance is comparable to conjugate gradients.

\begin{theorem}\label{thm::opt_qN_quad}
Consider \textit{any} multisecant quasi-Newton method \eqref{eq:gqn_t2} with unitary step-size and $m=\infty$,
\begin{equation}
    \xx_{k+1} = \xx_k - \HH_k \nabla f(\xx_k) \label{eq:qn_update}
\end{equation}
where $f$ is the quadratic form $(\xx-\xx_\star)^{T}\frac{\QQ}{2}(\xx-\xx_\star)$ for some $Q\succ 0$, and $\HH$ satisfies exactly the secant equations. If the update \eqref{eq:qn_update} is a \textit{preconditioned first-order method}, i.e., there exists a symmetric positive definite matrix $\tilde \HH$ independent of $k$ such that
\[
    \xx_{k+1} \in \xx_0 + \tilde \HH \Span\{\nabla f(\xx_0), \ldots, \nabla f(\xx_k)\}
\]
then $\xx_k=\xx_{\star}$ if $k\geq d+1$; for smaller $k$ the method satisfies the rate
\[
    \| \nabla f(\xx_k) \| \leq \mathcal{O} \Big(\textstyle \frac{1-\sqrt{\kappa}}{1+\sqrt{\kappa}}\Big)^k \| \nabla f(\xx_0) \|,
\]
Where $\kappa$ is the inverse of the condition number of $\tilde \HH \QQ$.
\end{theorem}

The proof can be found in Appendix \ref{sec:convergence_quad}. Notice that, for instance, the multisecant Broyden updates \eqref{eq::regular_broyden} or the multisecant BFGS update \citep{schnabel1983quasi} satisfies the assumptions of Theorem~\ref{thm::opt_qN_quad} if $\BBref$ or $\HHref$ are symmetric positive definite matrices (see Appendix \ref{sec:example_qn_methods}). For all these methods, we have $\tilde \HH = \HHref$ (or $\BBref^{-1}$). This indicates that the initialization is crucial, since a good initial approximation of $\QQ^{-1}$ drastically reduces the condition number $\kappa$.

We have now a generic form of qN update, but it raises some important questions. Which practical losses and regularization functions should we use, and what happens if $\lambda$ does not go to zero? The next section addresses the first point by giving an example that extends (limited memory) DFP and multi-secant Broyden methods. Then, we analyse the robustness of the method when $\lambda$ is non-zero.

%% file: sections/multisecant.tex
\section{Robust Symmetric Multisecant Updates} \label{sec:robust_multisecant}

We now extend the BFGS and multisecant Broyden method into the type-II Symmetric Multisecant Update \eqref{eq:smu-2_extended} below, solving the problem \eqref{eq:gqn_t2_precond} in the special case where the loss and the regularization are Frobenius norms. For simplicity, we do not consider any preconditioner here. The method reads
\begin{align}
    \HH_k \hspace{-0.5ex} = \hspace{-0.5ex} \argmin_{\HH=\HH^T} \| \HH\DX-\DG \|_{F}^2 + \frac{\lambda}{2} \| \HH-\HHref \|^2 \label{eq:smu-2_extended} % \tag{SMU-II}
\end{align}
and its type-I counterpart is $\BB_k^{-1}$, where
\begin{align}
    \BB_k \hspace{-0.5ex} = \hspace{-0.5ex} \argmin_{\BB=\BB^T} \| \BB\DG-\DX \|_{F}^2 + \frac{\lambda}{2} \| \BB-\BBref \|^2 \label{eq:smu-1_extended} 
\end{align}

\paragraph{Explicit Formula} We now solve problem \eqref{eq:smu-2_extended} efficiently. This is an extension of the \textit{symmetric Procrusted problem} from \citep{Higham88Procrustes}. Indeed, \citet{Higham88Procrustes} solves the problem
\[
    \min_{\ZZ=\ZZ^T} \| \ZZ \AA - \DD \|,
\]
where $\AA$ and $\DD$ are $\R^{d\times m}$ matrices, where $m>d$. In our case, we have $m\ll d$, and an extra regularization term, that makes the update formula more complicated. Fortunately, the matrix-vector multiplication $\ZZ \vv$ can still be done efficiently even in our case, the bottleneck being the computation of the SVD of a thin matrix. The next theorem details the explicit formula to compute $\HH_k$ (and its inverse if one wants to use a type-I method).
\begin{theorem}\label{thm:reguralized_procrustes}
    Consider the Regularized Symmetric Procrustes \eqref{eq:reg_sym_procrustres} problem
    \begin{equation}
        \ZZ_{\star} = \argmin_{\ZZ=\ZZ^T} \| \ZZ \AA - \DD \|^2 + \frac{\lambda}{2}\| \ZZ-\ZZref \|^2, \label{eq:reg_sym_procrustres} \tag{RSP}
    \end{equation}
    where $\ZZref$ is symmetric (otherwise, take the symmetric part of $\ZZref$),  $\ZZ,\,\ZZref \in\R^{d \times d}$, and $\AA,\,\DD\in\R^{d\times m}$, $m\leq d$. Then, the solution $\ZZ_{\star}$ is given by
    \begin{equation} \label{eq:zstar}\tag{Sol-RSP}
        \ZZ_{\star} = \VV_1 \ZZ_1 \VV_1^T + \VV_1\ZZ_2  + \ZZ_2^T\VV_1^T + (\II-\PP)\ZZref(\II-\PP)
    \end{equation}
    where
    \begin{align*}
        [\UU,\Sigma,\VV_1] & = \textbf{SVD}(\AA^T,\,\texttt{'econ'}),\;\; \text{(economic SVD)}\\
        \ZZ_1 &= \SS \odot \left[ \VV_1^T \left( \AA\DD^T + \DD\AA^T + \lambda\ZZref \right)\VV_1\right],\\
        \SS   &= \frac{1}{\Sigma^2 \textbf{1}\textbf{1}^T + \textbf{1}\textbf{1}^T \Sigma^2 + \lambda  \textbf{1}\textbf{1}^T}, \\
        \PP & = \VV_1\VV_1^T,\\
        \ZZ_2 &= (\Sigma^2 + \lambda\II)^{-1}  \VV_1^T(\AA\DD^T\hspace{-1ex}+\hspace{-0.5ex}\lambda \ZZref)(\II-\PP).
    \end{align*}
    The fraction in $\SS$ stands for the element-wise inversion (Hadamard inverse), 
    and the notation $\odot$ stands for the element-wise product (Hadamard product). 
    The inverse $\ZZ_{\star}^{-1}$ reads
    \begin{align} \label{eq:zstar_inv}
        \ZZ_{\star}^{-1} \hspace{-0.5ex} & = \hspace{-0.5ex} \EE\left( \ZZ_1-\ZZ_2\ZZref^{-1}\ZZ_2^T  \right)^{\hspace{-0.5ex}-1}\hspace{-0.5ex}\EE^T + (\II-\PP)\ZZref^{-1}(\II-\PP) \nonumber\\
        \EE & = \VV_1 - (\II-\PP)\ZZref^{-1}\ZZ_2^T. \tag{Inv-RSP}
    \end{align}
    % \begin{align} \label{eq:zstar_inv} % For optml
    %     \ZZ_{\star}^{-1} \hspace{-0.5ex} = \hspace{-0.5ex} \EE\left( \ZZ_1-\ZZ_2\ZZref^{-1}\ZZ_2^T  \right)^{\hspace{-0.5ex}-1}\hspace{-0.5ex}\EE^T + (\II-\PP)\ZZref^{-1}(\II-\PP) ;\quad 
    %     \EE = \VV_1 - (\II-\PP)\ZZref^{-1}\ZZ_2^T. \tag{Inv-RSP}
    % \end{align}
\end{theorem}

The type-I update uses the matrix $\ZZ_\star^{-1}$, using $\AA = \DX$ and $\DD = \DG$. The type-II uses instead $\ZZ_\star$, with $\AA = \DG$ and $\DD = \DX$. 

The next proposition shows the complexity of performing one matrix-vector multiplication with $\ZZ_\star$ and its inverse. The bottleneck of the method is the SVD of a $\R^{m\times d}$ matrix, whose complexity is $O(m^2d)$, thus linear in the dimension.

\begin{proposition}
    The complexity of evaluating $\ZZ_\star \vv$ and  $\ZZ_\star^{-1} \vv$ is $O(m^2d)$, assuming $m\ll d$ and that the complexity of $\ZZref\vv$ and $\ZZref^{-1}\vv$ is at most $O(m^2d)$.
\end{proposition}

\paragraph{Robustness} The symmetric multisecant update can be used in two different modes, one that lets $\lambda\rightarrow 0$, the other, biased but more robust, that sets $\lambda > 0$. 

The update formula is slightly simpler when $\lambda = 0$. However, due to the presence of matrix inversion, this may lead to instability issues in some cases, similarly to the BFGS method when
\[
    (\xx_{k+1}-\xx_k)^T(\nabla f(\xx_{k+1})-\nabla f(\xx_k))\approx 0,
\]
i.e., when the step and difference of gradients are close to being orthogonal. In BFGS, such issues are tackled by a filtering step, discarding the update if the scalar product goes below some threshold. Unfortunately, when the gradient is corrupted by some noise, the impact on the BFGS update can be huge. 

In the case where $\lambda > 0$, we can show that our update is robust when $\AA$ and $\DD$ are corrupted.
\begin{proposition}\label{prop:stab}
    Let $\ZZ_\star(\lambda)$ be defined as the solution of \eqref{eq:zstar} for some $\lambda$, and $\ZZ_\star(\lambda) = \lim_{\lambda \rightarrow 0} \ZZ_{\lambda}$. Let $\tilde \AA$, $\tilde \DD $ be a corrupted version of $\AA$ and $\DD$ where
    \[
        \|\AA-\tilde \AA\| \leq \delta_\AA, \quad  \|\DD-\tilde \DD\| \leq \delta_\DD.
    \]
    Finally, let $\tilde \ZZ_\star(\lambda)$ be the solution of \eqref{eq:zstar} using $\tilde \AA$ and $\tilde \CC$. Then, we have
    \[
        \| \tilde \ZZ_\star(\lambda) - \ZZ_\star(0) \| \leq \underbrace{\| \ZZ_\star(\lambda) - \ZZ_\star(0)\|}_{\textbf{Bias}} + \underbrace{\| \tilde \ZZ_\star(\lambda) - \ZZ_\star(\lambda) \|}_{\textbf{Stability}},
    \]
    where
    \begin{align}
    \label{eq:prop_stab1}
        \| \ZZ_\star(\lambda) - \ZZ_\star(0)\| & \leq \frac{\lambda \| \ZZ_{\star}(0)-\ZZref \|}{\sigma_{\min}^2(\AA)+\lambda} ,\\ 
        \label{eq:prop_stab2}
        \| \tilde \ZZ_\star(\lambda) - \ZZ_\star(\lambda) \|  & \leq \mathcal{O}\left(\frac{\delta_\AA + \delta_\DD}{\lambda}\right).
    \end{align}
    % \begin{align} % For ml
    %     \| \ZZ_\star(\lambda) - \ZZ_\star(0)\|  \leq \frac{\lambda \| \ZZ_{\star}(0)-\ZZref \|}{\sigma_{\min}^2(\AA)+\lambda} , \quad 
    %     \| \tilde \ZZ_\star(\lambda) - \ZZ_\star(\lambda) \|   \leq \mathcal{O}\left(\frac{1}{\lambda}(\|\tilde \AA\| + \|\tilde \DD\|)^2\right).
    % \end{align}
\end{proposition}

This suggests that $\lambda$ should satisfy a trade-off to achieve the best performing approximation. Notice that when $\lambda = 0$ in the noise-less case, we recover the optimal $\ZZ_{\star}$, and when $\lambda \rightarrow\infty$, we have $\ZZ_{\star} = \ZZref$.

Our result is called \textit{robust} as we can bound the maximum perturbation without restriction on its magnitude. This is \textit{not} the case in \citep{Higham88Procrustes}, whose main assumption  is $\delta_\AA \leq \sigma_{\min}(\AA)$ (which is extremely restrictive), where $\sigma_{\min}$ is the smallest non-zero singular value of $\AA$.

Since the singular values of $\AA$ are, in practice, often small, it is always recommended to set a small $\lambda$: we will show latter, in the numerical experiments, that even for quadratic functions (i.e., in the ``perturbation-free regime''), a small value of $\lambda$ drastically changes the final result, as this makes the method robust to numerical noise.

\paragraph{Scaling of $\lambda$.} The parameter $\lambda$ has to be scaled w.r.t.\ the problem input. It is clear, from Theorem \ref{thm:reguralized_procrustes}, that the role of $\lambda$ is to regularize the matrix inversion by lower-bounding the eigenvalues of the inverted matrix. Therefore, we advise to set $\lambda = \bar\lambda \|\AA^T\AA\|_2$, i.e., proportional to $\|\AA^T\AA\|_2$. This way, assuming $\sigma_{\min}$ small, the conditioning of $(\AA^T\AA+\lambda\II)^{-1}$ is upper-bounded by $1+1/\bar\lambda$.

\section{Numerical Experiment}

This section compares our symmetric multisecant algorithms to existing methods in the literature. We present in this section only a few experiments concerning stochastic-related experiments: We first compare the quality of the estimate of the Hessian (and its inverse). Then, we compare the speed of convergence when using this estimate to estimate the Newton-step in the case where the gradient is stochastic.

\paragraph{Hessian Recovery} Consider the problem of recovering the inverse of a symmetric Hessian $\QQ^{-1}$ of a quadratic function, that satisfies
\[
    \QQ^{-1}\DG=\DX, \quad \QQ=\QQ^T.
\]
However, we have only access to $\tilde{\DG}$, a corrupted version of $\DG$. This notably happens when the oracle provides stochastic gradients.

In our case, we consider the worst-case $\ell_2$ corruption
\[
    \tilde{\DG} = \UU_{\DG} \max\{\Sigma_{\DG}-\epsilon\cdot \sigma_{1}(\DG),\,0\} \VV_{\DG}^T,
\]
where $\UU_{\DG} \Sigma_{\DG}\VV_{\DG}^T$ is the SVD of $\DG$, and $\epsilon$ is the relative perturbation intensity. When $\epsilon = 1$, the matrix $\tilde{\DG}$ is full of zeros.

We estimate $\QQ^{-1}$ using different techniques, that we compare using the relative residual error
\[
    \text{error}(\QQ^{-1}_{\text{est}}) = \| \QQ^{-1}_{\text{est}}\DG - \DX  \| / \|\DX\|.
\]
Note that, in our error function, we use the noise-free version of $\DG$.

Our baseline is the diagonal estimate, corresponding to the inverse of the Lipchitz constant of $\QQ$, typically used as a step-size in the gradient method. We compare $\ell$-BFGS, Multisecant Broyden updates \citep{fang2009two} and our Type-1 and Type-2 multisecant algorithms, solving respectively \eqref{eq:zstar_inv} and \eqref{eq:zstar} with $\AA = \tilde{\DG} $, $\DD=\DX$, $\BB_0 = \HH_0^{-1} = \|\QQ\|$. The number of secant equations is $50$ and the dimension of the problem is $250$. The results are reported in Figure \ref{fig:recovery}.

\paragraph{Optimization problem} We aim to solve
\begin{equation}
    \min_{\xx\in\R^{d}} f(\xx) \defas \frac{1}{N}\sum_{i=0}^N \ell(\aa_i^T\xx,\bb_i) + \frac{\tau}{2} \|\xx\|^2, 
\end{equation}
where $\ell(\cdot,\cdot)$ is a loss function. The pair $(\AA,\bb)$ is a dataset, where $\aa_i \in \mathbb{R}^d$ is a data point composed by $d$ features, and $b_i$ is the label of the $i^{th}$ data point.

Here, we present the specific case where $\ell$ is a quadratic loss, on the Madelon \citep{guyon2008feature} dataset, with $\lambda = 10^{-2}\|\AA\|$. We solve it using SAGA \citep{defazio2014saga} stochastic estimates of the gradient, with a batch size of 64. We also have other experiments on other datasets, other losses and also on deterministic estimate of the gradient in Appendix \ref{sec:numerics}. We also show the evolution of the spectrum of $\HH_k$ and $\BB_k^{-1}$ in Figure \ref{fig:spectrum_iterates}, Appendix \ref{sec:numerics}.

\begin{figure}[t!]
    \centering
    \includegraphics[width=1\linewidth]{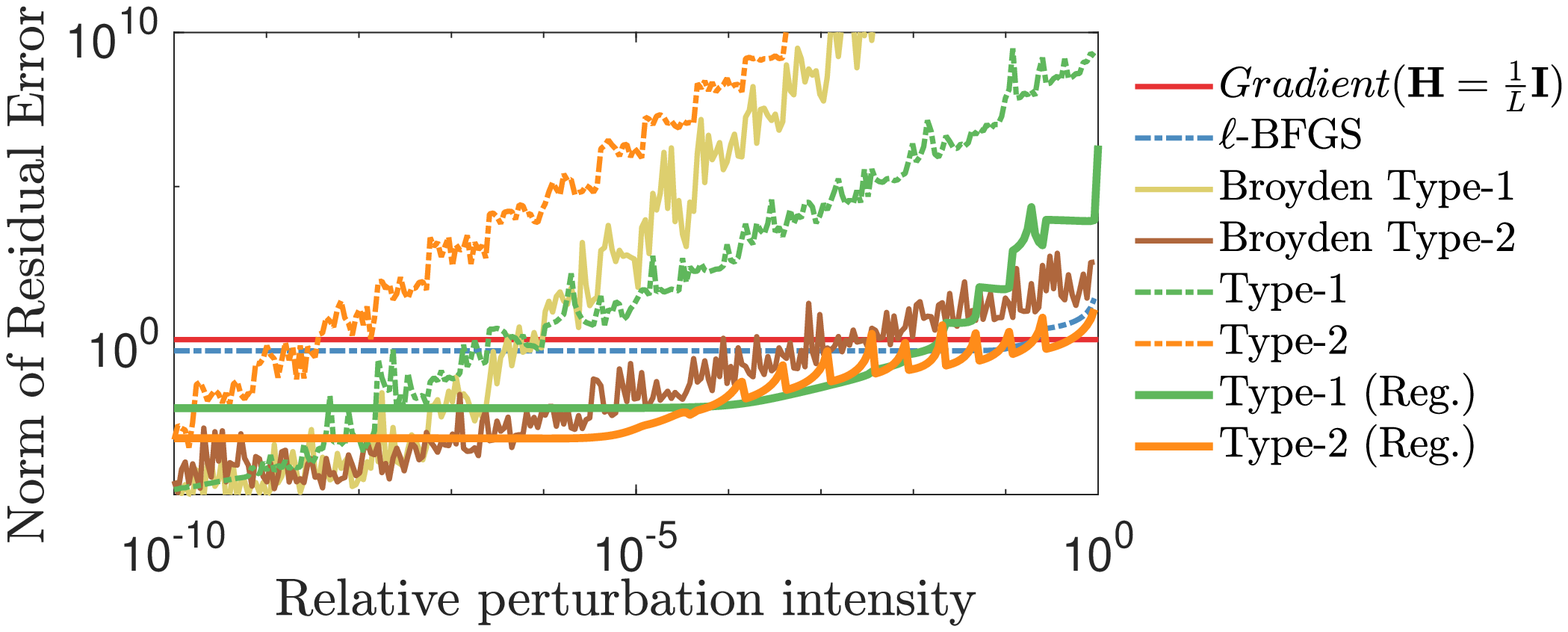}
    \caption{Comparison of different methods to estimate a symmetric matrix. We see that symmetric multisecant methods perform well in a small-noise regime, but quickly get out of control for larger perturbations. This is not the case for their regularized counterpart ($\lambda = 10^{-10}$), clearly showing a more stable behavior. BFGS performs poorly compared to multisecant algorithms, since it can only satisfy one secant equation at a time. Finally, the type-II multisecant Broyden method seems stable, but does not recover a symmetric matrix.}
    \includegraphics[width=0.75\linewidth]{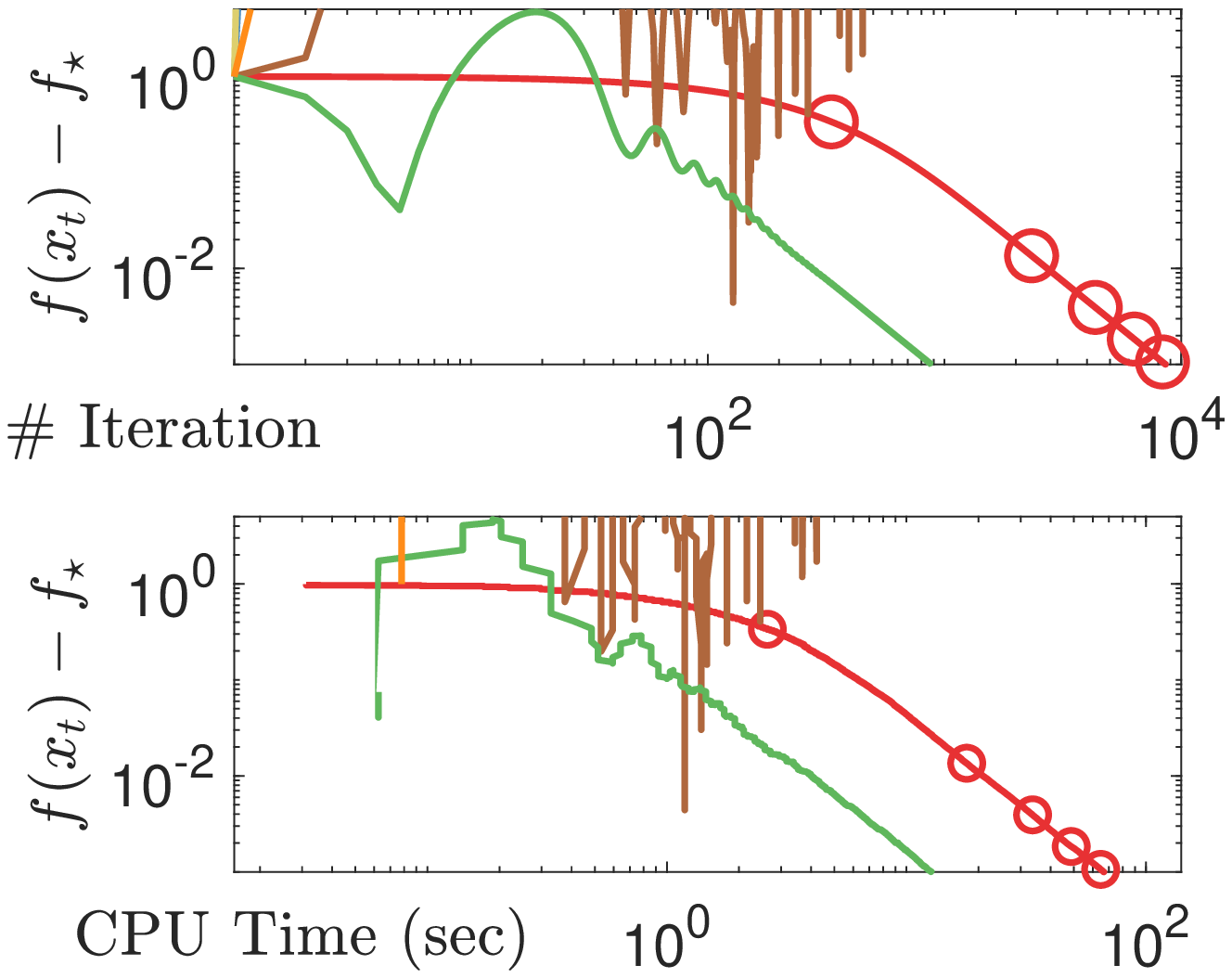}
    \caption{Comparison of the stability of qN methods with stochastic gradients on Madelon dataset. We report the function value of the average of the iterates. The batch size is $64$ points. Since the function is stochastic, we used only unitary stepsizes. The memory is $25$, and the relative regularization $\bar\lambda=10^{-2}$. The condition number is $10^3$. $\ell-$BFGS and the Type-I multisecant Broyden are divergent in this situation. With unitary stepsizes, the regularized symmetric multisecant Type-I method is slightly faster than stochastic gradient.}
    \label{fig:recovery}
\end{figure}

\section{Discussion and Future Directions}

We briefly discuss our contributions and propose possible improvements. Although our approach performs sufficiently well to be competitive with current qN updates, the authors believe the method can be improved in several aspects.

Contrary to BFGS, the update \eqref{eq:smu-1_extended} (resp.\ \eqref{eq:smu-2_extended}) does not guarantee its positive-definiteness when applied to a smooth and strongly convex function. However, for large enough $\lambda$ the matrix is p.s.d.\ given that $\HHref$ (resp.\ $\BBref$) is also positive-definite. Also, it is possible to project a small matrix in \eqref{eq:zstar_inv} (resp.\ \eqref{eq:zstar}) to ensure positive definiteness. We discuss this in more details in Appendix \ref{sec:projection_qn}. The ideal way would be to solve the symmetric Procrustes problem with a semi-definite constraint, but this is still considered as an open problem \citep{Higham88Procrustes}.

A direct consequence of the non positive-definiteness is the lack of robustness guarantees for the Type-I method, that inverts a matrix that is  possibly not positive definite. Therefore, it is probably impossible to bound the smallest eigenvalue, unless we use the robust projection trick in Appendix \ref{sec:projection_qn}. Surprisingly however, in our experiments the Type-I method seems to be the most stable among all updates.

Moreover, we considered here a plain method with \textit{no preconditioner}. In BFGS and DFP updates, the preconditioner $\WW$ is \textit{any} matrix such that $\WW\DX=\DG$ where $\DX$ and $\DG$ are \textit{vectors}. This matrix is used implicitly in the update: all occurrences of $\WW\DX$ are replaced by $\DG$, in a way that $\WW$ disappears. We cannot use a similar trick here, since such matrices do not exist in general when $\DX$ and $\DG$ are matrices \citep{schnabel1983quasi}. We propose in Appendix~\ref{sec:precon_update}
% \textcolor{red}{TODO} 
possible options to include such preconditioners that may potentially improve the method.

It is also possible to consider a general qN step, that takes the direction $\HH \GG \vv$ (or $\BB^{-1} \GG \vv$), where $\vv$ is a vector that sums to one, instead of taking the direction  computed with the latest gradient, $\HH \nabla f(\xx_k)$. In the special case where $\vv$ is full of zeros but one as the last element, this reduces to the standards qN step. We discuss this strategy in Appendix \ref{sec:generalized_qn_step}, and we suspect this technique may reduce even more the impact of the noise on the qN step if $\vv$ is chosen to be the averaging vector $\textbf{1}_m/m$, for instance.

The complexity of the method is somewhat worse than current qN methods: $O(m^2d)$ instead of $O(md)$. The authors believe it may be possible to reduce the complexity by a factor $m$ by using a low-rank SVD update \citep{Brand06FastSVD} and by changing our direct formulas in Theorem \ref{thm:reguralized_procrustes} into recursive ones.

Another interesting direction is the study of the the matrix $\CC$ that forms $\DX$ and $\DG$. We suspect that, in the case where those matrices are corrupted, choosing the right $\CC$ may affect the stability of the method. For instance, it is possible to design $\CC$ to set more weight on some selected secant equations that may be more recent, or that contain less noise.

% \section{Conclusion}
We proposed a novel method with distinct theoretical properties, including symmetry, optimality on quadratics with \textit{unitary stepsize}, and robustness, and which performs encouragingly well in practice. In view of the new questions that multisecant methods raise, we hope our work can add to efforts for the design of possibly other, better-performing \mbox{quasi-Newton} schemes.

%% file: sections/appendix/multisecant_algo.tex
\section{Robust Symmetric Multisecant Algorithms} \label{sec:multisecant_algo}

% We present here the two multisecant algorithm we introduced in this paper.

\begin{algorithm}[h!t]
  \caption{Type-I Symmetric Multisecant step}\label{algo:symmetric_multisecant_qn_type_1_full}
  \begin{algorithmic}[1]
    \REQUIRE Function $f$ and gradient $\nabla f$, initial approximation of the Hessian $\BBref$, maximum memory $m$ (can be $\infty$), relative regularization parameter $\bar \lambda$.
    \STATE Compute $g_0 = \nabla f(x_0)$ and perform the initial step $\xx_1 = \xx_0-\BBref^{-1} \gg_0$
    \FOR{$t=1,2,\ldots$}
    	\STATE Form the matrices $\DX$ and $\DG$ using the $m$ last pairs $(x_i,\nabla f(x_i))$.
    \STATE Compute the qN direction $\dd$ as $\dd_t = -\BB^{-1} g_t$, where
    \begin{align*}
         \BB^{-1} & = \EE\left( \ZZ_1-\ZZ_2\BBref^{-1}\ZZ_2^T  \right)^{-1}\EE^T + (\II-\PP)\BBref^{-1}(\II-\PP), \\
        [\UU,\bSigma,\VV_1] & = \textbf{SVD}(\DX,\,\texttt{'econ'})\\
        \ZZ_1   &= \SS \odot \left[ \VV_1^T \left( \DX\DG^T + \DG\DX^T + \lambda\BBref \right)\VV_1\right],\\
        \SS     &= \frac{1}{\bSigma^2 \textbf{1}\textbf{1}^T + \textbf{1}\textbf{1}^T \bSigma^2 + \lambda  \textbf{1}\textbf{1}^T}, \\
        \PP     & = \VV_1\VV_1^T,\\
        \ZZ_2   &= (\bSigma^2 + \lambda\II)^{-1}  \VV_1^T(\DX\DG^T+\lambda \ZZref)(\II-\PP)\\
        \EE     & = \VV_1 - (\II-\PP)\ZZref^{-1}\ZZ_2^T.
    \end{align*}
    \STATE Perform an approximate-line search: $\xx_{t+1} = \xx_t + h_t\dd_t, \quad h_t \approx \argmin_h f\big( \xx_t + h_t\dd_t \big)$.
    \ENDFOR
  \end{algorithmic}
\end{algorithm}

\begin{algorithm}[h!t]
  \caption{Type-II Symmetric Multisecant step}\label{algo:symmetric_multisecant_qn_type_2_full}
  \begin{algorithmic}[1]
    \REQUIRE Function $f$ and gradient $\nabla f$, initial approximation of the Hessian $\HHref$, maximum memory $m$ (can be $\infty$), relative regularization parameter $\bar \lambda$.
    \STATE Compute $g_0 = \nabla f(x_0)$ and perform the initial step $\xx_1 = \xx_0-\HHref \gg_0$
    \FOR{$t=1,2,\ldots$}
    	\STATE Form the matrices $\DX$ and $\DG$ using the $m$ last pairs $(x_i,\nabla f(x_i))$.
    \STATE Compute the qN direction $\dd$ as $\dd_t = -\HH^{-1} g_t$, where
    \begin{align*}
        \HH & = \VV_1 \ZZ_1 \VV_1^T + \VV_1\ZZ_2  + \ZZ_2^T\VV_1^T + (\II-\PP)\HHref(\II-\PP),\\
        [\UU,\Sigma,\VV_1] & = \textbf{SVD}(\DG^T,\,\texttt{'econ'}),\\
        \ZZ_1 &= \SS \odot \left[ \VV_1^T \left( \DG\DX^T + \DX\DG^T + \lambda\HHref \right)\VV_1\right],\\
        \SS   &= \frac{1}{\Sigma^2 \textbf{1}\textbf{1}^T + \textbf{1}\textbf{1}^T \Sigma^2 + \lambda  \textbf{1}\textbf{1}^T}, \\
        \PP & = \VV_1\VV_1^T,\\
        \ZZ_2 &= (\Sigma^2 + \lambda\II)^{-1}  \VV_1^T(\DG\DX^T+\lambda \ZZref)(\II-\PP)
    \end{align*}
    \STATE Perform an approximate-line search: $\xx_{t+1} = \xx_t + h_t\dd_t, \quad h_t \approx \argmin_h f\big( \xx_t + h_t\dd_t \big)$.
    \ENDFOR
  \end{algorithmic}
\end{algorithm}

%% file: sections/appendix/projection_psd.tex
\section{Positive Definite Estimates} \label{sec:projection_qn}

\subsection{Schur Complement and Robust Projection}

We quickly discuss here a strategy to make the estimate $\HH$ or $\BB^{-1}$ positive definite. If we rewrite $\ZZ$ from Theorem \ref{thm:reguralized_procrustes}, we have
\[
    \ZZ_{\star} = \argmin_{\ZZ=\ZZ^T} \| \ZZ \AA - \DD \|_F^2 + \frac{\lambda}{2}\| \ZZ-\ZZref \|_F^2,
\]
where the matrices $\ZZ_2,\; \ZZref, \; \VV_1$ are defined in \ref{thm:reguralized_procrustes}, and the matrix $\PP=\VV_1\VV_1^T$ is a projector. Let $\VV_2$ be the orthonormal complement of $\VV_1$, i.e., $\II-\PP = \VV_2\VV_2^T$. We can write $\ZZ_\star$ as follow,
\[
    \ZZ_\star = \begin{bmatrix}
    \VV_1 | \VV_2
    \end{bmatrix}\begin{bmatrix}
    \ZZ_1 & \ZZ_2\VV_2 \\
    \VV_2^T\ZZ_2^T & \VV_2^T\ZZref\VV_2
    \end{bmatrix}\begin{bmatrix}
    \VV_1 | \VV_2
    \end{bmatrix}^T
\]
By the Schur complement, the matrix is positive semi-definite if and only if
\[
    \VV_2^T\ZZref\VV_2 \succeq 0 \quad \text{and} \quad \ZZ_1 - (\ZZ_2\VV_2)(\VV_2^T\ZZref\VV_2)(\ZZ_2\VV_2)^T \succeq 0
\]
Since $\VV_2^T\VV_2 = \II$, and because we start with a positive definite $\ZZref$, the only condition is $\ZZ_1 \succeq \ZZ_2\ZZref\ZZ_2^T$. The matrix $\ZZ_1$ is small ($m\times m$) and symmetric, therefore the projection of its eigenvalues to ensure the positive definiteness is cheap.

To project the matrix, let the variable $\bChi$ and $\bChi_0 = \ZZ_1 - \ZZ_2\ZZref\ZZ_2^T$. We have to solve
\[
    \min_{\bChi} \| \bChi - \bChi_0 \|_F \quad s.t. \;\;\bChi \succeq \sigma \II.
\]
This way, we ensure that $\ZZ \succeq \sigma$. Let $\UU \Lambda \UU^T$ the eigenvalue decomposition of $\bChi_0$. the solution $\bChi_\star$ reads
\[
    \bChi_{\star} = \UU \max\{ \Lambda, \, \sigma \II \} \UU^T \quad \text{ (maximum element-wise)}.
\]
We retrieve the modified matrix $\ZZ_1^{+}$ as 
\[
    \ZZ_1^{\sigma} = \bChi_{\star} + \ZZ_2\ZZref\ZZ_2^T.
\]
We call this projection "robust" as we project the matrix s.t. the eigenvalues of $\ZZ$ are strictly positive, if $\sigma > 0$.

\subsection{Robust Positive Definite Type-I Multisecant Update}

We propose here a Robust version of the Multisecant Type-I update. The major stability problem in the  Type-I update is the lack of guarantee that the eigenvalues of $\ZZ$ (i.e., $\BB$) are away from zero. This means, when we will invert $\ZZ$, the eigenvalues of the matrix can be arbitrarily large. On the other side, large eigenvalues of $\ZZ$ are not a problem, since after inversion they will be very close to zero. That means we do not need to compute a regularized version of $\ZZ$, i.e., we do not need to set $\lambda > 0$ to compute $\ZZ$. 

All together, we propose the following strategy: We compute all required matrices to form $\ZZ_\star^{-1}$, but can replace the matrix $\ZZ_1$ by $\ZZ_1^{\sigma}$. This controls the norm of $\ZZ_\star^{-1}$, and ensure its positive definiteness. We let the detailed analysis of the robustness of the method for future work.

\subsection{Robust Positive Definite Type-II Multisecant Update}

Here, the idea is simpler. As we already have the robustness property, it suffice to use the matrix $\ZZ_1^{\sigma}$ directly in the update formula of $\ZZ_{\star}$. Again, we let the detailed analysis of this method for future work.

%% file: sections/appendix/preconditioner.tex
\section{Preconditioned Updates}\label{sec:precon_update}

We discuss in this section several strategies for the choice of the preconditioner $\WW$, presented in Section \ref{sec:preconditioner}. We present here the example for the Type-II method, but everything also applies to the Type-I. We recall that the preconditioner matrix $\WW$ is an estimate of the Hessian, and is applies as follow,
\[
    \MM = \WW^{\alpha} \HH \WW^{(1-\alpha)}
\]
Then, we solve the problem with $\WW^{-\alpha}\DX$ instead of $\DX$, and with $\WW^{(1-\alpha)} \DG$ instead of $\DG$. The estimate $\HH$ is then recovered by solving $\HH = \WW^{-\alpha}\MM\WW^{(\alpha-1)}$.

\subsection{Last estimate}
Since we have computed all matrices $\ZZ_1,\, \ZZ_2,\, \ldots$ for form $\HH_{k-1}$, it is easy to form $\WW=\HH_{k-1}$ and $\WW^{-1} = \HH_{k-1}^{-1}$ to create $\HH_k$, given Theorem \ref{thm:reguralized_procrustes}. Since we only have access to $\HH$ or $\HH^{-1}$, we have to set $\alpha = 1$ or $\alpha = 0$.

\subsection{Successive Preconditioning}

As before, we can use the information stored in the secant equation to compute the preconditioner $\WW$. However, instead of using the previous secant equation, we use the current ones. We have two possibilities here: we can either use the Type-I approximation to compute $\WW$, or the type-II, then compute $\HH$ with this preconditioner. For each of these possibilities, we can use $\WW$ on the left, or the right of $\HH$. At the end, we have 4 possibilities:
\begin{align*}
    \WW = \texttt{Procrustes}(\DX,\, \DG, \, \HHref),\quad \HH & = \texttt{Procrustes}(\WW^{-1}\DG,\, \DX, \, \HHref\WW) \WW  & \text{(Type-I, $\alpha = 0$)},\\
    \HH & = (\WW)^{-1}\texttt{Procrustes}(\DG,\, \WW\DX, \, \WW\HHref)   & \text{(Type-I, $\alpha = 1$)},\\
    \WW^{-1} = \texttt{InvProcrustes}(\DG,\, \DX, \, \HHref^{-1}),\quad \HH & = \texttt{Procrustes}(\WW^{-1}\DG,\, \DX, \, \HHref\WW) \WW  & \text{(Type-II, $\alpha = 0$)},\\
    \HH & = (\WW)^{-1}\texttt{Procrustes}(\DG,\, \WW\DX, \, \WW\HHref)   & \text{(Type-II, $\alpha = 1$)}.\\
\end{align*}

In fact, we can iteratively compute several $\WW$ (since the SVD is already computed, it's only a matter of matrix-vector multiplications). We give here the example of the Type-I, $\alpha = 0$ preconditioner,
\[
    \WW_{i} = \texttt{Procrustes}(\WW_{i-1}^{-1}\DX,\, \DG, \, \HHref\WW_{i-1})\WW_{i-1} \; \text{or} \; \WW_{i} = \WW_{i-1}^{-1}\texttt{Procrustes}(\DX,\, \WW_{i-1}\DG, \, \WW_{i-1}\HHref).
\]

We do not know if this process is convergent, or if it is useful to do several iteration to find the preconditioner. We let these investigations as future work.

\subsection{Semi-Implicit Preconditioning}

We discuss here a semi-implicit strategy, inspired by the preconditioner of BFGS and DFP. Indeed, we assume that there exist a matrix $\WW$ such that
\[
    \WW\DX = \DG.
\]
In such case, we have 4 possibilities for the preconditioned secant equations,
\begin{align*}
    (\WW \HH) \DG = \WW\DX,\\
    (\HH \WW)\WW^{-1} \DG = \DX,\\
    (\WW^{-1} \BB) \DX = \WW^{-1}\DG,\\
    (\BB \WW^{-1})\WW \DX = \DG,
\end{align*}
which gives, if we use the implicit property of $\WW$,
\begin{align*}
    (\WW \HH) \DG = \DG,\\
    (\HH \WW) \DX = \DX,\\
    (\WW^{-1} \BB) \DX = \DX.\\
    (\BB \WW^{-1})\DG = \DG,
\end{align*}

We give here the example when $\WW$ multiplies the secant equation on the left. We left the full study for future work.

\begin{theorem}\label{thm::QN}
The solution of the Type-II semi-implicit preconditioned update is given by
\begin{equation}
    \min_{\HH=\HH^T}\| \WW(\HH-\HHref) \|  \quad \text{s.t.} \quad \WW\HH\DG = \DG \label{eq:preconditioned_problem}
\end{equation}
where $\DG$ is a full column-rank matrix and $\HHref$ a symmetric matrix is given by 
\begin{equation}
    \HH = \WW^{-1}\DG \TT_1^{-1}\DG^T\WW^{-1} + (\II-\PP_1)^T \HHref(\II-\PP_1)
\end{equation}
where
\begin{align*}
    \TT_1 = \DG^T\WW^{-1}\DG, \qquad \text{and} \quad \PP_1 = \DG \TT^{-1}_1\DG^T\WW^{-1} \text{ is a projector}.
\end{align*}
The Type-I solves instead
\[
    \min_{\BB=\BB^T}\| \WW^{-1}(\BB-\BBref) \|  \quad \text{s.t.} \quad \WW^{-1}\BB\DX = \DX,
\]
whose inverse reads
\[
    \BB^{-1} = \DX\TT_2^{-1}\DX^T + \BBref^{-1} - \BBref^{-1}\WW\DX (\DX^T\WW\BBref^{-1}\WW\DX)^{-1}\DX^T\WW\BBref^{-1},
\]
where 
\[
    \TT_2 = \DX^T\WW\DX.
\]
\end{theorem}
% \begin{proof}
%     Let $\ZZ_{\WW} = \WW\HH$ and $\ZZref = \WW\HHref$. The formula of $\ZZ_\WW$ comes directly from Theorem \ref{thm:reguralized_procrustes} with $\lambda = 0$ and $\AA \DG$ and $\DD=\WW\DG$. The matrix $\HH$ is retrieved with $\WW^{-1}\ZZ_{\WW}$. The same thing happen with $\BB$. Notice that, since problem \eqref{eq:preconditioned_problem} is feasible (i.e., $\AA^T\DD = (\AA^T\DD)^T$), the matrix $\ZZ_1$ in Theorem \ref{thm:reguralized_procrustes} reads 
%     \[
%         \ZZ_1 = \DD(\AA^T\DD)^{-1}\AA^T.
%     \]
% \end{proof}
The major problem here is to obtain the matrix $\WW$ or $\WW^{-1}$, which can be approximated using one of the two techniques presented in the previous subsections. Moreover, it would be interesting to consider a robust version of the preconditioned update.

%% file: sections/appendix/generalized_qn_step.tex
\section{Generalized qN step} \label{sec:generalized_qn_step}

We describe here the generalized qN update (Algorithm \ref{algo:generalized_qN_direction}) and qN step (Algorithm \ref{algo:generalized_qN_step}).

\begin{algorithm}[ht]
  \caption{Generalized qN direction}\label{algo:generalized_qN_direction}
  \begin{algorithmic}[1]
    \REQUIRE Matrices $\DG$, $\DX$, regularization $\lambda$, reference matrices $\HHref=\BBref^{-1}$, direction $\ww$.
    \STATEx \textbf{Parameters:} Loss function $\calL$, Regularization function $\calR$, constraint set $\calC$.
    \STATE Solve the problem
    \begin{align*}
        \BB & = \argmin_{\BB\in\calC} \calL\left(\BB\DX,\,\DG\right) + \frac{\lambda}{2}\calR\left(\BB,\, \BBref\right) & \text{(Type-I)} \\
        \HH & = \argmin_{\HH \in \calC} \calL\left(\HH\DG,\,\DX\right) + \frac{\lambda}{2}\calR\left(\HH, \,\HHref\right) &  \text{(Type-II)}
    \end{align*}
    \ENSURE qN direction $\dd = \BB^{-1}\ww$ or $ \dd = \HH\ww$.
  \end{algorithmic}
\end{algorithm}

\begin{algorithm}[ht]
  \caption{Generalized qN step}\label{algo:generalized_qN_step}
  \begin{algorithmic}[1]
    \REQUIRE Sequence of $m+1$ pairs iterates-gradient
    \[
        \{(\xx_0,\,\gg_0),\;(\xx_1,\,\gg_1), \ldots, (\xx_m,\,\gg_m)\}, \qquad \text{where} \;\;\gg_i = \nabla f(\xx_i).
    \]
    \STATEx \textbf{Parameters:} Matrix of differences $\CC\in\R^{m+1,m}$ of rank $m$, vector of coefficients $\vv\in\R^{m+1}$, such that
    \[
        \textbf{1}_{m+1}^T\CC = 0, \qquad \vv^T\textbf{1}_{n+1} = 1.
    \]
    \STATE Form the matrices $\DX$ and $\DG$ as
    \[
        \DX = \XX \CC, \quad \DG = \GG\CC.
    \]
    \STATE Form the gradient direction $\ww$ as
    \[
        \ww = \GG \vv
    \]
    \STATE Call Algorithm \ref{algo:generalized_qN_direction} with $\DG,\,\DX,\ww$ (and other parameters), and retrieve the qN direction $\dd$.
    \STATE Form the next iterate $\xx_+$ using approximate line-search,
    \[
        \xx_{+} = \XX \vv - h^* \dd, \quad \text{where} \;\; h^* \approx \argmin_{h} f(\XX\vv - h\dd).
    \]
  \end{algorithmic}
\end{algorithm}

Algorithm \ref{algo:generalized_qN_step} is inspired by the fact that, if $\QQ$ is the true hessian such that
\[
    \QQ^{-1} \GG = \XX-\XX_\star, \quad where \;\; \XX_\star = \xx_{\star}\textbf{1}^T,
\]
when, if $\HH\approx \QQ^{-1}$ (equivalently $\BB0^{-1}\approx \QQ^{-1}$), we have
\[
    \XX - \HH \GG \approx \XX_\star.
\]
Multiplying both size by $\vv$, where $\vv^T\textbf{1} = 1$, we have $\XX_\star \vv = \xx_\star$ and
\[
    \underbrace{(\XX - \HH \GG)\vv  = \XX\vv - \HH\ww}_{\text{Generalized qN step}} \approx \xx_\star.
\]

%% file: sections/appendix/convergence_analysis_new.tex
\section{Convergence analysis on quadratics} \label{sec:convergence_quad}

We now analyze the convergence speed of the generalized qN step (Algorithm \ref{algo:generalized_qN_step}) when applied on a quadratic function.

\subsection{Setting}

\paragraph{Objective function.} We consider the minimization problem
\begin{align} \label{eq:quadratic_function}
    \min_{\xx} f(\xx) \defas \frac{1}{2}(\xx-\xx_\star)^T\QQ(\xx-\xx_\star) + f_{\star}.
\end{align}
Notice that \ref{eq:quadratic_function} is equivalent to $f(\xx) = \xx^T\QQ\xx + b^T\xx + c$, but the notation in \eqref{eq:quadratic_function} is more convenient. Since the function $f$ is quadratic, we have the following relations,
\begin{equation} \label{eq:relation_dx_dg_quadratics}
    \QQ\DX = \DG, \quad \QQ(\XX-\XX_{\star}) = \GG.
\end{equation}

\paragraph{Algorithm.} We consider the algorithm
\begin{equation}
    \xx_{k+1} = (\XX_k -\HH_k \GG_k)\vv_k, \qquad \text{where} \quad \XX_k = [x_0,\,\ldots,\, x_k], \quad \GG_k = [g_0,\,\ldots,\, g_k], \quad \vv_k : \vv_k^T\textbf{1}_{k+1} = 1, \label{eq:generalized_qn_step}
\end{equation}
and $\HH_k$ is formed by Algorithm \ref{algo:generalized_qN_direction}.

\paragraph{Assumptions} We assume
\begin{itemize}
    \item The spectrum of the true Hessian $\QQ$ is bounded by $\ell \II \preceq \QQ \preceq L\II$, $0<\ell<L$.
    \item (Simplifying assumption) We use only the notation $\HH_k$ for the approximation of the inverse of the Hessian at the iteration $k$, in opposition to making the distinction between $\HH_k$ and $\BB_k^{-1}$.
    \item We assume that the qN approximation satisfies \textit{exactly} the secant equations, i.e.,
    \[
        \HH_k \DG _k= \DX_k.
    \]
    \item The qN method is used with \textit{full memory}, i.e., $\XX_k$ contains all iterates from $0$ to $k$ and grows indefinitely.
    \item The matrices $\DX$ and $\DG$ are full column rank.
\end{itemize}

\subsection{Generic formula of \textbf{H}}

In the case where $\HH$ satisfies exactly the secant equation, the generic formula of $\HH$ reads
\begin{equation} \label{eq:generic_formula}
    \HH_k = \DX_k \DG_k^\dagger + \tilde \Theta_k (\II-\PP_k), \quad \PP_k = \DG_k\DG_k^\dagger,
\end{equation}
where $\Theta_k$ is a matrix that depends on the initialization $\HHref$, the constraints set $\calC$ and the regularization function $\calR$ (but not on the loss since $\HH$ satisfies exactly the secant equations). The notation $\DG_k^\dagger$ is any left pseudo-inverse of $\DG$ that satisfies
\[
    \DG_k^\dagger\DG_k = \textbf{I}_k,
\]
which exists since $\DG_k$ is full column rank. The matrix $\PP$ is a projector such that $\PP \DG = \DG$ and $\PP^2 = \PP$, which is \textit{not} symmetric because it's not an orthonormal projection (unlike most projection matrices). Finally, the matrix $\tilde \HH$ depends on the initialization and constraints of the qN method.

Indeed, if $\HH_k$ satisfies \eqref{eq:generic_formula}, we have that $\HH_k$ satisfies the secant equations since
\[
    \HH \DG = \DX_k \underbrace{\DG_k^\dagger\DG}_{=\II} + \Theta_k \underbrace{(\II-\PP_k)\DG}_{=0} = \DX.
\]

\subsection{Independence of \textbf{v}}

We first show that the generalized qN step \eqref{eq:generalized_qn_step} is (surprisingly) \textit{independent} of the choice of $\vv$. We omit the subscript $k$ in this section for simplicity.

\begin{proposition}[Invariance under $\vv$] \label{prop:invariance_v}
    Let $\tilde \xx_{+}$ and $\xx_+$ be formed by \eqref{eq:generalized_qn_step} using resp. $\tilde \vv$ and $\vv$. Then, $\tilde \xx = \xx$.
\end{proposition}
\begin{proof}
    We first write the difference between $\xx_+$ and $\tilde \xx_+$,
    \[
        \xx_+-\tilde \xx_+ = (\XX -\HH \GG)\underbrace{(\vv - \tilde \vv )}_{\Delta \vv}.
    \]
    However, $\Delta \vv = \vv-\tilde \vv$ is a vector that sum to $0$. Since $\CC$ is a matrix such that
    \[
        \textbf{1}^T \CC = 0, \quad \CC\text{ is full column rank},
    \]
    this means $\CC$ is a basis for all vectors that sum to zero. Therefore, there exists a vector of coefficients $\balpha$ such that $\CC\balpha = \Delta\vv$. Rewriting the difference, we obtain
    \[
        \xx_+-\tilde \xx_+ = (\XX -\HH \GG)\CC\balpha.
    \]
    However, $\GG\CC = \DG$ and $\XX\CC = \DX$. Since $\HH\DG = \DX$, the difference is zero, which prove the statement.
\end{proof}

\subsection{Krylov subspace structure of the iterates}

Before proving the rate of convergence of the qN step, we show that the iterates follows a Krylov structure.
\begin{proposition} \label{prop:span_krylov}
    Assume that, for all $i=0\ldots k$, we have
    \[
        \xx_{i} \in \xx_0 + \tilde \HH \Span\{ \nabla f(\xx_0),\ldots, \nabla f(\xx_{i-1}) \}.
    \]
    In such case,
    \[
        \xx_{i} - \xx_\star \in \xx_0-\xx_\star + \Span\{\tilde \HH \QQ (\xx_0-\xx_\star), (\tilde \HH \QQ)^2 (\xx_0-\xx_\star), \ldots, (\tilde \HH \QQ)^{i-1} (\xx_0-\xx_\star) \}
    \]
\end{proposition}
\begin{proof}
    We prove the result iteratively. For $i=0$, we have
    \[
        \xx_0-\xx_\star = \II (\xx_0-\xx_\star).
    \]
    For $i=1$,
    \[
        \xx_1-\xx_\star \in \xx_0-\xx_\star + \tilde \HH\Span\{ \nabla f(\xx_0) \}
    \]
    Since $\nabla f(\xx_0) = \QQ (\xx_0-\xx_{\star})$,
    \[
        \xx_1-\xx_\star\in \xx_0-\xx_\star + \tilde \HH\Span\{ \QQ (\xx_0-\xx_{\star}) \} \in\xx_0-\xx_\star +  \Span\{ \tilde \HH\QQ (\xx_0-\xx_{\star}) \}.
    \]
    For $i=2$,
    \begin{align*}
        \xx_1-\xx_\star &\in \xx_0-\xx_\star + \tilde \HH\Span\{ \QQ (\xx_0-\xx_{\star}), \,  \QQ (\xx_1-\xx_{\star})\} \\
        & \in \xx_0-\xx_\star + \tilde \HH\Span\{ \QQ (\xx_0-\xx_{\star}), \,  \QQ \left( \xx_0-\xx_\star +  \Span\{ \tilde \HH\QQ (\xx_0-\xx_{\star}\}) \right)\}\\
        & \in \xx_0-\xx_\star + \tilde \HH\Span\{ \QQ (\xx_0-\xx_{\star}), \,  \QQ \left( \tilde \HH\QQ (\xx_0-\xx_{\star}) \right)\}\\
        & \in \xx_0-\xx_\star +  \Span\{ \tilde\HH\QQ (\xx_0-\xx_{\star}), \,  \tilde\HH\QQ \left( \tilde \HH\QQ (\xx_0-\xx_{\star}) \right)\}\\
        & \in \xx_0-\xx_\star +  \Span\{ \tilde\HH\QQ (\xx_0-\xx_{\star}), \,  (\tilde \HH\QQ)^2 (\xx_0-\xx_{\star}) \}
    \end{align*}
    We can repeat the process up to $i$.
\end{proof}

\subsection{Rate of convergence}

We now analyse the rate of convergence of algorithm \eqref{eq:generalized_qn_step} in term of the distance to the solution. 

\begin{theorem} \label{thm:rate_conv}
    Assume that, for all $i=0\ldots k$, we have
    \[
        \xx_{i} \in \xx_0 + \tilde \HH \Span\{ \nabla f(x_0),\ldots, \nabla f(x_{i-1}) \}.
    \]
    Moreover, assume that 
    \[
        \tilde \HH\QQ \text{ is psd, and } \kappa = \frac{\|\tilde \HH \QQ\|}{\|(\tilde \HH \QQ)^{-1}\|} \text{ is bounded.}
    \]
    In such case, the accuracy of the $k-th$ qN step is bounded by
    \[
        \|\xx_k-\xx_\star\| \leq \|\II-\HH_k\QQ\|  \left(\frac{1-\sqrt{\kappa^{-1}}}{1+\sqrt{\kappa^{-1}}}\right)^k \|\xx_0-\xx_\star\|
    \]
\end{theorem}
\begin{proof}
    If we expand the expression, we obtain
    \begin{align}
        \xx_{k+1}   & =  \left( \XX_k - \HH_k \GG_k\right)\vv - \xx_{\star}, \nonumber \\
                    & = \left( \II - \HH_k \GG_k\right)\vv_k, \nonumber\\
                    & = \left( \II - \HH_k\QQ\right) \left( \XX_k - \XX_\star \right) \vv_k. \label{grad_qn_step}
    \end{align}
    By Proposition \ref{prop:invariance_v}, we can take any $\vv_k$ such that $\vv^T\textbf{1} = 1$. In particular, we chose $\vv_k=\vv^\star_k$ such that
    \[
        \vv_k^\star \defas \argmin_{\vv : \vv^T\textbf{1} = 1} \| \left( \XX_k - \XX_\star \right)\vv \|_2^2
    \]
    Therefore,
    \[
        \| \nabla f(\xx_{k+1}) \| \leq \| \II - \HH_k\QQ \| \|\left( \XX_k - \XX_\star \right) \vv_k\| =  \| \II - \HH_k \| \cdot  \min_{\vv:\vv^T\textbf{1}=1}\|\left( \XX_k - \XX_\star \right) \vv\|.
    \]
    By definition of $\left( \XX_k - \XX_\star \right)$, we have
    \[
        \left( \XX_k - \XX_\star \right)\vv_k = \left( \sum_{i=0}^k \vv_i\left(\xx_0-\xx_\star + \Span\{ \nabla f(\xx_0),\ldots,\nabla f(\xx_i) \}  \right)  \right).
    \]
    Since $\vv$ sum to one,
    \[
        \left( \XX_k - \XX_\star \right)\vv_k = \xx_0-\xx_\star + \left( \sum_{i=0}^k \vv_i \Span\{ \nabla f(\xx_0),\ldots,\nabla f(\xx_i) \}    \right).
    \]
    By definition of a $\Span$,
    \[
        \left( \XX_k - \XX_\star \right)\vv_k \in  \xx_0-\xx_\star + \Span\{ \nabla f(\xx_0),\ldots,\nabla f(\xx_i) \}.
    \]
    By Proposition \ref{prop:span_krylov},
    \[
        \left( \XX_k - \XX_\star \right)\vv_k \in  \xx_0-\xx_\star + \Span\{\tilde \HH \QQ (\xx_0-\xx_\star), (\tilde \HH \QQ)^2 (\xx_0-\xx_\star), \ldots, (\tilde \HH \QQ)^{i-1} (\xx_0-\xx_\star) \}.
    \]
    Notice that, because $\GG$ is full rank the $\Span$ is a basis, therefore there is a one-to-one correspondence between the span and $\vv_k$ (i.e., there exists a unique vector $\vv_k$ such that $\vv_k^T\textbf{1}=1$ such that $(\XX_k-\XX_\star)\vv_k$ is a vector of the $\Span$). Using the definition of the $\Span$,
    \[
        \left( \XX_k - \XX_\star \right)\vv_k = \Pi_k(\tilde\HH\QQ)(\xx_0-\xx_\star), \quad \Pi_k \text{ is a polynomial of degree at most $k$, such that }\Pi_k(0) = 1.
    \]
    Therefore,
    \[
        \| \nabla f(\xx_{k+1}) \| \leq \| \II - \HH_k\QQ \|   \cdot \min_{\Pi : \deg(\Pi) \leq k,\, \Pi(0)=1} \left\|\Pi_k(\tilde\HH\QQ)(\xx_0-\xx_\star)\right\|
    \]
    Now, assume that $\tilde \HH \QQ$ is symmetric, p.s.d., and let $\kappa$ be its condition number, i.e.,
    \[
        \kappa = \frac{\|\tilde \HH \QQ\|}{\|(\tilde \HH \QQ)^{-1}\|}.
    \]
    Then, standard result from Krylov subspace gives the bound
    \[
         \min_{\Pi : \deg(\Pi) \leq k,\, \Pi(0)=1} \left\|\Pi_k(\tilde\HH\QQ)(\xx_0-\xx_\star)\right\| \leq \left(\frac{1-\sqrt{\kappa^{-1}}}{1+\sqrt{\kappa^{-1}}}\right)^k \|\xx_0-\xx_\star\|,
    \]
    for $k \leq d$, and converges exactly to 0 when $k \geq d$, which prove the statement.
\end{proof}

\subsection{Example of qN method satisfying the assumptions} \label{sec:example_qn_methods}

We show here that standard qN method satisfies the assumptions of Theorem \ref{thm:rate_conv}. We first show a simpler condition for the method that ensure it satisfies the assumptions of Theorem \ref{thm:rate_conv}.

\begin{proposition}\label{prop:simple_condition}
    Let $\HH$ be any matrix that satisfies the secant equation, which means
    \[
        \HH = \DX\DG^{\dagger} + \Theta (\II-\PP), \quad \DG^{\dagger} : \DX\DG^{\dagger}\DG = \DX, \quad \PP : \PP\DG = \DG.
    \]
    If 
    \[
        \Theta(\II-\PP)\GG\vv \in \tilde{\HH}  \Span\{\GG\},
    \]
    then $\xx_+ \in \xx_0 + \tilde \HH \Span \{ \GG \}$. Moreover, if $\tilde \HH$ is symmetric positive definite then the method satisfies the assumption of Theorem \ref{thm:rate_conv}.
\end{proposition}
\begin{proof}
    We start by expanding the generalized qN step,
    \begin{align*}
        \xx_{+} & = \XX\vv - \DX\DG^{\dagger}\GG\vv - \Theta(\II-\PP)\GG\vv \\
        & = \XX\underbrace{\left( \II - \CC\DG^{\dagger}\GG\right)\vv}_{=\ww} - \Theta(\II-\PP)\GG\vv \\
        & = \XX\ww - \Theta(\II-\PP)\GG\vv.
    \end{align*}
    Notice that $\textbf{1}^T\ww = 1$, since
    \[
        \textbf{1}^T\ww = \textbf{1}^T \left( \II - \CC\DG^{\dagger}\GG\right)\vv = \underbrace{\textbf{1}^T\vv}_{=1} - \underbrace{\textbf{1}^T\CC}_{=0}\DG^{\dagger}\GG\vv.
    \]
    We now show the property recursively. The property is true at $\xx_0$, and assume it's true up to $k$. Therefore,
    \[
        \XX\ww = \XX_k\ww_k = \sum_{i=0}^{k} \ww_i\xx_i \in \underbrace{\sum_{i=0}^{k} \ww_i}_{=1}\xx_0 + \sum_{i=0}^{k} \ww_i\tilde \HH\Span\{\GG_{i-1}\} \quad \text{(recursivity assumption)},
    \]
    Which means $\XX\ww \in \xx_0 + \tilde \HH \Span\{\GG_{k-1}\} $. Therefore, if $\Theta(\II-\PP)\GG\vv \in \tilde{\HH}  \Span\{\GG\}$, we have $\xx_+ \in \xx_0 + \tilde \HH \Span \{ \GG \}$.
\end{proof}

\subsubsection{Multisecant Broyden Type-I}

\paragraph{TL;DR} The method satisfies Theorem \ref{thm:rate_conv} if $\BBref$ is symmetric positive definite.

The Multisecant Broyden Type-I reads
\[
    \BB^{-1} = \BB_0^{-1} + (\DX-\BB_0^{-1}\DG)(\DX^T\BB_0^{-1}\DG)^{-1}\DX^T\BB_0^{-1}
\]
After reorganization,
\[
    \BB^{-1} = \DX \DG^{\dagger} + \BB_0^{-1}(\II\DG\DG^{\dagger}), \quad \DG^{\dagger} = (\DX^T\BB_0^{-1}\DG)^{-1}\DX^T\BB_0^{-1}.
\]
We clearly identity $\Theta(\II-\PP) = \BBref^{-1}(\II-\DG\DG^{\dagger})$. After expansion,
\begin{align*}
    \Theta(\II-\PP)\GG\vv = \BBref^{-1}(\II-\DG\DG^{\dagger})\GG\vv & = \BBref^{-1}\GG(\II-\CC\DG^{\dagger}\GG)\vv,\\
    & = \BBref^{-1}\GG \tilde \vv,\\
    & \in \BBref^{-1} \Span\{ \GG \}.
\end{align*}
Defining $\tilde \HH = \BBref^{-1}$, we have $\Theta(\II-\PP)\GG\vv \in \tilde \HH \Span\{\GG\}$. If $\HHref$ is full rank, symmetric and positive definite, then by Proposition \ref{prop:simple_condition} the method satisfies Theorem \ref{thm:rate_conv}.

\subsubsection{Multisecant Broyden Type-II}

\paragraph{TL;DR} The method satisfies Theorem \ref{thm:rate_conv} if $\HHref$ is symmetric positive definite.

The Multisecant Broyden Type-II update reads
\[
    \HH = \DX\DG^{\dagger} + \HHref(\II-\DG\DG^{\dagger}).
\]
We clearly identity $\Theta(\II-\PP) = \HHref(\II-\DG\DG^{\dagger})$. After expansion,
\begin{align*}
    \Theta(\II-\PP)\GG\vv = \HHref(\II-\DG\DG^{\dagger})\GG\vv & = \HHref\GG(\II-\CC\DG^{\dagger}\GG)\vv,\\
    & = \HHref\GG \tilde \vv,\\
    & \in \HHref \Span\{ \GG \}.
\end{align*}
Defining $\tilde \HH = \HHref$, we have $\Theta(\II-\PP)\GG\vv \in \tilde \HH \Span\{\GG\}$. If $\HHref$ is full rank, symmetric and positive definite, then by Proposition \ref{prop:simple_condition} the method satisfies Theorem \ref{thm:rate_conv}.

\subsubsection{Multisecant BFGS for quadratics}

\paragraph{TL;DR} The method satisfies Theorem \ref{thm:rate_conv} if $\HHref$ is symmetric positive definite.

The multisecant BFGS for quadratics reads
\[
    \HH = \DX\DG^{\dagger} + \DX(\DG^{\dagger})^T(\II-\PP) + (\II-\PP)^T\HHref(\II-\PP) ,\quad \DG^{\dagger} = (\DX^T\DG)^{-1}\DX^T,
\]
which is symmetric if and only if $\DX^T\DG$ is a symmetric matrix. Notice that this reduces to the standard BFGS update when $\DX$ and $\DG$ are vectors. We identify $\Theta(\II-\PP)$ as
\[
    \Theta(\II-\PP) = \left(\DX(\DG^{\dagger})^T + (\II-\PP)^T\HHref\right)(\II-\PP).
\]
After expanding $\PP$,
\[
    \Theta(\II-\PP) = \left(\HHref + \DX\left((\DG^{\dagger})^T- \DG^\dagger\HHref)^T\right)\right)(\II-\PP).
\]
Since $\DX$ already belong to the span, it suffices to show
\[
    \HHref (\II-\PP) \GG\vv \in \tilde \HH \Span\{\GG\}.
\]
Following the same technique as before, we have $\tilde \HH = \HHref$. Therefore, the methods satisfies the assumptions if $\HHref$ is symmetric and positive definite.

%% file: sections/procrustes.tex
% \documentclass{article}
% \usepackage[utf8]{inputenc}
% \usepackage{smile}

% \title{Multisecant quasi-Newton Methods}
% \author{}
% \date{August 2020}

% \usepackage{natbib}
% \usepackage{graphicx}

% \begin{document}

% \maketitle

\section{Symmetric Procrustes Problem}
Consider the following problem, known as Symmetric Procrustes.

\begin{theorem}\label{thm:reguralized_procrustes_app}
    Consider the Regularized Symmetric Procrustes \eqref{eq:reg_sym_procrustres} problem
    \begin{equation}
        \ZZ_{\star} = \argmin_{\ZZ=\ZZ^T} \| \ZZ \AA - \DD \|^2 + \frac{\lambda}{2}\| \ZZ-\ZZref \|^2, \label{eq:reg_sym_procrustres_app} \tag{RSP}
    \end{equation}
    where $\ZZref$ is symmetric (otherwise, take the symmetric part of $\ZZref$),  $\ZZ,\,\ZZref \in\R^{d \times d}$, and $\AA,\,\DD\in\R^{d\times m}$, $m\leq d$, $\lambda>0$. Then, the solution $\ZZ_{\star}$ is given by
    \begin{equation} \label{eq:zstar_app}\tag{Sol-RSP}
        \ZZ_{\star} = \VV_1 \ZZ_1 \VV_1^T + \VV_1\ZZ_2  + \ZZ_2^T\VV_1^T + (\II-\PP)\ZZref(\II-\PP)
    \end{equation}
    where
    \begin{align*}
        [\UU,\bSigma,\VV_1] & = \textbf{SVD}(\AA^T,\,\texttt{'econ'}),\;\; \text{(economic SVD)}\\
        \ZZ_1 &= \SS \odot \left[ \VV_1^T \left( \AA\DD^T + \DD\AA^T + \lambda\ZZref \right)\VV_1\right],\\
        \SS   &= \frac{1}{\bSigma^2 \textbf{1}\textbf{1}^T + \textbf{1}\textbf{1}^T \bSigma^2 + \lambda  \textbf{1}\textbf{1}^T}, \\
        \PP & = \VV_1\VV_1^T,\\
        \ZZ_2 &= (\bSigma^2 + \lambda\II)^{-1}  \VV_1^T(\AA\DD^T\hspace{-1ex}+\hspace{-0.5ex}\lambda \ZZref)(\II-\PP)
    \end{align*}
    The fraction in $\SS$ stands for the element-wise inversion (Hadamard inverse). The inverse $\ZZ_{\star}^{-1}$ reads
    \begin{align} \label{eq:zstar_inv_app}
        \ZZ_{\star}^{-1} \hspace{-0.5ex} & = \hspace{-0.5ex} \EE\left( \ZZ_1-\ZZ_2\ZZref^{-1}\ZZ_2^T  \right)^{\hspace{-0.5ex}-1}\hspace{-0.5ex}\EE^T + (\II-\PP)\ZZref^{-1}(\II-\PP) \nonumber\\
        \EE & = \VV_1 - (\II-\PP)\ZZref^{-1}\ZZ_2^T. \tag{Inv-RSP}
    \end{align}
    % \begin{align} \label{eq:zstar_inv} % For optml
    %     \ZZ_{\star}^{-1} \hspace{-0.5ex} = \hspace{-0.5ex} \EE\left( \ZZ_1-\ZZ_2\ZZref^{-1}\ZZ_2^T  \right)^{\hspace{-0.5ex}-1}\hspace{-0.5ex}\EE^T + (\II-\PP)\ZZref^{-1}(\II-\PP) ;\quad 
    %     \EE = \VV_1 - (\II-\PP)\ZZref^{-1}\ZZ_2^T. \tag{Inv-RSP}
    % \end{align}
\end{theorem}

\begin{proof}
We begin by deriving the solution of~\eqref{eq:reg_sym_procrustres_app}. By taking the transposition of the matrices inside the Frobenius norm of the first term in~\eqref{eq:reg_sym_procrustres_app}, we obtain the equivalent problem
\begin{align}
    \min_{\ZZ=\ZZ^T\in\RR^{d\times d}}\|\AA^{T} \ZZ - \DD^{T}\|^2 + \frac{\lambda}{2} \| \ZZ-\ZZref \|_F^2.
\end{align}

We write the (full) singular value decomposition of $\AA^{T} $ as 
\begin{equation}
    \UU
    \begin{bmatrix}
    \bSigma & 0 
    \end{bmatrix}
    \underbrace{\begin{bmatrix}
    \VV_1^T \\ \VV_2^T
    \end{bmatrix}}_{=\VV},
\end{equation}
where $\UU \in \RR^{m \times m}$, $\VV \in \RR^{d \times d}$ are orthogonal matrices, $\bSigma \in  \RR^{m \times m}$ is a diagonal matrix with nonnegative entries, and $\VV_1 \in \RR^{m\times d},\,\VV_2 \in \RR^{d-m\times d} $. Thus, we obtain another problem equivalent to~\eqref{eq:reg_sym_procrustres_app}, that reads
\begin{align}\label{eq::procrustes_1}
    & \min_{\tilde \ZZ= \tilde \ZZ^T\in\RR^{d\times d}}\|[\bSigma,\,0] \tilde\ZZ - \tilde \DD^{T}\|^2 + \frac{\lambda}{2} \| \tilde \ZZ-\tilde \ZZref \|_F^2,\\
    \text{where} \quad & \tilde \ZZ = \VV \ZZ\VV^T, \nonumber\\
    & \tilde \DD = \UU^T\DD^T\VV, \nonumber \\
    & \tilde \ZZref = \VV^T\ZZref\VV \nonumber.
\end{align}
Equation \eqref{eq::procrustes_1} is equivalent to \eqref{eq:reg_sym_procrustres_app} after multiplying the inside of the norm bu $\UU^T$ on the left, and $\VV$ on the right, since the Frobenius norm is invariant to orthonormal transformation. We now decompose the matrices in blocks as follow,
\begin{align}
    \tilde \ZZ = \begin{bmatrix}
    \tilde \ZZ_1 &  \tilde \ZZ_D \\ \tilde \ZZ_D^T & \tilde \ZZ_2
    \end{bmatrix} \qquad 
    \tilde \DD = \begin{bmatrix}
    \tilde \DD_1 & \DD_2
    \end{bmatrix} \qquad 
    \tilde \ZZref = \begin{bmatrix}
    (\tilde \ZZref)_1 &  (\tilde \ZZref)_D \\ (\tilde \ZZref)_D^T & (\tilde \ZZref)_2
    \end{bmatrix}
\end{align}
where $\ZZ_1, \, (\tilde \ZZref)_1,\, \tilde \DD_1 \in \RR^{m\times m}$,  $\ZZ_2, \, (\tilde \ZZref)_2 \in \RR^{d-m\times d-m}$, $\ZZ_D, \, (\tilde \ZZref)_D, \, \DD_2 \in \RR^{m\times d-m}$. Hence, we can problem~\eqref{eq::procrustes_1} as
\begin{align}
& \min_{\tilde \ZZ= \tilde \ZZ^T\in\RR^{d\times d}}\|[\bSigma,\,0] \tilde\ZZ - \tilde \DD^{T}\|^2 + \frac{\lambda}{2} \| \tilde \ZZ-\tilde \ZZref \|_F^2,\nonumber \\
&\quad = \min_{\tilde \ZZ_1= \tilde \ZZ_1^T, \, \tilde\ZZ_2=\tilde\ZZ_2^T,\, \ZZ_D} \norm{[\bSigma \tilde{\ZZ}_{1}, \bSigma \tilde{\ZZ}_{D}]  - [\tilde{\DD}_{1},\tilde{\DD}_{2}]}^{2}
+ \frac{\lambda}{2} \left(\| \tilde{\ZZ}_1-(\tilde\ZZref)_1 \|^2 + 2\| \tilde{\ZZ}_D-(\tilde\ZZref)_D \|^2 + \| \tilde{\ZZ}_2-(\tilde\ZZref)_2 \|^2\right) \nonumber \\
& \quad = \min_{\tilde \ZZ_1= \tilde \ZZ_1^T} \|\bSigma \tilde{\ZZ}_{1} - \DD_1\|^2 + \frac{\lambda}{2} \|\tilde \ZZ_1 -(\tilde \ZZref)_1\|^2 \label{eq:procrustes_i}\tag{i}\\
& \qquad + \min_{\ZZ_D} \|\bSigma \tilde{\ZZ}_{D} - \DD_2\|^2 + \lambda \|\tilde \ZZ_D -(\tilde \ZZref)_D\|^2 \label{eq:procrustes_ii}\tag{ii}\\
& \qquad +  \min_{\tilde \ZZ_21= \tilde \ZZ_2^T} \frac{\lambda}{2} \|\tilde \ZZ_2 -(\tilde \ZZref)_2\|^2\label{eq:procrustes_iii}\tag{iii}
% &\quad = \underbrace{\norm{\bSigma \tilde{\ZZ}_{1} - \tilde{\DD}_{1}}^{2} + \lambda \norm{\tilde{\ZZ}_{1}-(\tilde{\ZZ}_{0})_{1}}^2}_{\displaystyle \text{(i)}}
% + \underbrace{\norm{\bSigma \tilde{\ZZ}_{D} - \tilde{\DD}_{2}}^{2} 
% + 2\lambda \norm{\tilde{\ZZ}_{D} - (\tilde{\ZZ}_{0})_{D}}^2}_{\displaystyle \text{(ii)}}+ \underbrace{\lambda \norm{\tilde{\ZZ}_{2} - (\tilde{\ZZ}_{0})_{2}}^2}_{\displaystyle \text{(iii)}}.
\end{align} 
Hence, we derive the solution to~\eqref{eq:reg_sym_procrustres_app} by minimizing three independent terms as below.

\noindent{\bf Term \eqref{eq:procrustes_iii}:} The term 
\begin{equation*}
\underset{\tilde{\ZZ}_{2} = (\tilde{\ZZ}_{2})^T}{\mathrm{argmin}}\frac{\lambda}{2} \norm{\tilde{\ZZ}_{2} - (\tilde\ZZref)_{2}}^2
\end{equation*}
imposes the constraint $\tilde{\ZZ}_{2} = (\tilde{\ZZref})_{2}$.
In other words, we have
\begin{align}\label{eq::procrustes_z2}
  \tilde{\ZZ}_{2} = \VV_{2}^T \ZZ_{0} \VV_{2}  .
\end{align}

\noindent{\bf Term \eqref{eq:procrustes_ii}:} The term
\[
    \min_{\ZZ_D} \|\bSigma \tilde{\ZZ}_{D} - \DD_2\|^2 + \lambda \|\tilde \ZZ_D -(\tilde \ZZref)_D\|^2
\]
is a simple regularized least-square, which can be solved by setting the derivative to zero. Therefore,
\begin{align} \label{eq:zd}
    \tilde \ZZ_D = (\bSigma^T\bSigma + \lambda \II)^{-1} (\DD_2+\lambda (\tilde \ZZref)_D)
\end{align}

\noindent{\bf Term (i):} In what follows, we solve the problem (similar to the one in \citep{Higham88Procrustes})
\begin{equation*} \underset{\tilde{\ZZ}_{1} = (\tilde{\ZZ}_{1})^T \in \R^{m \times m}}{\min} \norm{\bSigma \tilde{\ZZ}_{1} - \tilde{\DD}_{1}}^{2} + \lambda \norm{\tilde{\ZZ}_{1}-(\tilde{\ZZ}_{0})_{1}}^2, % \label{eq:block_z1}
\end{equation*}
We first rewrite the optimization problems in terms of the entries in $\tilde \ZZ$ as below, using the fact that $\tilde \ZZ_1$ is symmetric,
\begin{align*}
    \min_{\tilde \ZZ = \ZZ^T\in\RR^{m\times m}} &\sum_{i=1}^m (\sigma_i(\tilde \ZZ_1)_{ii} -  (\tilde \DD_1)_{ii} )^2 +\sum_{i=1}^m \sum_{j=i+1}^m \left( \big( \sigma_i (\tilde \ZZ_1)_{ij} - (\tilde \DD_1)_{ij}\big)^2 + \big(\sigma_j (\ZZ_1)_{ij} -  (\tilde\DD_1)_{ji}\big)^2 \right) \notag\\
    &+ \lambda \left( \sum_{i=1}^m \left( (\tilde \ZZ_1)_{ii} - (\tilde\ZZref)_{ii} \right)^2 + \sum_{i=1}^m\sum_{j=i+1}^m \left(\left( (\tilde \ZZ_1)_{ij} - (\tilde\ZZref)_{ij}\right)^2 +  \left( (\tilde \ZZ_1)_{ij} - (\tilde\ZZref)_{ji}\right)^2\right) \right).
\end{align*}  
By setting the derivative w.r.t. $z_{ij}$, we obtain for $\lambda > 0$
\begin{align*}
    (\tilde \ZZ_{1})_{ij} = \frac{\sigma_i (\tilde \DD_1)_{ij} + \sigma_j (\tilde \DD_1)_{ji} +  \lambda ((\tilde \ZZref)_{ij} + (\tilde \ZZref)_{ji}) }{\sigma_i^2 + \sigma_j^2 + 2\lambda},
\end{align*}
Since $\bSigma\tilde\DD^T = \bSigma\UU^T\DD^T\VV_1^T = \VV_1\AA\DD^T\VV_1^T$, We can equivalently write
\begin{align} \label{eq::procrustes_z1}
    \tilde \ZZ_1  & =  \left( \frac{1}{\Sigma^2 \textbf{1}\textbf{1}^T+\textbf{1}\textbf{1}^T\Sigma^2 + 2\lambda\textbf{1}\textbf{1}^T } \right) \odot \VV_1^T\left( \AA \DD^T  + 
    \DD\AA^T + \lambda (\ZZref+\ZZref^T)\right)\VV_1, \\
\end{align}
where $\odot$ is the Hadamard product computing the product element-wise.

\textbf{Summing the terms together.}
% \noindent{\bf Term (ii):} In addition, the problem 
% \begin{align}
% \underset{\tilde{\ZZ}_{D} \in \mathbb{R}^{m \times d - m}}{\min} \norm{\bSigma \tilde{\ZZ}_{D} - \tilde{\DD}_{2}}^{2} + 2\lambda \norm{\tilde{\ZZ}_{D} - (\tilde{\ZZ}_{0})_{D}}^2
% \end{align}
% has a closed form solution given by
% \begin{align}
%     \tilde \ZZ_D & = \big( \bSigma^\top\bSigma + 2\lambda \II \big)^{-1} \big( \bSigma^\top \tilde \DD_2 + 2\lambda(\tilde \ZZ_0)_D \big). \\
%      \ZZ_D & = \VV_1 \tilde\ZZ_D \VV_2^T \notag\\
%      & = \VV_1\big( \bSigma^\top\bSigma + 2\lambda \II \big)^{-1} \big( \bSigma^\top \tilde \DD_2 + 2\lambda(\tilde \ZZ_0)_D \big) \VV_2^T\\
%   & = \VV_1\big( \bSigma^\top\bSigma + 2\lambda \II \big)^{-1} \big( \bSigma^\top \UU^T \DD^T \VV_2 + 2\lambda \VV_1^T (\ZZ_0)_D \VV_2 \big) \VV_2^T.
% \end{align}
% Since we have $\VV_2 \VV_2^T = \II  -\VV_1\VV_1^T = \II - \PP$,
% \begin{align}\label{eq::zd}
%      \ZZ_D & = \VV_1\big( \bSigma^\top\bSigma + 2\lambda \II \big)^{-1} \big( \bSigma^\top \UU^T \DD^T + 2\lambda \VV_1^T \ZZref  \big) (\II-\PP)\notag\\
%   & = \VV_1\big( \bSigma^\top\bSigma + 2\lambda \II \big)^{-1} \VV_1^T\big( \AA \DD^T + 2\lambda \ZZref  \big) (\II-\PP)
% \end{align}
% In the case where $Z_0$ is a diagonal matrix, and since $\VV_1^T(\II-\PP) = 0$,
% \[
%     \ZZ_D = \VV_1\big( \bSigma^T \bSigma + 2\lambda \II \big)^{-1} \VV_1^T\big( \AA \DD^T \big) (\II-\PP).
% \]\\\
From equations \eqref{eq::procrustes_z2}, \eqref{eq:zd} and \eqref{eq::procrustes_z1}, the solution can be written as
\begin{align}
    \ZZ_\lambda &= \begin{bmatrix}
    \VV_1  & \VV_2    
    \end{bmatrix}
    \begin{bmatrix}
\tilde{\ZZ}_{1} & \tilde{\ZZ}_{D}\\(\tilde{\ZZ}_{D})^{T} & \tilde{\ZZ}_{2}
    \end{bmatrix}
    \begin{bmatrix}
    \VV_1  & \VV_2    
    \end{bmatrix}^T \notag\\
    &=\VV_{1}\tilde{\ZZ}_{1}\VV_{1}^{T}
+ \VV_{1}\tilde{\ZZ}_{D}\VV_{2}^{T}
+ \VV_{2}\tilde{\ZZ}_{D}^{T}\VV_{1}^{T}
+ \VV_{2}\tilde{\ZZ}_{2}\VV_{2}^{T} \notag\\
&= \ZZ_1
+ \ZZ_D + \ZZ_D^T + (\II - \PP) \ZZref (\II - \PP),
\label{eq::opt_z}
\end{align}
where $\PP = \VV_1 \VV_1^T = \II - \VV_2 \VV_2^T$ and 
$\ZZ_D = \VV_1\big( \bSigma^T\bSigma + 2\lambda \II \big)^{-1} \VV_1^T\big( \AA \DD^T + 2\lambda (\ZZ_0)_D  \big) (\II-\PP)$, 
and $\ZZ_1 = \VV_{1}\tilde{\ZZ}_{1}\VV_{1}^{T}$.

Below we compute the inverse of $\ZZ_*$.
Since
\begin{align}\label{eq::H}
\ZZ_* &= \VV
\begin{bmatrix}
\tilde{\ZZ}_{1} & \tilde{\ZZ_D}\\\tilde{\ZZ_D}^{T} & \tilde{\ZZ_2}\\
\end{bmatrix}
\VV^{T} \\
&= \VV \tilde{\ZZ}\VV^{T}, \nonumber
\end{align}
we can write  
\begin{equation*}
\ZZ_*^{-1} = \VV\tilde{\ZZ}^{-1}\VV^{T}.
\end{equation*}
By the Woodbury matrix identity~\citep{Woodbury50Indentity},
% and~\eqref{eq:solution_blocks}, 
we have
\begin{equation}\label{eq::woodburyZ}
\tilde{\ZZ}^{-1} = \begin{bmatrix} \MM_{1} & -\MM_{1}\tilde{\ZZ}_D\tilde{\ZZ}_2^{-1} \\
-\tilde{\ZZ}_2^{-1}\tilde{\ZZ}_D^{T}\MM_{1} & \tilde{\ZZ}_2^{-1} + \tilde{\ZZ}_2^{-1}\tilde{\ZZ}_D^{T}\MM_{1}\tilde{\ZZ}_{D}\tilde{\ZZ}_{2}^{-1} \end{bmatrix},
\end{equation}
with $\MM_{1} = (\tilde{\ZZ_1}- \tilde{\ZZ_D} \tilde{\ZZ_2}^{-1}\tilde{\ZZ}_D^{T})^{-1}$.
Hence $\ZZ_*^{-1} = \VV\tilde{\ZZ}^{-1}\VV^{T}$ can be rewritten as
 \begin{align}\label{invW}
 \begin{split}
\ZZ_*^{-1} &= \VV_{1}\MM_{1}\VV_{1}^{T} + \VV_{2}\tilde{\ZZ_2}^{-1}\tilde{\ZZ}_D^{T}\MM_{1}\tilde{\ZZ_D}\tilde{\ZZ_2}^{-1}\VV_{2}^{T}\\
&+ \VV_{2}\tilde{\ZZ_2}^{-1}\VV_{2}^{T} - \VV_{1}\MM_{1}\tilde{\ZZ_D}\tilde{\ZZ_2}^{-1}\VV_{2}^{T} - \VV_{2}\tilde{\ZZ_2}^{-1}\tilde{\ZZ_D}^{T}\MM_{1}\VV_{1}^{T}\\
&= \QQ\MM\QQ^T + (\II-\PP)\ZZ_0^{-1}(\II-\PP),
\end{split}
\end{align}
 where 
 $\MM = \left( \ZZ_1 - \ZZ_D\ZZ_0^{-1} \ZZ_D^T\right)^{-1}$ and $\QQ = \VV_1-(\II-\PP)\ZZ_0^{-1}\ZZ_D^T$.
% For a proof that for any matrix $\DD\in \mathbb{R}^{d\times m}$.
\end{proof}

% \end{document}

%% file: sections/stability_analysis.tex
\section{Proof of Proposition~\ref{prop:stab}}
In this section, we divide the proof of Proposition~\ref{prop:stab} into Lemma~\ref{lem:perturb_lambda} and Lemma~\ref{lem:stab_ad}, which correspond to the effect of nonzero $\lambda$ for~\eqref{eq:prop_stab1} and the perturbation of $\AA$ and $\DD$ for~\eqref{eq:prop_stab2}, respectively.
\subsection{Effect of regularization}
\begin{lemma}\label{lem:perturb_lambda}
Let 
\begin{align} \label{eq::procrustes_problem_zero}
    \ZZ_* = \lim_{\lambda\to 0}\argmin_{\ZZ=\ZZ^T \in \mathbb{R}^{d\times d}} \| \ZZ\AA-\DD \|_F^2 + \lambda \| \ZZ - \ZZref \|_F^2
\end{align}
be the solution to the procrustes problem with $\lambda$ going to 0, and $\ZZ_\lambda$ be the solution to~\eqref{eq:reg_sym_procrustres_app} given $\lambda > 0$. Then, it holds that
\begin{align}
    \| \ZZ_\lambda - \ZZ_\star\|_F & \leq \frac{5\lambda \| \ZZ_{\star}-\ZZref \|_F}{\sigma_{\min}^2(\AA)+\lambda}.
\end{align}
\end{lemma}

\begin{proof}
We rewrite~\eqref{eq::opt_z}
for $\ZZ_\lambda$ and $\ZZ_*$ respectively,
\begin{align}
     \ZZ_\lambda & = (\ZZ_\lambda)_1
+ (\ZZ_\lambda)_D + (\ZZ_\lambda)_D^T + (\II - \PP) \ZZref (\II - \PP), \\
\ZZ_* & = (\ZZ_*)_1
+ (\ZZ_*)_D + (\ZZ_*)_D^T + (\II - \PP) \ZZref (\II - \PP).
\end{align}
% Therefore,
% \begin{align}
% \ZZ_* - \ZZ_\lambda = 
% \VV\begin{bmatrix}
% \tilde{\ZZ}_{1} & \tilde{\ZZ_D}\\\tilde{\ZZ_D}^{T} & \tilde{\ZZ_2}\\
% \end{bmatrix}\VV^{T} - \VV\begin{bmatrix}
% \tilde{\ZZ}_{1} & \tilde{\ZZ_D}\\\tilde{\ZZ_D}^{T} & \tilde{\ZZ_2}\\
% \end{bmatrix}
% \VV^{T}
% \end{align}
With such notations, we have by triangle inequality, 
\begin{align}\label{eq::lambda_pert}
    \|\ZZ_\lambda - \ZZ_*\|_F 
    & \leq
    \underbrace{\|(\ZZ_\lambda)_1 - (\ZZ_*)_1 \|_F}_{\displaystyle\text{(i)}} + 2\underbrace{\|(\ZZ_\lambda)_D - (\ZZ_*)_D\|_F}_{\displaystyle\text{(ii)}}.
\end{align}
To simplify notations, we define $\max|\XX|$ and $\min|\XX|$ as the maximum and minimum entry with the absolute value of matrix $\XX$, respectively.

For term (i), by~\eqref{eq::procrustes_z1} and the symmetry of $\ZZref$ we have
\begin{align}\label{eq::ub_lambda1}
    \|(\ZZ_\lambda)_1 - (\ZZ_*)_1 \|_F & = 
    \bigg\|
    \bigg(\frac{1}{\Sigma^2 \textbf{1}\textbf{1}^T+\textbf{1}\textbf{1}^T\Sigma^2 + 2\lambda\textbf{1}\textbf{1}^T}-\frac{1}{\Sigma^2 \textbf{1}\textbf{1}^T+\textbf{1}\textbf{1}^T\Sigma^2 }\bigg)\odot(\AA\DD^T+\DD\AA^T)\notag\\
    &\quad+\frac{1}{\Sigma^2 \textbf{1}\textbf{1}^T+\textbf{1}\textbf{1}^T\Sigma^2 + 2\lambda\textbf{1}\textbf{1}^T}\odot 2\lambda\ZZref
    \bigg\|_F \notag\\
    &= \bigg\|
    -2\lambda\cdot\bigg(\frac{1}{(\bSigma^2 \textbf{1}\textbf{1}^T+\textbf{1}\textbf{1}^T\bSigma^2 + 2\lambda\textbf{1}\textbf{1}^T)\odot(\bSigma^2 \textbf{1}\textbf{1}^T+\textbf{1}\textbf{1}^T\bSigma^2)}\bigg)\odot(\AA\DD^T+\DD\AA^T)\notag\\
    &\quad+\frac{1}{\bSigma^2 \textbf{1}\textbf{1}^T+\textbf{1}\textbf{1}^T\bSigma^2 + 2\lambda\textbf{1}\textbf{1}^T}\odot 2\lambda\ZZref
    \bigg\|_F \notag\\
    &=2\lambda\bigg\|
    \bigg(\frac{1}{\bSigma^2 \textbf{1}\textbf{1}^T+\textbf{1}\textbf{1}^T\bSigma^2 + 2\lambda\textbf{1}\textbf{1}^T}\bigg)\odot\bigg(\frac{1}{\bSigma^2 \textbf{1}\textbf{1}^T+\textbf{1}\textbf{1}^T\bSigma^2}\odot(\AA\DD^T+\DD\AA^T)\bigg)\notag\\
    &\quad-\bigg(\frac{1}{\bSigma^2 \textbf{1}\textbf{1}^T+\textbf{1}\textbf{1}^T\bSigma^2 + 2\lambda\textbf{1}\textbf{1}^T}\bigg) \odot \ZZref
    \bigg\|_F \notag\\
    &=2\lambda\bigg\|
    \bigg(\frac{1}{\bSigma^2 \textbf{1}\textbf{1}^T+\textbf{1}\textbf{1}^T\bSigma^2 + 2\lambda\textbf{1}\textbf{1}^T}\bigg)\odot\big((\ZZ_*)_1- \ZZref\big) \bigg\|_F \notag\\
    &\leq 2\lambda\cdot
    \frac{1}{\min|\bSigma^2 \textbf{1}\textbf{1}^T+\textbf{1}\textbf{1}^T\bSigma^2 + 2\lambda\textbf{1}\textbf{1}^T|}\cdot\big\|(\ZZ_*)_1- \ZZref \big\|_F,
\end{align}
where the computations of matrices are element-wise, and the first three equalities follows from the identity $\AA\odot\XX+
\BB\odot\XX = (\AA+\BB)\odot\XX$ of the Hadamard product for any matrices $\AA$, $\BB$ and $\XX$ of the same dimensions.
The fourth equality in~\eqref{eq::ub_lambda1} holds by the definition of $(\ZZ_*)_1$, and the last inequality is due to the fact that 
\begin{align}
    \|\AA\odot\BB\|_F \leq \max|\AA|\cdot\|\BB\|_F
\end{align}
for any two matrices $\AA$ and $\BB$ of the same dimensions.

% \begin{align}
%     \|(\ZZ_\lambda)_1 - (\ZZ_*)_1 \|^2_F &=\sum_{i,j}\Big(\frac{\sigma_ic_{ij} + \sigma_j c_{ji}}{\sigma_i^2 +\sigma_j^2 + 2\lambda} - \frac{(\sigma_i^2 +\sigma_j^2)z_{ij}}{\sigma_i^2 +\sigma_j^2 + 2\lambda} + z_{ij} - z^*_{ij}\Big)^2\notag\\
%     &\leq 2\sum_{i,j} \bigg[\Big(\frac{\sigma_ic_{ij} + \sigma_j c_{ji}}{\sigma_i^2 +\sigma_j^2 + 2\lambda} - \frac{(\sigma_i^2 +\sigma_j^2)z_{ij}}{\sigma_i^2 +\sigma_j^2 + 2\lambda}
%     \Big)^2 + (z_{ij} - z^*_{ij})^2
%     \bigg] \notag\\
%     &=2\sum_{i,j} \bigg[
%     \Big(\frac{\sigma_i^2 + \sigma_j^2}{\sigma_i^2 +\sigma_j^2 + 2\lambda}
%     \Big)^2(z_{ij} - z_{ij}^*)^2
%     + (z_{ij} - z^*_{ij})^2
%     \bigg]\notag\\
%     &= \bigg[
%     1+2\Big(\frac{\sigma_i^2 + \sigma_j^2}{\sigma_i^2 +\sigma_j^2 + 2\lambda}
%     \Big)^2
%     \bigg]\cdot\|(\ZZ_*)_1-(\ZZref)_1\|_F^2.
% \end{align}
For the term (ii), note that 
$(\ZZ_*)_D = \VV_1 \bSigma^{-1}\UU^T\DD^T\VV_2\VV_2^T=\VV_1\big( \bSigma^\top\bSigma\big)^{-1} \VV_1^T\AA \DD^T \VV_2\VV_2^T$.
Since $\VV_1^T\VV_1 = \VV_2^T\VV_2 = \II$, we have 
\begin{align}\label{eq::zd_def2}
    \VV_1^T(\ZZ_*)_D\VV_2 = \big( \bSigma^\top\bSigma\big)^{-1} \VV_1^T\AA \DD^T \VV_2.
\end{align}
Furthermore, by using the unitary invariance of the orthogonal matrix w.r.t.\;the Frobenius norm, we obtain
\begin{align}
    \|(\ZZ_\lambda)_D - (\ZZ_*)_D\|_F
    & =
    \big\|
     \VV_1\Big(\big( \bSigma^\top\bSigma + 2\lambda \II \big)^{-1} \VV_1^T\big( \AA \DD^T + 2\lambda \ZZref  \big)
     - \big( \bSigma^\top\bSigma\big)^{-1} \VV_1^T\AA \DD^T
     \Big)
     \VV_2\VV_2^T
    \big\|_F \notag\\
    &\stackrel{(\text{a})}{=}\big\|\big( \bSigma^\top\bSigma + 2\lambda \II \big)^{-1} \VV_1^T\big( \AA \DD^T + 2\lambda \ZZref  \big)\VV_2
     - \big( \bSigma^\top\bSigma\big)^{-1} \VV_1^T\AA \DD^T \VV_2
    \big\|_F \notag\\
    &=\big\|\Big(\big( \bSigma^\top\bSigma + 2\lambda \II \big)^{-1}- \big( \bSigma^\top\bSigma\big)^{-1}
    \Big)
    \VV_1^T \AA \DD^T \VV_2
     + 2\lambda\big( \bSigma^\top\bSigma + 2\lambda \II \big)^{-1} \VV_1^T \ZZref\VV_2 \big\|_F\notag\\
     &\stackrel{\text{(b)}}{=}\big\|-2\lambda\big( \bSigma^\top\bSigma + 2\lambda \II \big)^{-1} \big( \bSigma^\top\bSigma\big)^{-1}
    \VV_1^T \AA \DD^T \VV_2
     + 2\lambda\big( \bSigma^\top\bSigma + 2\lambda \II \big)^{-1} \VV_1^T \ZZref\VV_2 \big\|_F\notag\\
     &\stackrel{\text{(c)}}{=}2\lambda\Big\|\big( \bSigma^\top\bSigma + 2\lambda \II \big)^{-1}
     \Big(
     \big( \bSigma^\top\bSigma\big)^{-1}
    \VV_1^T \AA \DD^T \VV_2
     -  \VV_1^T \ZZref \VV_2
     \Big)
     \Big\|_F\notag\\
     &\stackrel{\text{(d)}}{=}
     2\lambda\Big\|\big( \bSigma^\top\bSigma + 2\lambda \II \big)^{-1}
    \VV_1^T\big((\ZZ_*)_D - \ZZref\big)\VV_2
     \Big\|_F\notag\\
    &\leq 2\lambda \max\big|\big( \bSigma^\top\bSigma + 2\lambda \II \big)^{-1}\big|
    \cdot
    \big\|\VV_1^T\big((\ZZ_*)_D - \ZZref\big)\VV_2
     \big\|_F \notag\\
     &\leq 2\lambda\max\big|\big( \bSigma^\top\bSigma + 2\lambda \II \big)^{-1}\big|
     \cdot
      \big\|(\ZZ_*)_D - \ZZref
     \big\|_F,
     \label{eq::ub_lambdad}
\end{align}
where (a), (c) and the last equality hold by the unitary invariance of $\VV_1$ and $\VV_2$  w.r.t.\;the Frobenius norm, (b) holds since $\bSigma^\top\bSigma + 2\lambda \II$ and $\bSigma^\top\bSigma$ are diagonal matrices, (d) follows from~\eqref{eq::zd_def2}, and the first inequality holds since $(\bSigma^\top\bSigma + 2\lambda \II)^{-1}$ is a diagonal matrix. The last inequality holds since $\VV_1\VV_1^T$ and $\VV_2\VV_2^T$ are projections.

Therefore, by combining~\eqref{eq::lambda_pert},~\eqref{eq::ub_lambda1} and~\eqref{eq::ub_lambdad} we have
\begin{align}
     \|\ZZ_\lambda - \ZZ_*\|_F &\leq \|(\ZZ_\lambda)_1 - (\ZZ_*)_1 \|_F + 2\|(\ZZ_\lambda)_D - (\ZZ_*)_D\|_F \notag\\
     &\leq 2\lambda    \frac{1}{\min|\bSigma^2 \textbf{1}\textbf{1}^T+\textbf{1}\textbf{1}^T\bSigma^2 + 2\lambda\textbf{1}\textbf{1}^T|}\cdot \big\|(\ZZ_*)_1- \ZZref \big\|_F
    + 4\lambda\max\big|\big( \bSigma^\top\bSigma + 2\lambda \II \big)^{-1}\big|
     \cdot
      \big\|(\ZZ_*)_D - \ZZref
     \big\|_F \notag\\
     &\leq \frac{\lambda }{\sigma_{\min}^2(\AA)+\lambda}
     \cdot \| \ZZ_{\star}-\ZZref \|_F + \frac{4\lambda }{\sigma_{\min}^2(\AA)+\lambda}
     \cdot \| \ZZ_{\star}-\ZZref \|_F \notag\\
     &= \frac{5\lambda }{\sigma_{\min}^2(\AA)+\lambda}
     \cdot \| \ZZ_{\star}-\ZZref \|_F,
\end{align}
where the last inequality follows from the definition of the element-wise operator and the facts that $\|(\ZZ_*)_1- \ZZref\|_F\leq\|\ZZ_{\star}-\ZZref\|_F$ and $\|(\ZZ_*)_D- \ZZref\|_F\leq\|\ZZ_{\star}-\ZZref\|_F$.
Hence, we conclude the proof.
\end{proof}
% \begin{align}
%     \|(\ZZ_\lambda)_D - (\ZZ_*)_D\|_F
%     & \leq
%     2\big( \|(\ZZ_\lambda)_D - (\ZZref)_D\|_F^2 
%     + \|(\ZZ_*)_D - (\ZZref)_D\|_F^2
%     \big).
% \end{align}

% We further bound the first term of the right hand side of the above inequality as below.
% \begin{align}
%     &\|(\ZZ_\lambda)_D - (\ZZref)_D\|_F^2 \notag\\
%     &\quad=\|\big( \bSigma^\top\bSigma + 2\lambda \II \big)^{-1}
%     \bSigma^\top\bSigma \VV_1^T(\ZZ_*)_D
%     +2\lambda\VV_1\big( \bSigma^\top\bSigma + 2\lambda \II \big)^{-1} \VV_1^T (\ZZref)_D (\II-\PP) - \ZZref\|_F^2 \\
%     &\quad\leq
% \end{align}

\subsection{Perturbation of \textbf{A} and \textbf{D}}
We first present a stability analysis result of the regularized least squares (RLS), which is used in the analysis for the perturbation of $\AA$ and $\DD$ in Lemma~\ref{lem:stab_ad}.
\begin{lemma}[Stability analysis of regularized least squares]\label{lem:stab}
Let $\xx^*$ solve the problem
\begin{align}
    \min_{\xx}\|\AA\xx - \bb\|_2^2 + \beta\|\xx- \xx_0\|^2_2,
\end{align}
where $\xx,\xx_0\in\RR^p$, $\AA\in\RR^{q\times p}$, $\bb\in\RR^q$ for some integer $p,q>0$ and $\beta>0$.
Let $\hat\xx$ solve 
\begin{align}
     \min_{\xx}\|(\AA+\delta\AA)\xx - (\bb+\delta\bb)\|_2^2 + \beta\|\xx- \xx_0\|^2_2,
\end{align}
where $\delta\AA\in\RR^{q\times p}, \delta\bb\in\RR^q$, and $\|\delta\AA\|_2 \ll \|\AA\|_2$. 
Suppose that $\rank(\AA) = \rank(\AA+\delta\AA)$, we have 
\begin{align}
    \|\xx^* - \hat\xx\|_2 \leq
    \mathcal{O}\Big(\frac{ \|\delta\AA\|_2 + \|\delta\bb\|_2 }{\beta}\Big).
\end{align}
\end{lemma}
\begin{proof}
By definition, we have explicitly that 
\begin{align}
    \xx^* = (\AA^T\AA + \beta\II)^{-1}(\AA^T\bb+\beta\xx_0).
\end{align}
Let $\tilde\AA = \AA + \delta\AA$ and $\PP = -\AA^T\AA + \tilde\AA^T\tilde\AA$, we can write
\begin{align}
    \hat\xx = (\AA^T\AA + \PP + \beta\II)^{-1}((\AA+\delta\AA)^T(\bb+\delta\bb)+\beta\xx_0).
\end{align}
Hence, we obtain
\begin{align}
    \|\hat\xx - \xx^*\|_2 &\leq 
    \Big\|\big( (\AA^T\AA+\beta\II)^{-1} - (\AA^T\AA + \PP + \beta\II)^{-1} \big)(\AA^T\bb+\beta\xx_0)\Big\|_2 \notag\\
    &\quad+\big\|(\AA^T\AA + \PP + \beta\II)^{-1}\big\|_2 \|\delta\AA\|_2\|\bb + \delta\bb\|_2 + \big\|(\AA^T\AA + \PP + \beta\II)^{-1}\big\|_2\|\AA\|_2\|\delta\bb\|_2 \notag\\
    & = \Big\|(\AA^T\AA + \PP + \beta\II)^{-1}\PP (\AA^T\AA+\beta\II)^{-1}  (\AA^T\bb+\beta\xx_0)\Big\|_2
    \notag\\
    &\quad+\big\|(\AA^T\AA + \PP + \beta\II)^{-1}\big\|_2 \|\delta\AA\|_2\|\bb + \delta\bb\|_2 + \big\|(\AA^T\AA + \PP + \beta\II)^{-1}\big\|_2\|\AA\|_2\|\delta\bb\|_2 \notag\\
    &\leq \frac{1}{\beta}\cdot\Big(\|\PP\|_2\|\xx^*\|_2 + \|\delta\AA\|_2\|\bb\|_2 + \|\delta\bb\|_2\|\AA\|_2 + \|\delta\bb\|_2\|
    \delta\AA\|_2 \Big).
\end{align}
Since $\|\delta\AA\|_2 \ll \|\AA\|_2$, we obtain
\begin{align}
    \|\hat\xx - \xx^*\|_2 
    \leq \mathcal{O}\Big(\frac{ \|\delta\AA\|_2 + \|\delta\bb\|_2 }{\beta}\Big),
\end{align}
which concludes the proof of the lemma.
\end{proof}

Now we show the stability analysis with respect to the perturbation of \textbf{A} and \textbf{D} below.
\begin{lemma}\label{lem:stab_ad}
Let $\hat{\mathbf{Z}}$ solve 
\begin{align}\label{eq:perturbed_procrustes_A}
    \min_{\ZZ=\ZZ^T \in \mathbb{R}^{d\times d}} \| \ZZ(\AA+\delta\AA)-(\DD+\delta\DD) \|_F^2 + \frac{\lambda}{2} \| \ZZ - \ZZref \|_F^2,
\end{align}
where $\delta\AA, \delta\DD\in\RR^{d\times m}$, $\|\delta\AA\|_2 \ll \|\AA\|_2$, and $\|\delta\DD\|_2 \ll \|\DD\|_2$.
Also, suppose $\ZZ_\lambda$ to be the solution to~\eqref{eq:reg_sym_procrustres_app} given $\lambda > 0$. Then, it holds that
\begin{align}
    \|\hat\ZZ - \ZZ_\lambda\|_F \leq \mathcal{O}\Big(\frac{ \|\delta\AA\|_2 + \|\delta\dd\|_2 }{\lambda}\Big).
\end{align}
\end{lemma}
\begin{proof}
We first reduce~\eqref{eq:reg_sym_procrustres_app} to an unconstrained regularized least squares (RLS) problem as follows.
Let $r = md$ and $s = d^2$. We denote by $\vec$ the operator that stacks the columns of a  matrix into a long vector. Then, for any 
$\ZZ=\ZZ^T \in \mathbb{R}^{d\times d}$ it follows that
\begin{align}
    \|\ZZ\AA - \DD\|_F + \frac{\lambda}{2}\|\ZZ - \ZZref\|_F &= \|\vec(\ZZ\AA - \DD)\|_2 + \frac{\lambda}{2}\|\vec(\ZZ - \ZZref)\|_2  \notag\\
    &= \|(\II_d\otimes\AA)\zz - \dd\|_2 + 
    \frac{\lambda}{2}\|\vec(\zz - \zzref)\|_2 
\end{align}
Here $(\II_d \otimes \AA)_{ij} = \delta_{ij} \AA\in\RR^{r\times s}$, $\zz = \vec(\ZZ)\in\RR^s, \zzref = \vec(\ZZref)\in\RR^s$, and $\dd = \vec(\DD) \in\RR^r$. We define $\SS_d$ as the matrix where the columns form an orthonormal basis for a $\bar d$-dimensional subspace of $\RR^s$, where $\bar{d}=\dfrac{d(d+1)}{2}$. 
By using the symmetry of $\XX$, letting $\zz = \SS_d\yy$, $\zzref = \SS_d\yyref$ and $\HH = (\II_d\otimes\AA)\SS_d$,
we are able to obtain an regularized LS problem equivalent to~\eqref{eq:reg_sym_procrustres_app} as follows,
\begin{align}\label{eq:rls}
    \min_{\yy\in\RR^{\bar{d}}}
    \|\HH\yy - \dd\|_2 + 
    \frac{\lambda}{2}\|\yy - \yyref\|_2.
\end{align}
Here we have used the fact that for an orthonormal matrix $\SS_d$ we have $\|\SS_d(\yy - \yyref)\|_2 = \|\yy - \yyref\|_2$.
Likewise, we can identify the perturbed problem~\eqref{eq:perturbed_procrustes_A} with perturbations 
\begin{align}
    \HH \to \HH + \delta\HH, \quad \dd \to \dd + \delta\dd, \quad \yy \to \tilde\yy
\end{align}
in~\eqref{eq:rls}, where 
\begin{align}
    \delta \HH = (\II_d\otimes \delta \AA)\SS_d, 
    \quad 
    \delta \dd = \vec(\delta \DD), 
    \quad
    \SS_d\hat\yy = \vec(\tilde\ZZ).
\end{align} 
Furthermore, the solution to~\eqref{eq:rls} can be written as 
\begin{align}
    \yy^* = \big(\HH^T\HH + \lambda\II/2 \big)^{-1}\HH^T\dd.
\end{align}
Then, the solution is perturbed to \[
\hat\yy = \yy + \delta\yy =  \big((\HH+\delta\HH)^T(\HH+\delta\HH) + \lambda\II/2 \big)^{-1}(\HH+\delta\HH)^T(\dd+\delta\dd).
\]
After the reduction of~\eqref{eq:reg_sym_procrustres_app} to~\eqref{eq:rls},
we apply Lemma~\ref{lem:stab} to~\eqref{eq:rls}, which yields
\begin{align}\label{eq:direct_order}
    \|\hat\yy - \yy^*\|_2 \leq  \mathcal{O}\Big(\frac{ \|\delta\HH\|_2 + \|\delta\dd\|_2 }{\lambda}\Big).
\end{align}
Also, by the definition of $\HH$ we have $\|\delta\HH\|_2 = \|\delta\AA\|_2$ where $\AA$ is defined in the original problem~\eqref{eq:reg_sym_procrustres_app}. Hence, \eqref{eq:direct_order} reads
\begin{align}
    \|\hat\yy - \yy^*\|_2 \leq  \mathcal{O}\Big(\frac{ \|\delta\AA\|_2 + \|\delta\dd\|_2 }{\lambda}\Big).
\end{align}
Further, we can write 
\begin{align}
    \big\|\hat\ZZ - \ZZ_\lambda\big\|_F &= 
    \|\vec\big(\hat\ZZ\big) - \vec\big(\ZZ_\lambda\big)\|_2  \notag\\
    &= \|\SS_d\tilde\yy - \SS_d\yy^*\|_2 \notag\\
    &= 
    \|\hat\yy - \yy^*\|_2 \notag\\ 
    &\leq  \mathcal{O}\Big(\frac{ \|\delta\AA\|_2 + \|\delta\dd\|_2 }{\lambda}\Big),
\end{align}
where the third equality holds since the columns of $\SS_d$ form an orthonormal basis. This concludes the proof.
\end{proof}

Furthermore, we remark that Lemma~\ref{lem:perturb_lambda} and Lemma~\ref{lem:stab_ad} together concludes the proof of Proposition~\ref{prop:stab}.

%% file: sections/numexp.tex
\clearpage

\section{Numerical Experiments} \label{sec:numerics}

\subsection{Datasets}

We used several UCI datasets, whose main characteristics are summarized in Table \ref{tab:datasets}. In the case of the \textit{P53 mutant} dataset, we reduce its size to avoid memory problems. We kept all labels where $y=1$ (153 instances), and merge them with the 5000 first data points.

\begin{table}[ht]
    \centering
    \begin{tabular}{llccc}
        \toprule
        Dataset name & Tag & $\#$ features & $\#$ data points & Section \\
        \midrule
        Madelon \citep{guyon2008feature} & Madelon & 500 & 4400 & \ref{sec:numexp_madelon}\\
        % Random (medium size) & Random & 500 & 2000\\
        Internet Advertisements \citep{kushmerick1999learning} & Ad & 1558 & 3279 & \ref{sec:numexp_ad}\\
        QSAR oral toxicity \citep{ballabio2019integrated} & Qsar & 1024 & 8992 & \ref{sec:numexp_qsar}\\
        p53 Mutants Data Set \citep{danziger2006functional} & P53 mutant & 5406 & 5000 & \ref{sec:numexp_mutant}\\
        \bottomrule
    \end{tabular}
    \caption{Summary of the datasets used in the numerical experiments.}
    \label{tab:datasets}
\end{table}

\subsection{Setting}

We consider the regression problem
\begin{equation}
    \textstyle\min_{\xx\in\R^{d}} f(\xx) \defas \frac{1}{N}\sum_{i=0}^N \ell(\aa_i^T\xx,\bb_i) + \frac{\tau}{2} \|x\|_2^2, \label{eq:regression_equation}
\end{equation}
where $\ell(\cdot,\cdot)$ is either a quadratic or a logistic loss. The pair $(\AA,\bb)$ is a dataset, where $a_i \in \mathbb{R}^d$ is a data point composed by $d$ features, and $b_i$ is the label of the $i^{th}$ data point. We solve the problem using deterministic and stochastic gradient, whose parameters are described in Table \ref{tab:value_param_setting}. The optimal value of \eqref{eq:regression_equation} are obtained using the \textsc{Matlab} package \texttt{minfunc} from \citep{schmidt2005minfunc}.

\begin{table}[ht]
    \centering
    \begin{tabular}{lcc}
        \toprule
        Parameter & Deterministic setting & Stochastic setting \\
        \midrule
        $\tau$ & \texttt{1e-9} (ill-conditioned problem) & \texttt{1e-2} \\
        Descent direction & Full gradient & SAGA (see \citep{defazio2014saga}) \\
        Batch size & Full batch & 64 \\ 
        Limited memory $m$ & $10$ and $\infty$ & $25$ \\
        Line-search & None or approximate dichotomy & None \\
        $\BBref^{-1}$ and $\HHref$ (No LS) & $\frac{1}{\|A\|_2^2}$ (quad.), $\frac{1}{4\|A\|_2^2}$ (logistic) & $\frac{1}{3\max_i L_i}$ \citep{defazio2014saga} \\
        $\BBref^{-1}$ and $\HHref$ (with LS) & 1 & N/A. \\
        Rel. reg. $\bar\lambda$ (if applicable) & \texttt{1e-20} (quad), \texttt{1e-10} (logistic)  & \texttt{1e-2}\\
        Max. iteration & \texttt{250} (full batch) & \texttt{1e4} (mini-batches) \\
        \bottomrule
    \end{tabular}
    \caption{Parameters used to optimize \eqref{eq:regression_equation}}
    \label{tab:value_param_setting}
\end{table}

\subsection{Observation}

\paragraph{Unitary step VS line search.} Most of the presented method present a divergent behavior when we do not apply line search. However, it seems that the Multisecant Type-I method is the most robust one, converging for almost all instances. In fact, it seems that adding a line-search to method slow it down - probably because the optimal stepsize is close to one, but it takes time to have the guarantee. When it comes to line-search methods, there is no clear method whose speed is superior. Surprisingly, in both cases, the Type-II symmetric multisecant method seems to be the worst one (after gradient descent).

\paragraph{Stochastic optimization} As it may be expected, the symmetric multisecant type-I is the fastest method. Indeed, our updates have provably better robustness, and the type-I symmetric multisecant update is the best one amongst all method with unitary step-size. However, its performance are not much different than gradient descent. Moreover, the author indicate that the mini-batch size plays an important role in the convergence of the method, as smaller batches have too much variance. We suspect there is a trade-off to improve the speed of the method, where we should balance the size of the batch and the number of secant equations.

\clearpage

\subsection{Spectrum Recovery on Madelon (Quadratic Loss)}

\begin{figure}[ht]
    \centering
    \includegraphics[width=0.49\textwidth]{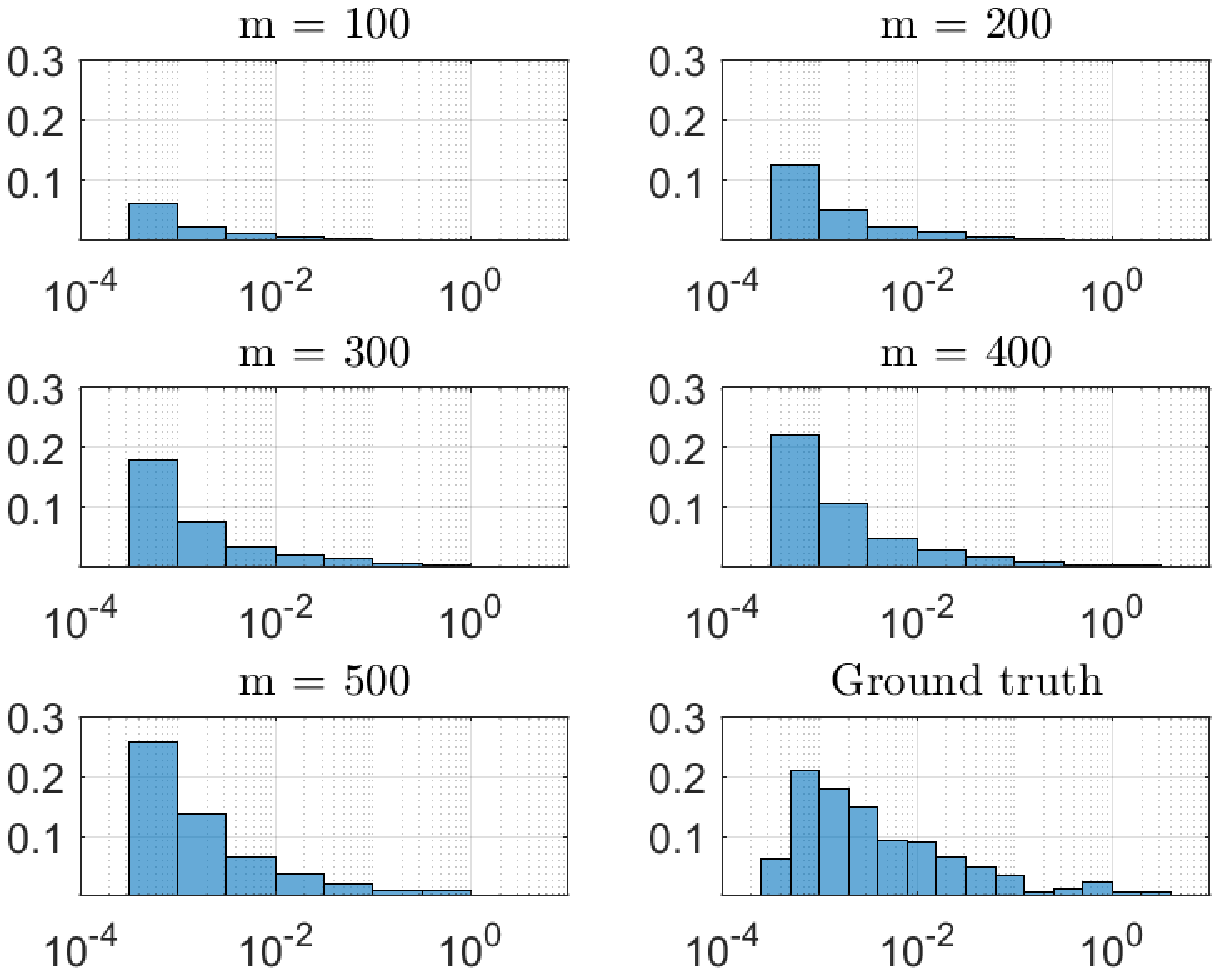}
    \includegraphics[width=0.49\textwidth]{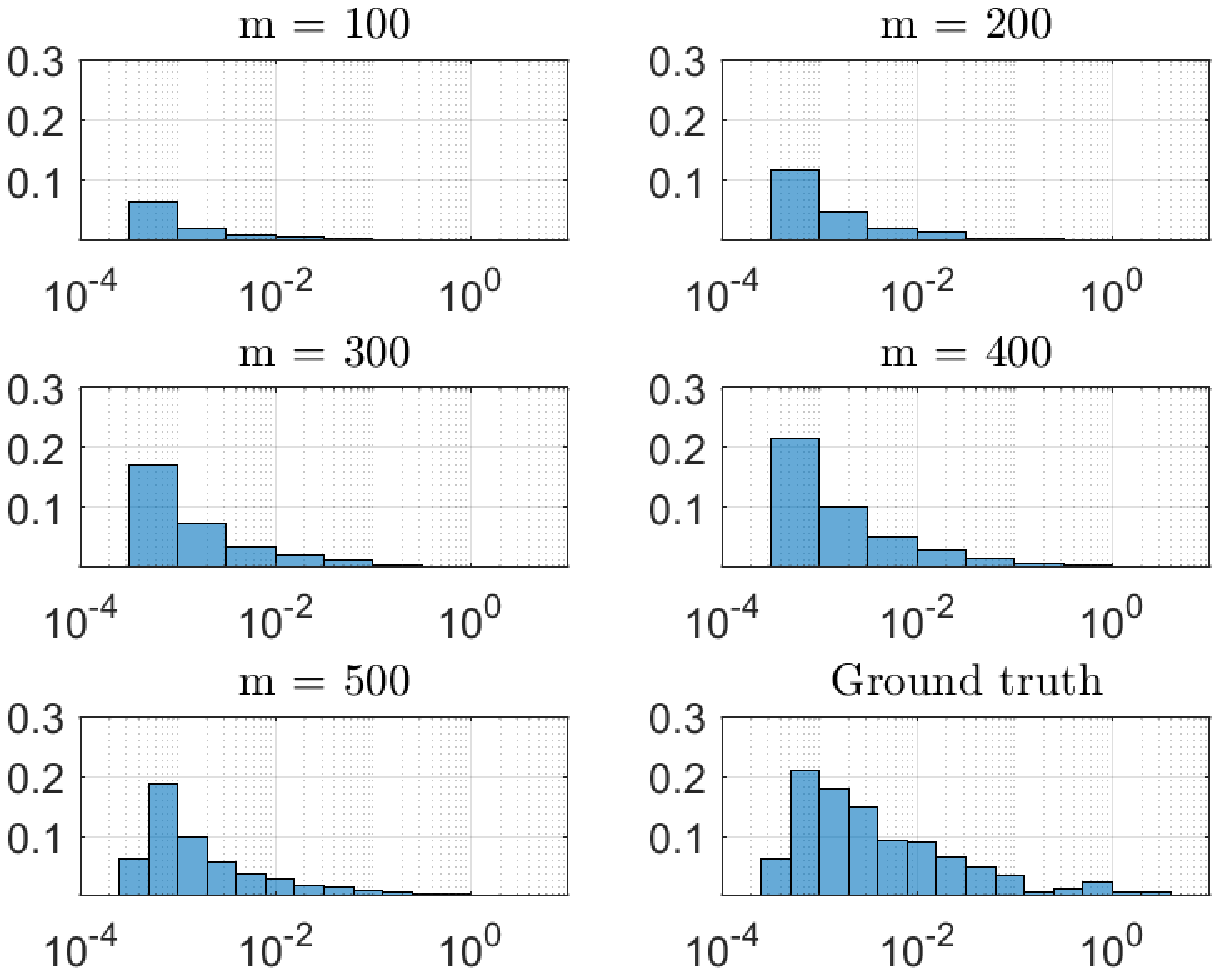}
    \includegraphics[width=0.49\textwidth]{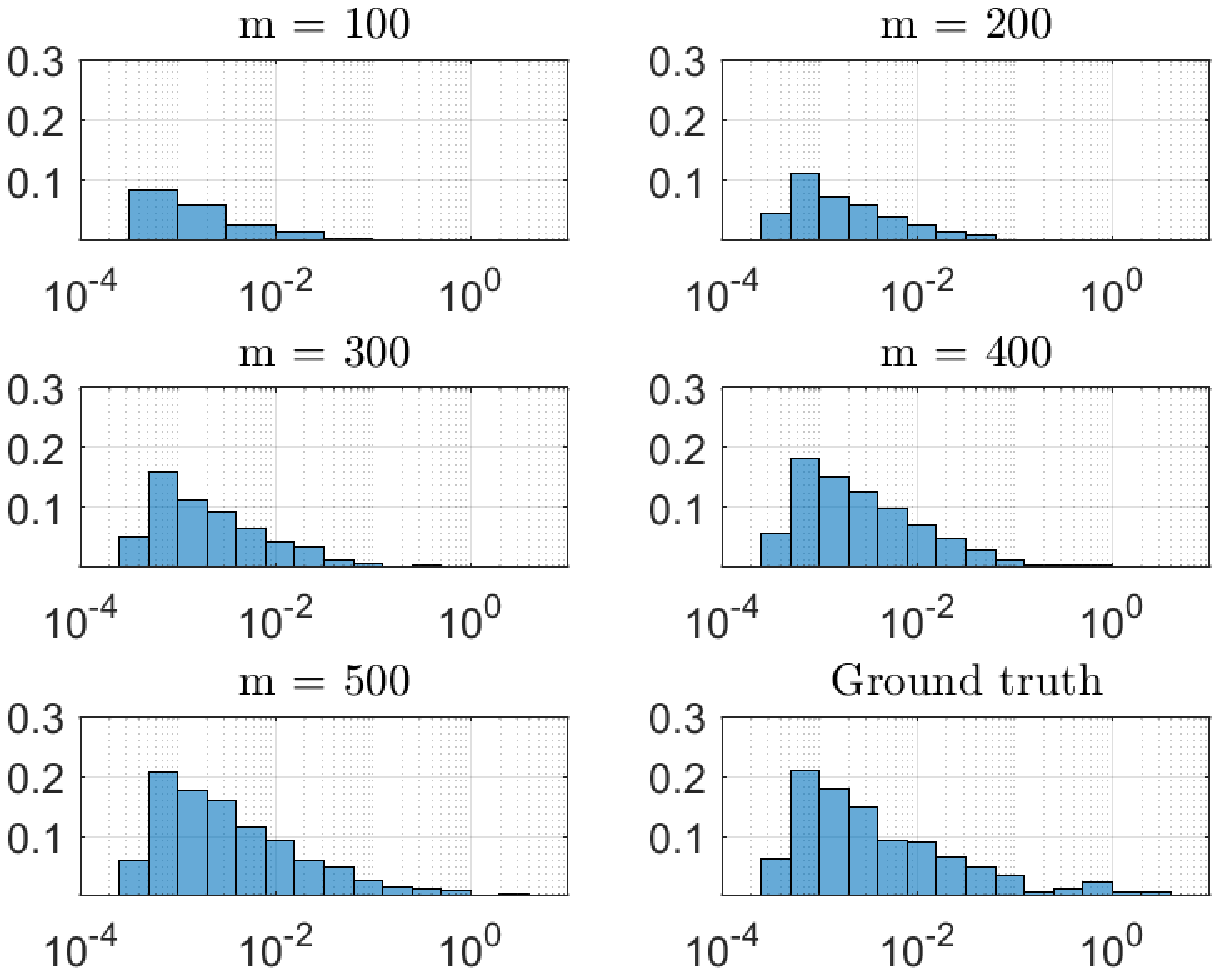}
    \includegraphics[width=0.49\textwidth]{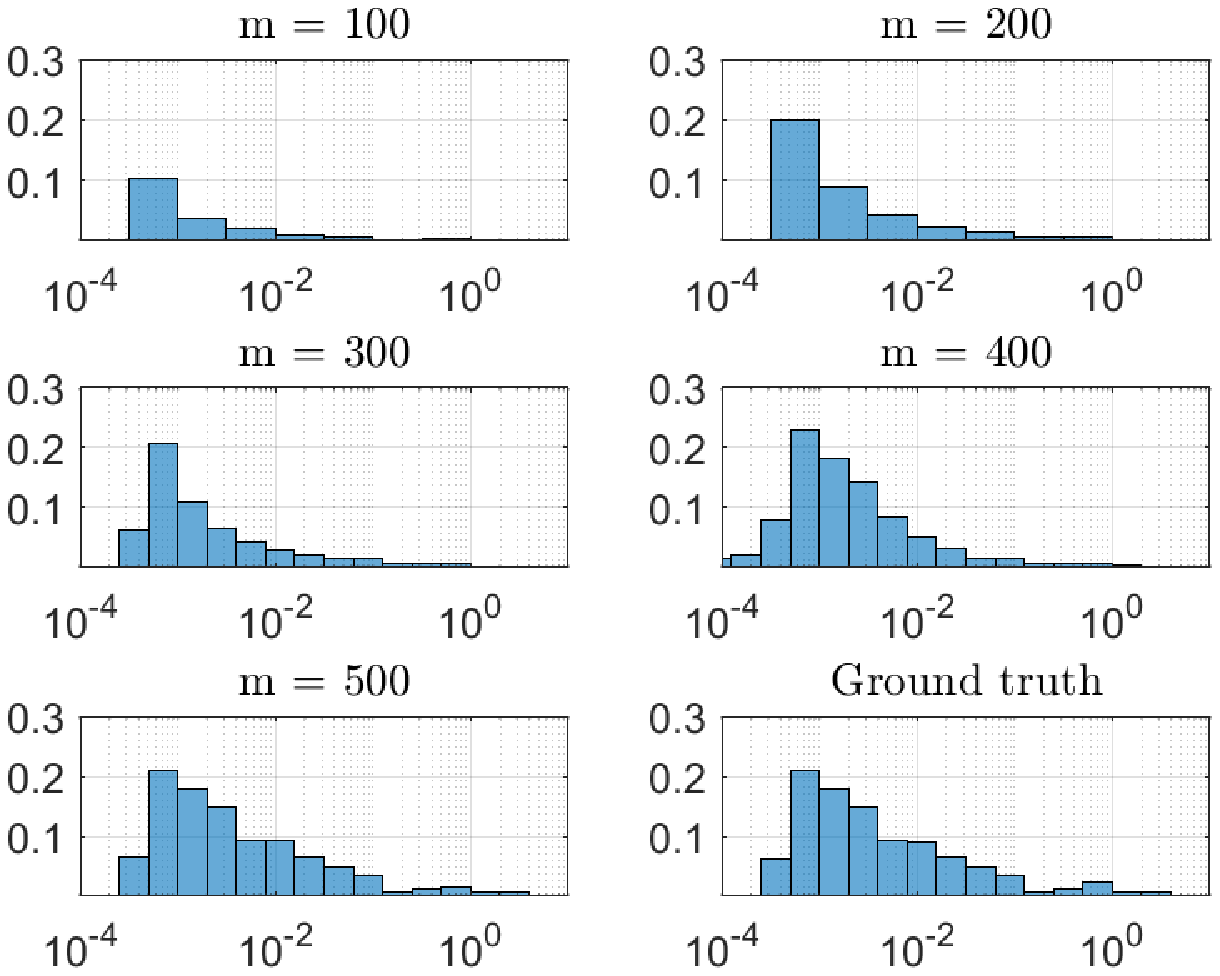}
    \caption{Histogram of the eigenvalues of the estimate $\HH_k$ or $\BB_k^{-1}$ in the function on the iteration counter (i.e., the number of secant equations), when optimizing the square loss on the Madelon dataset without regularization. \textbf{Top left:} Multisecant Broyden Type-I, \textbf{Top right:} Multisecant Broyden Type-II, \textbf{Bottom left:} Type-I symmetric multisecant, \textbf{Bottom right:} Type-II symmetric multisecant. For the non-symmetric updates, we took the real part of the eigenvalues. We removed from the histogram the spike of eigenvalues associated to $\HHref$ or $\BBref^{-1}$ (initialized at $1/L$, the smoothness constant of the function). It seems that the spectrum converges faster to the ground truth when we use symmetric updates. We did not report BFGS as the method is non-convergent with unitary stepsize.}
    \label{fig:spectrum_iterates}
\end{figure}

\clearpage

\subsection{Organization of figures}
\begin{figure}[ht]
    \centering
    \includegraphics[scale=0.55]{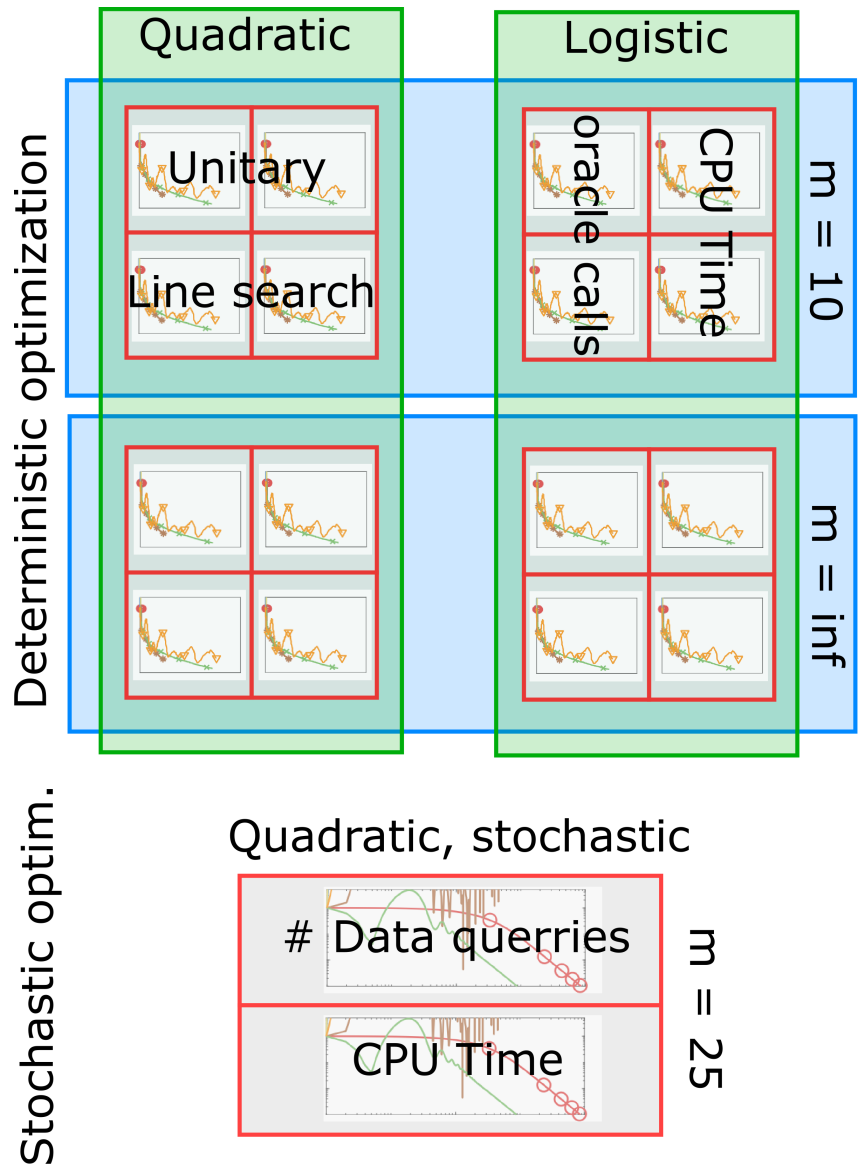}
    \caption{Organization of figures for the numerical experiments.}
    \label{fig:figures_orga}
\end{figure}

\subsection{Legend}

\begin{figure}[h!t]
    \centering
    \includegraphics[width=1\linewidth]{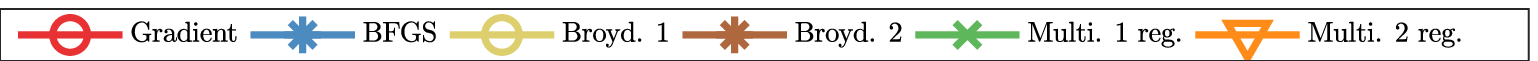}
    \caption{Legend for all subsequent figures}
    \label{fig:legend}
\end{figure}

\clearpage

\subsection{Madelon} \label{sec:numexp_madelon}
\begin{figure}[h!]
    \centering
    \includegraphics[width=0.48\textwidth]{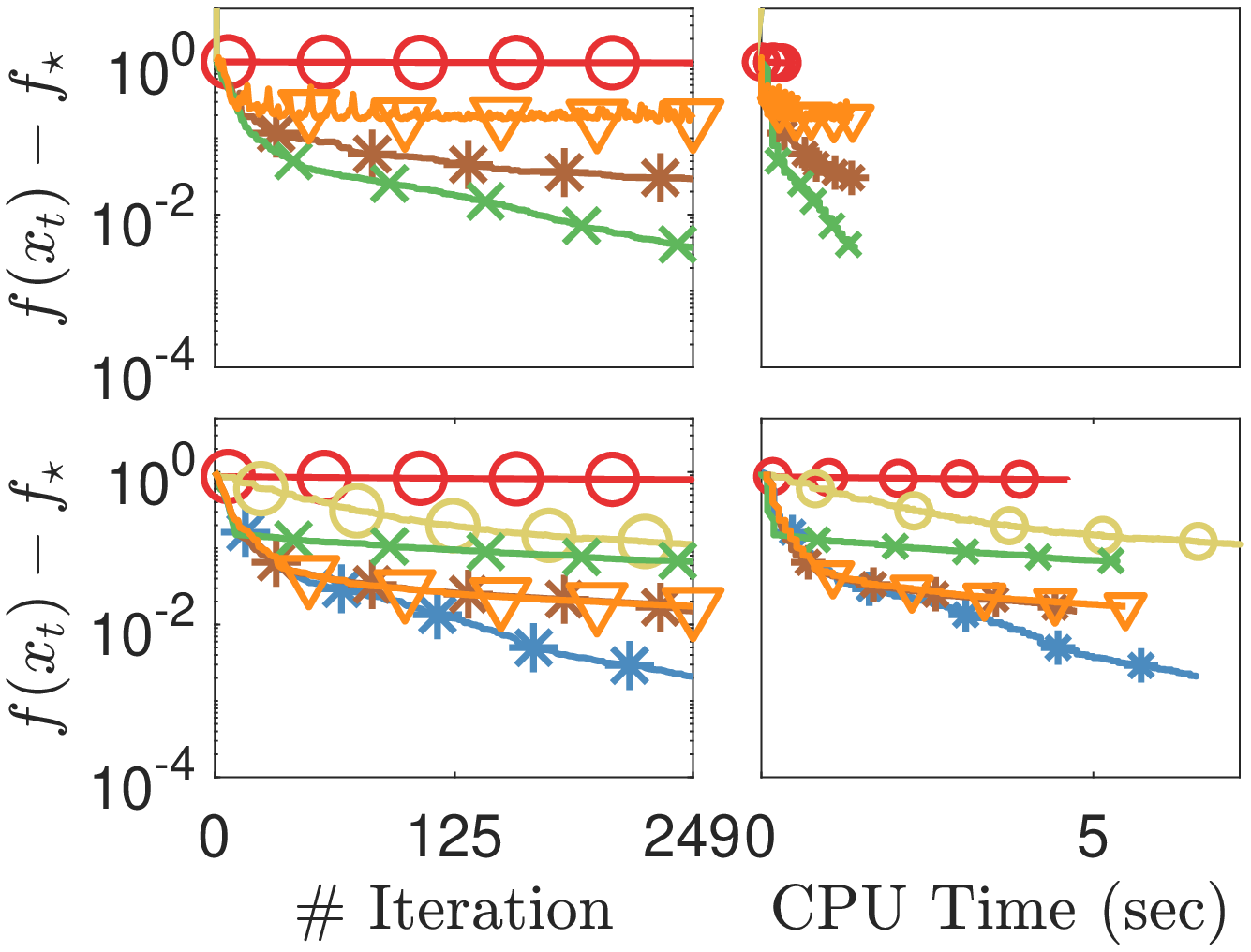}
    \includegraphics[width=0.48\textwidth]{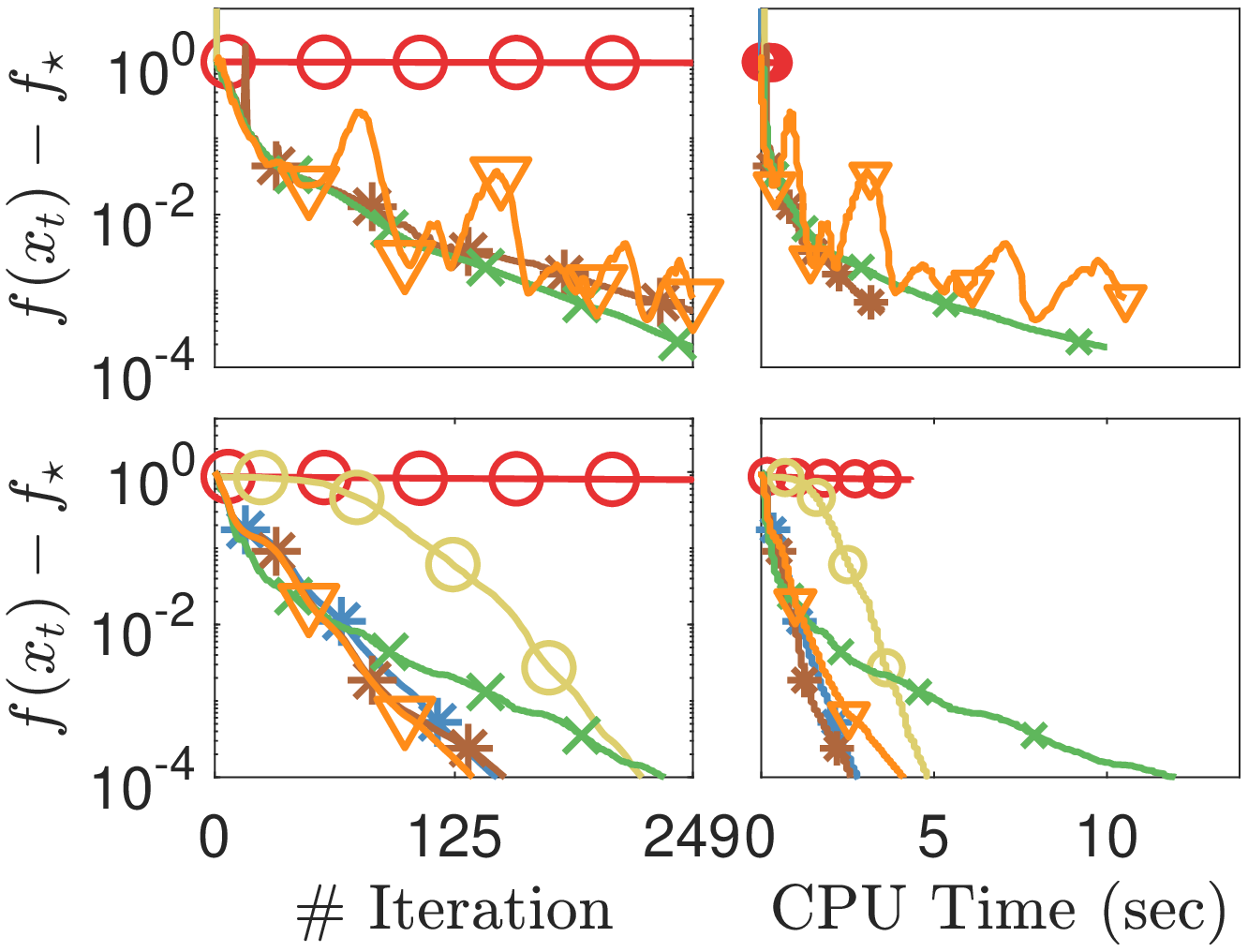}\\
    \includegraphics[width=0.48\textwidth]{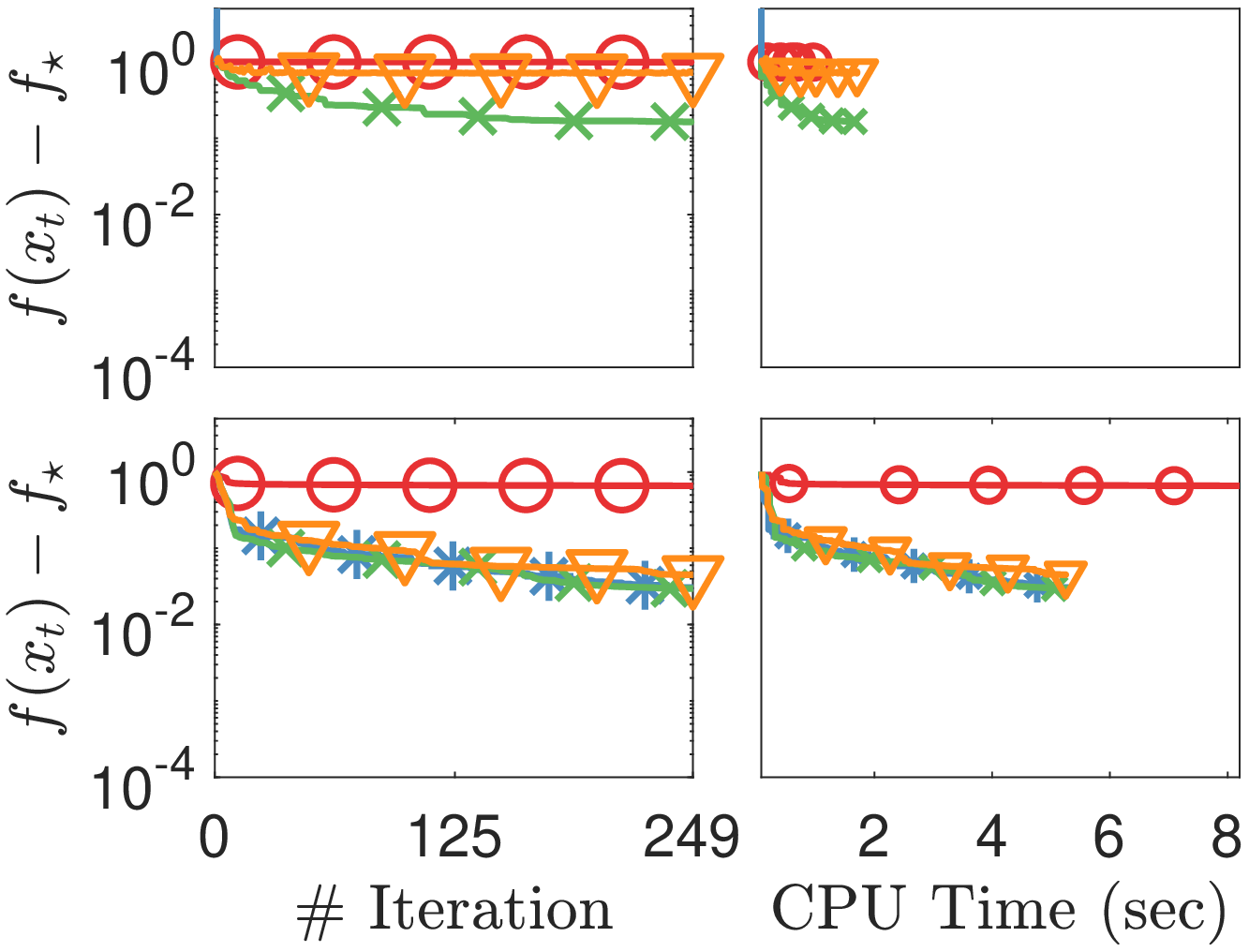}
    \includegraphics[width=0.48\textwidth]{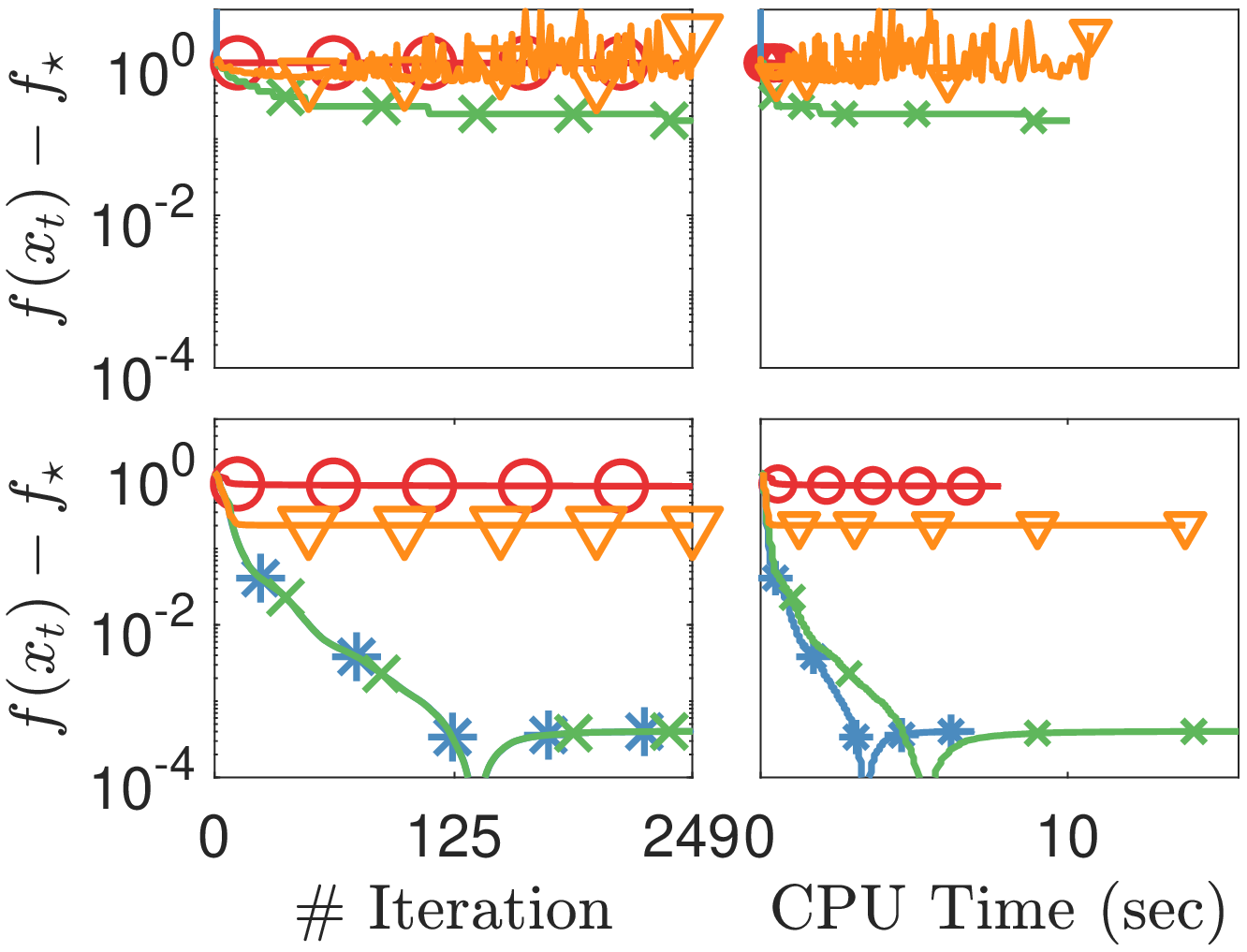}\\
    \includegraphics[width=0.5\textwidth]{figures/generated/experiment_stoch_madelon_square_m_25.eps}
\end{figure}

\clearpage
\subsection{Ad} \label{sec:numexp_ad}
\begin{figure}[h!]
    \centering
    \includegraphics[width=0.48\textwidth]{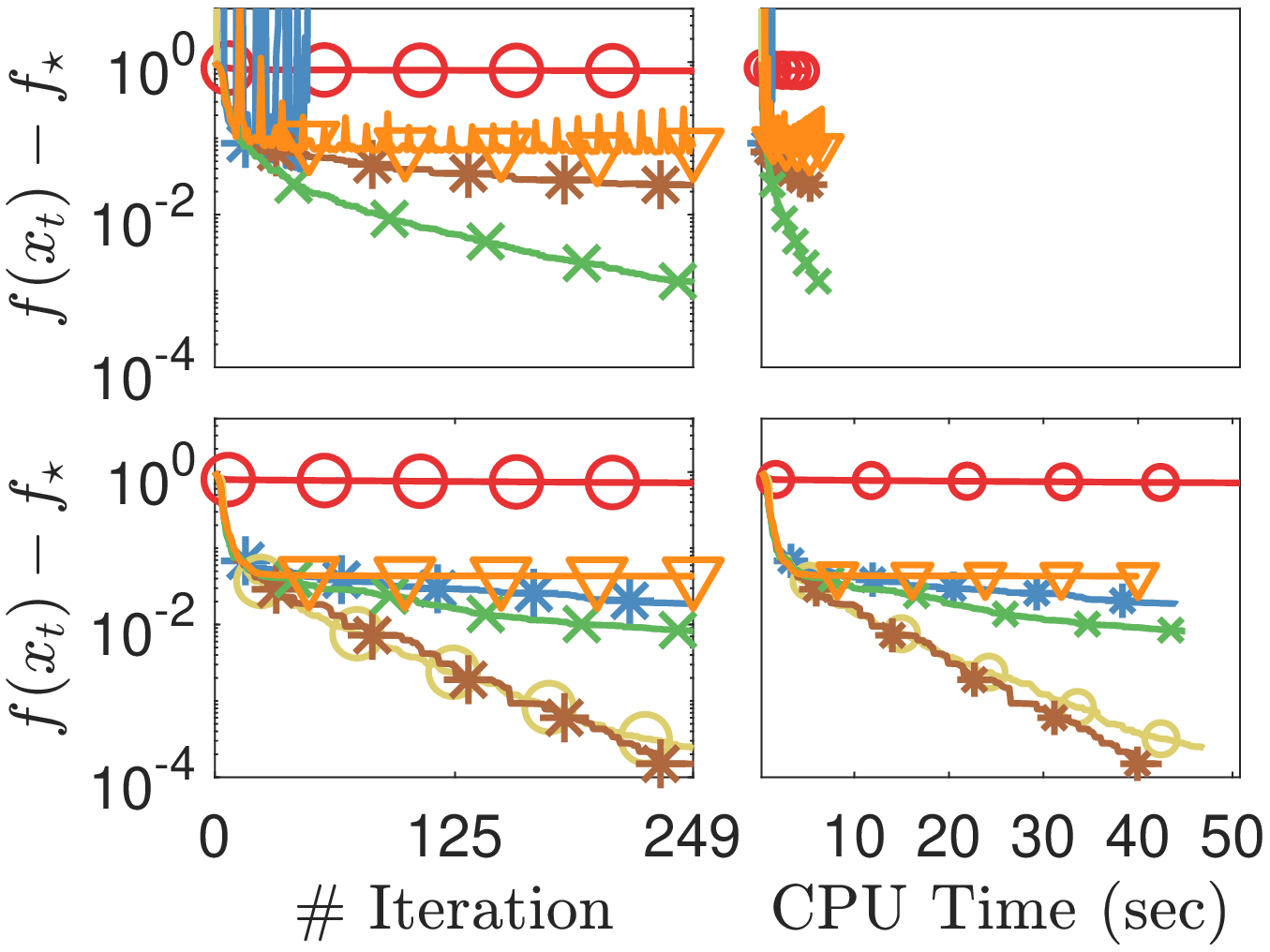}
    \includegraphics[width=0.48\textwidth]{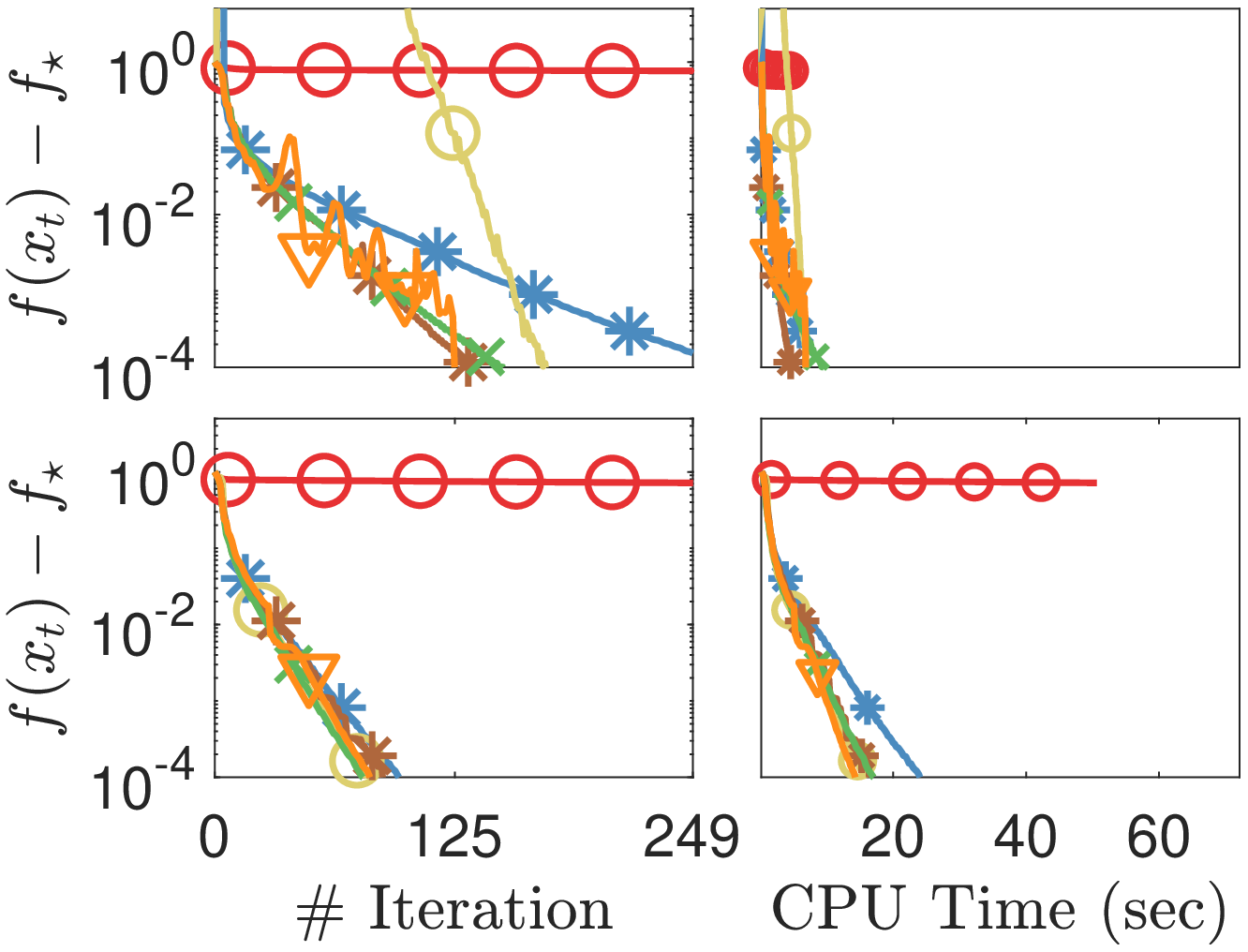}\\
    \includegraphics[width=0.48\textwidth]{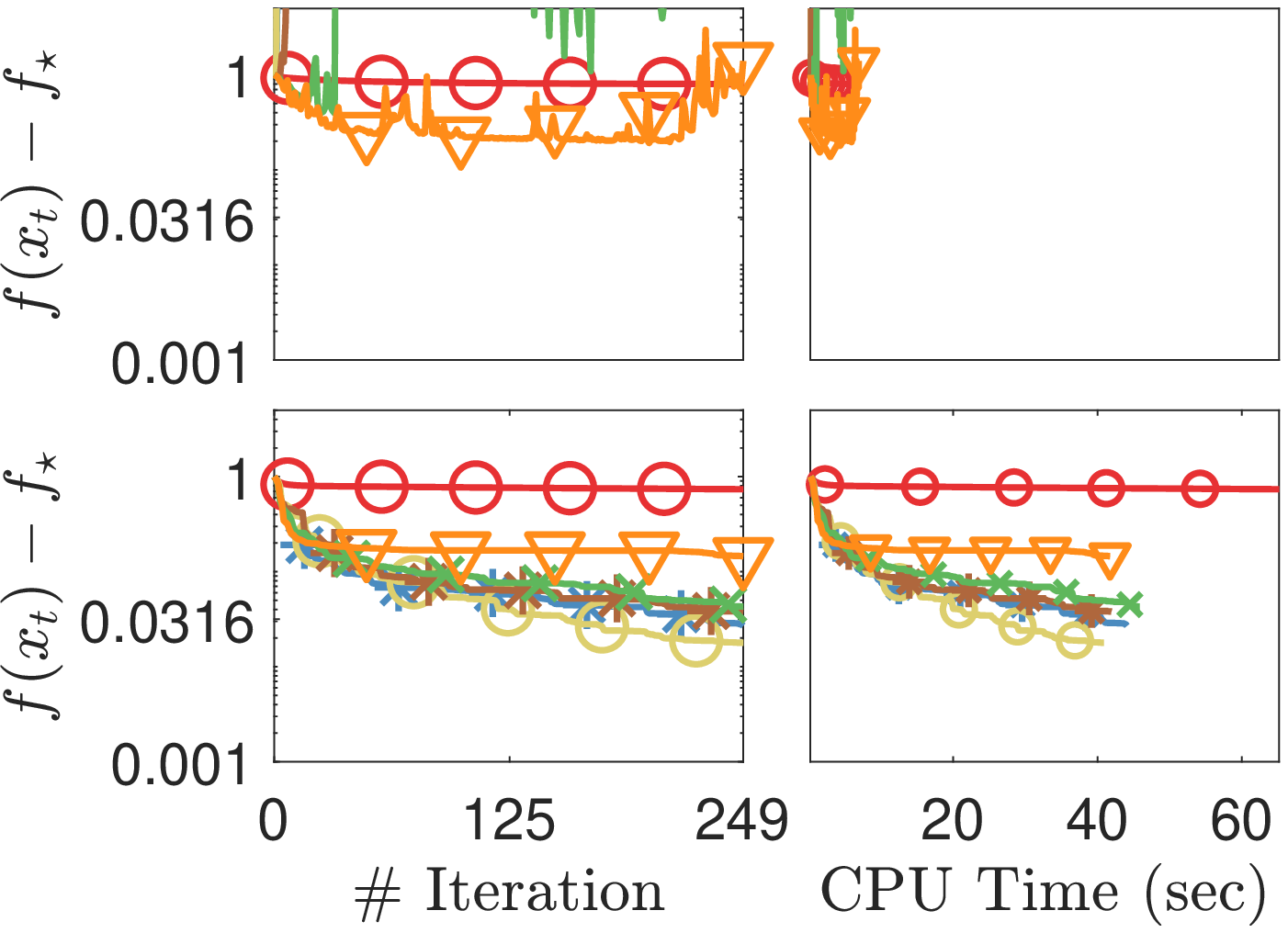}
    \includegraphics[width=0.48\textwidth]{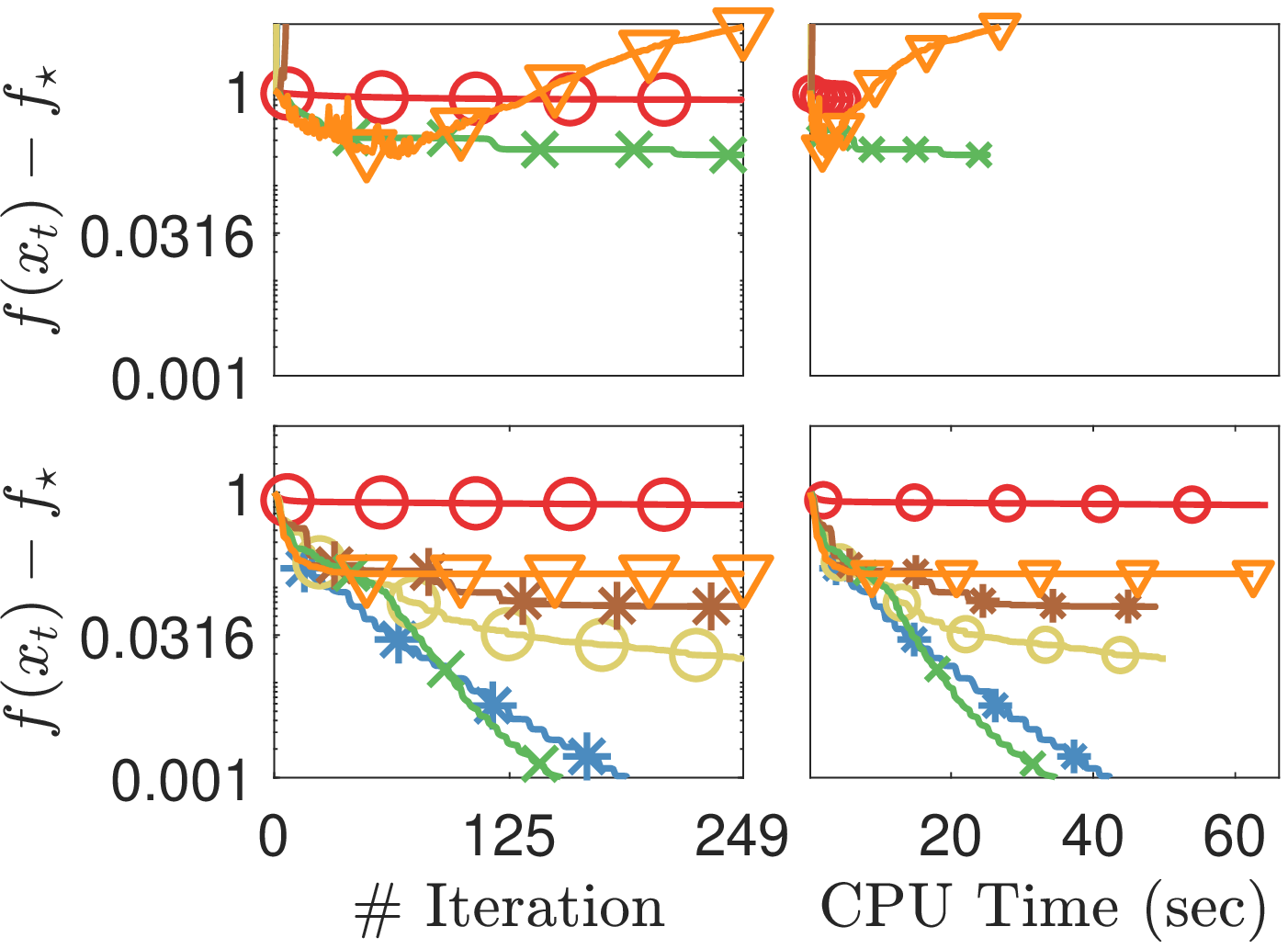}\\
    \includegraphics[width=0.5\textwidth]{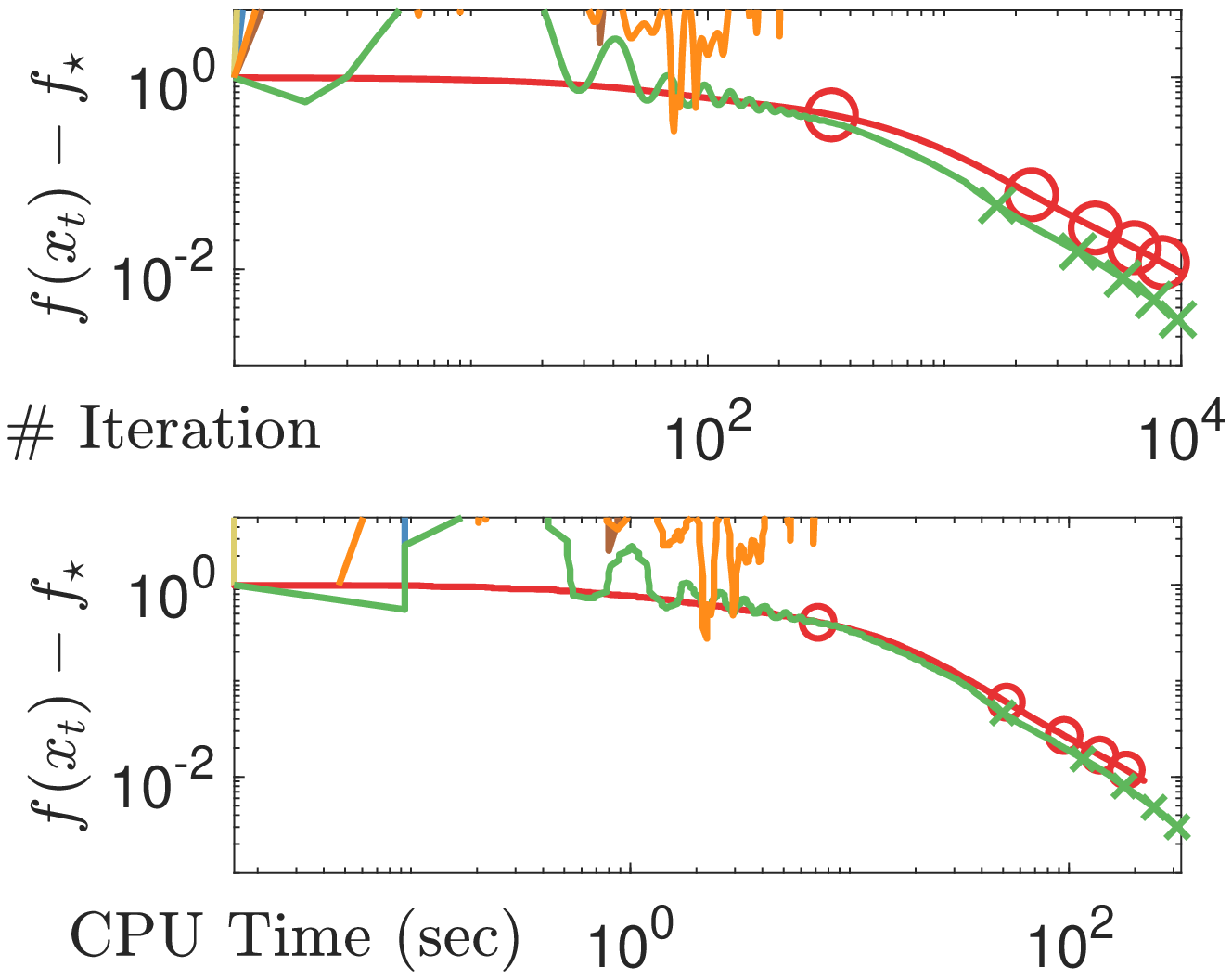}
\end{figure}

\clearpage
\subsection{Qsar} \label{sec:numexp_qsar}
\begin{figure}[h!]
    \centering
    \includegraphics[width=0.48\textwidth]{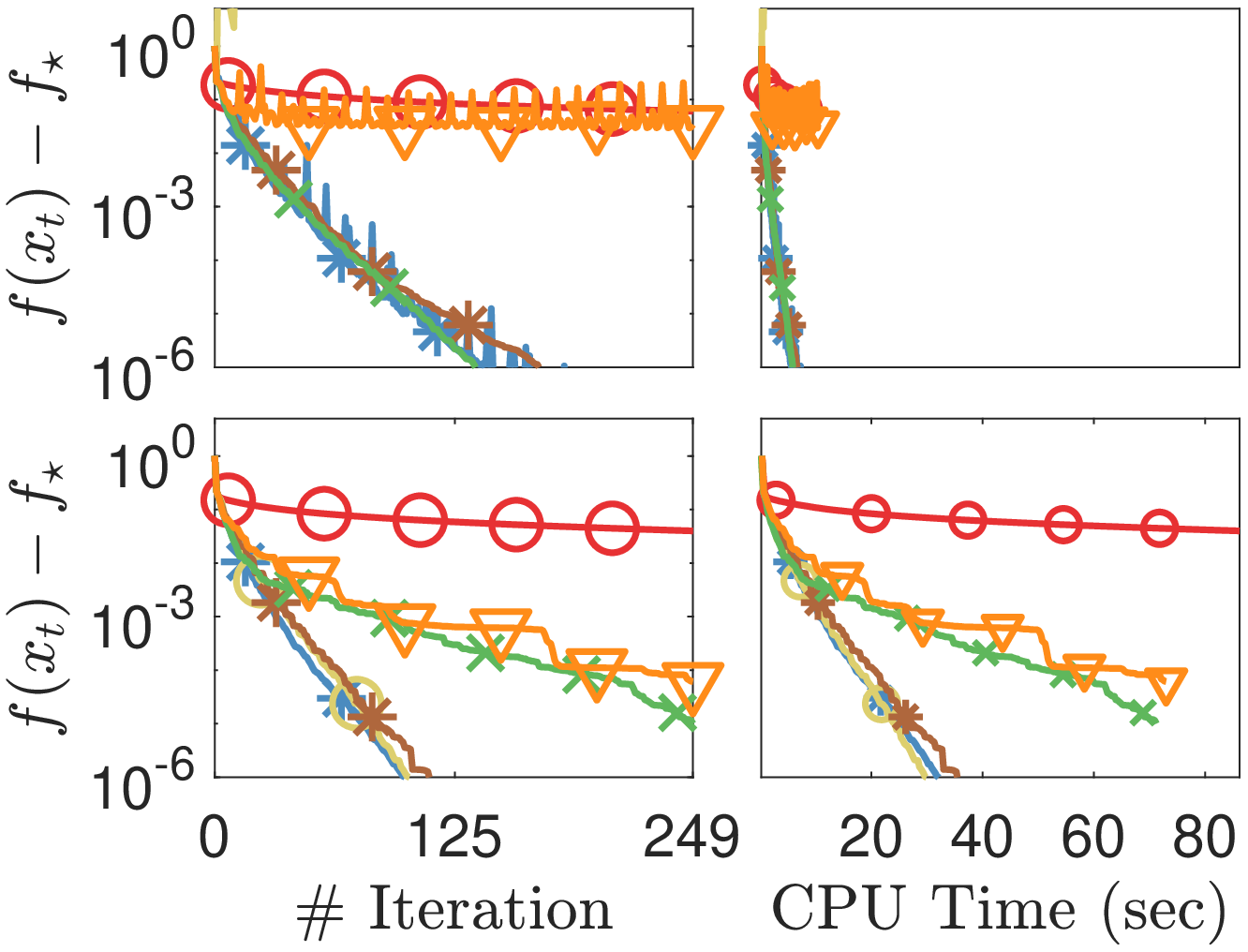}
    \includegraphics[width=0.48\textwidth]{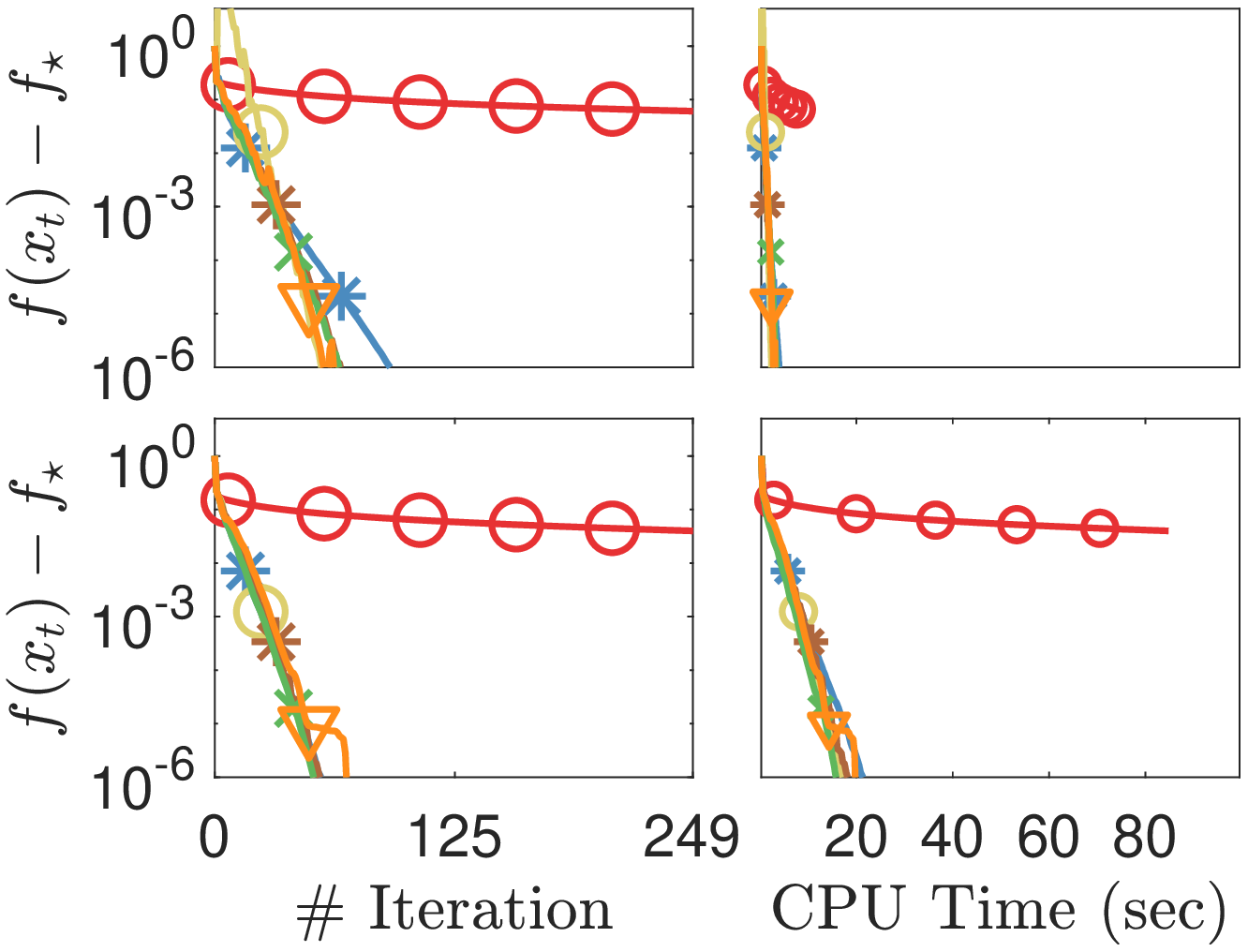}\\
    \includegraphics[width=0.48\textwidth]{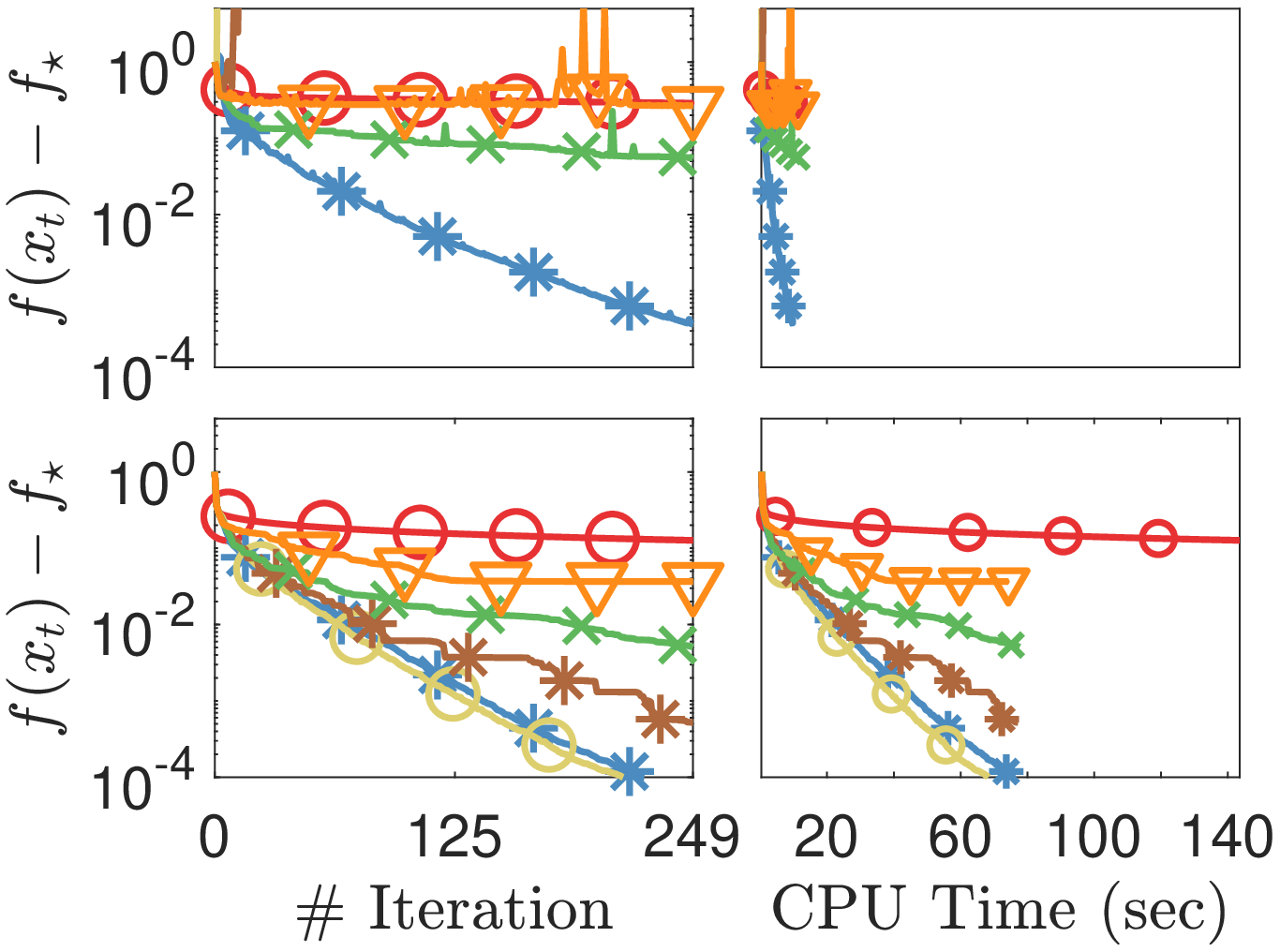}
    \includegraphics[width=0.48\textwidth]{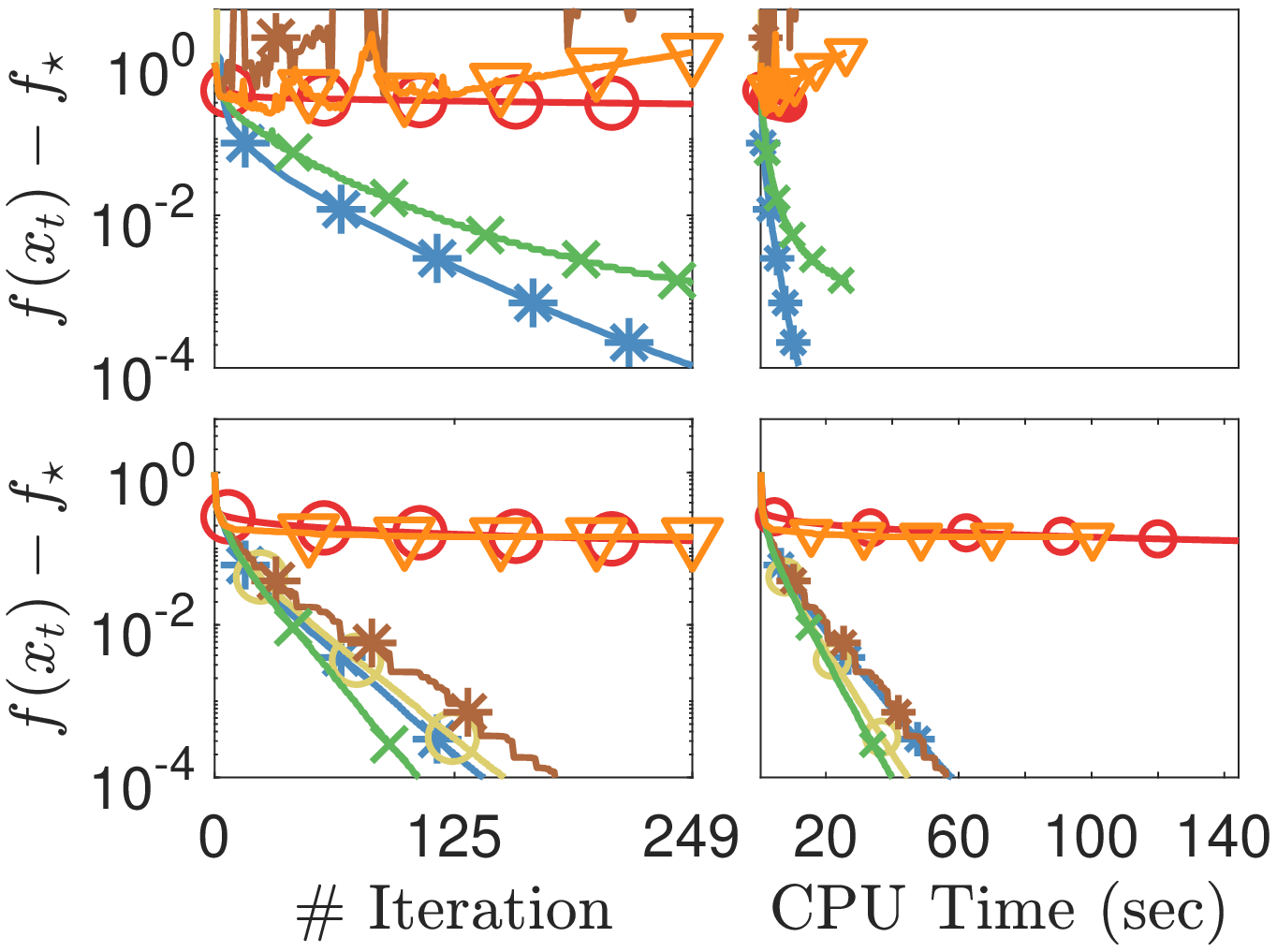}\\
    \includegraphics[width=0.5\textwidth]{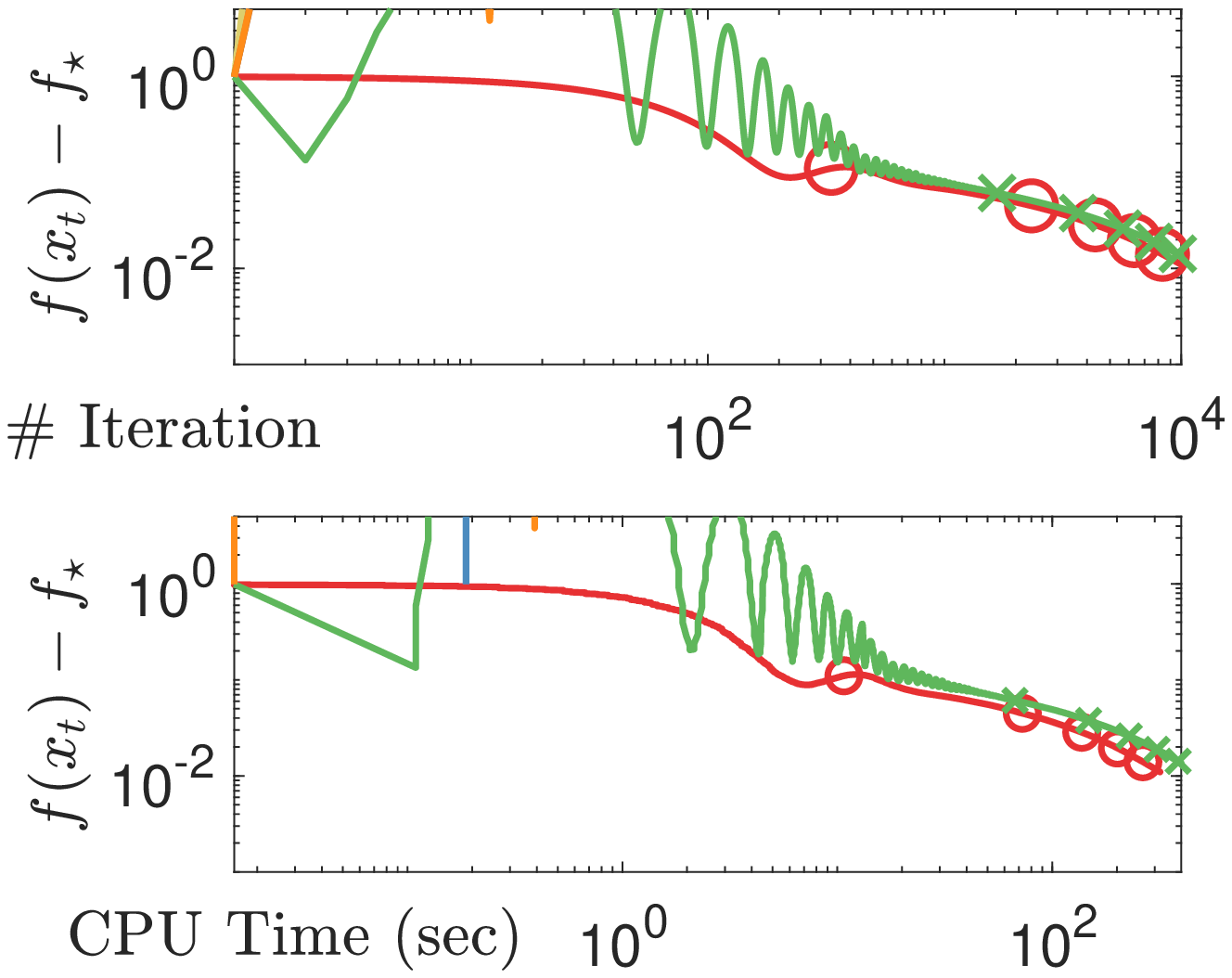}
\end{figure}

\clearpage
\subsection{P53 Mutant} \label{sec:numexp_mutant}
\begin{figure}[h!]
    \centering
    \includegraphics[width=0.48\textwidth]{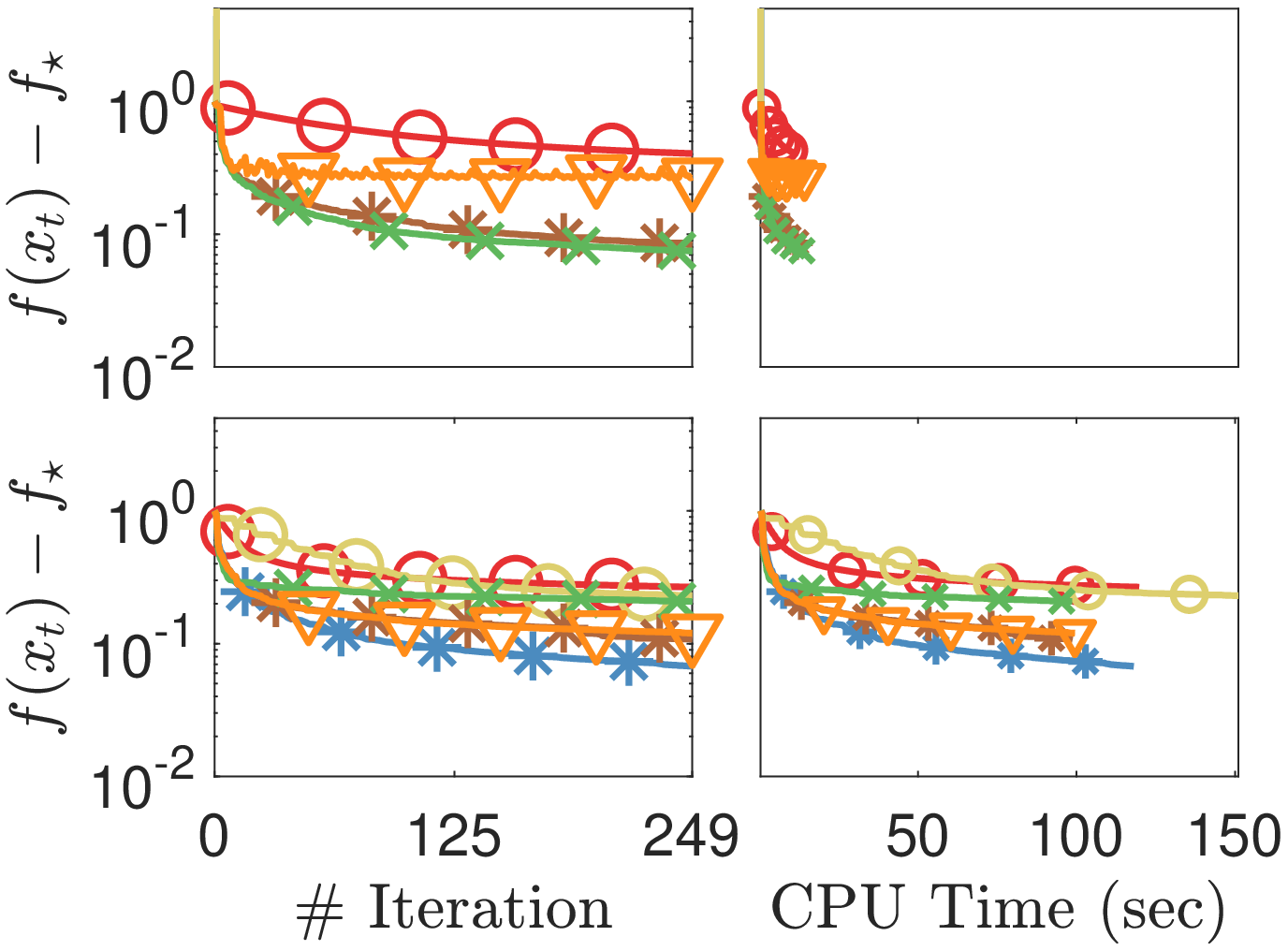}
    \includegraphics[width=0.48\textwidth]{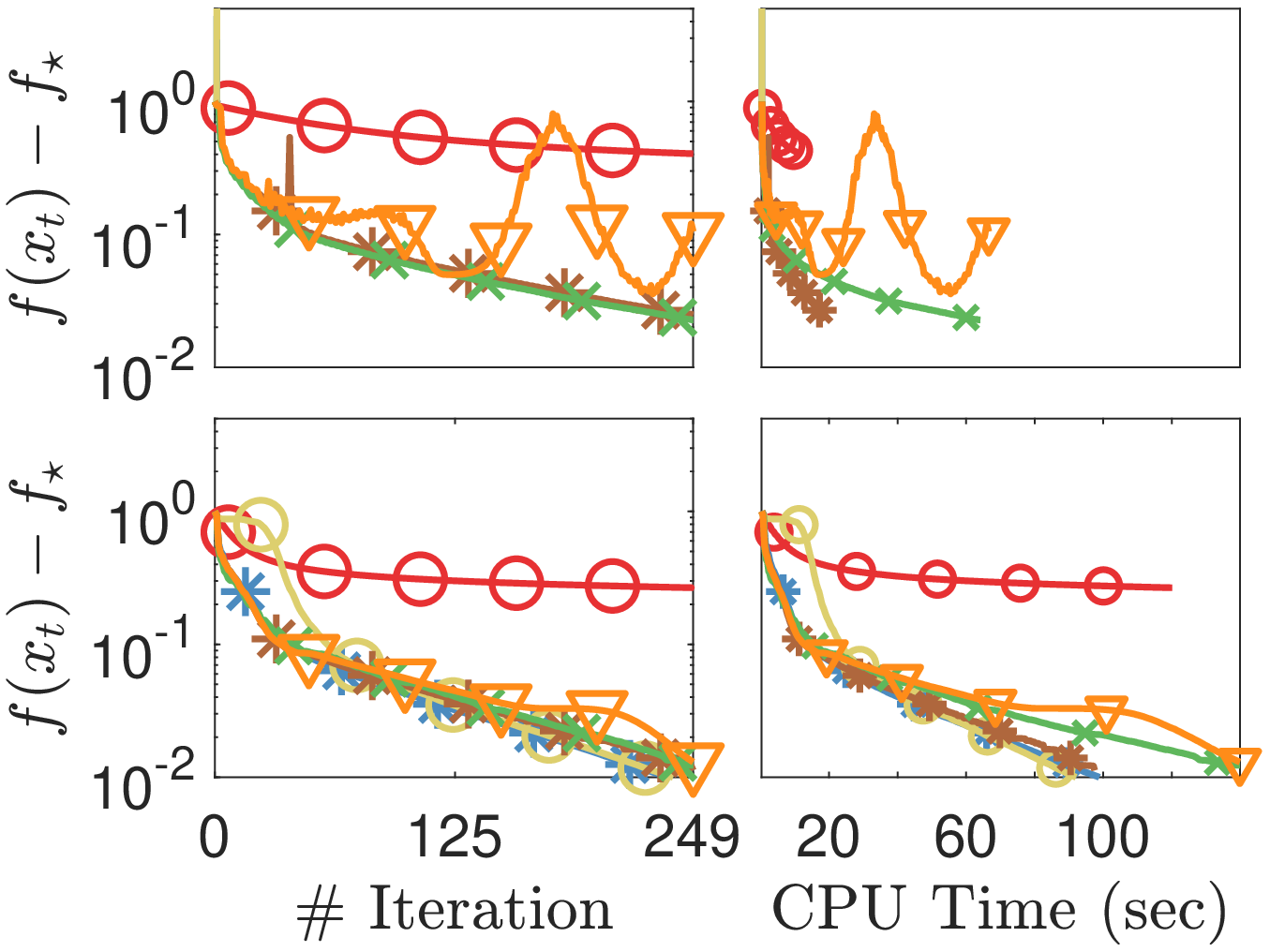}\\
    \includegraphics[width=0.48\textwidth]{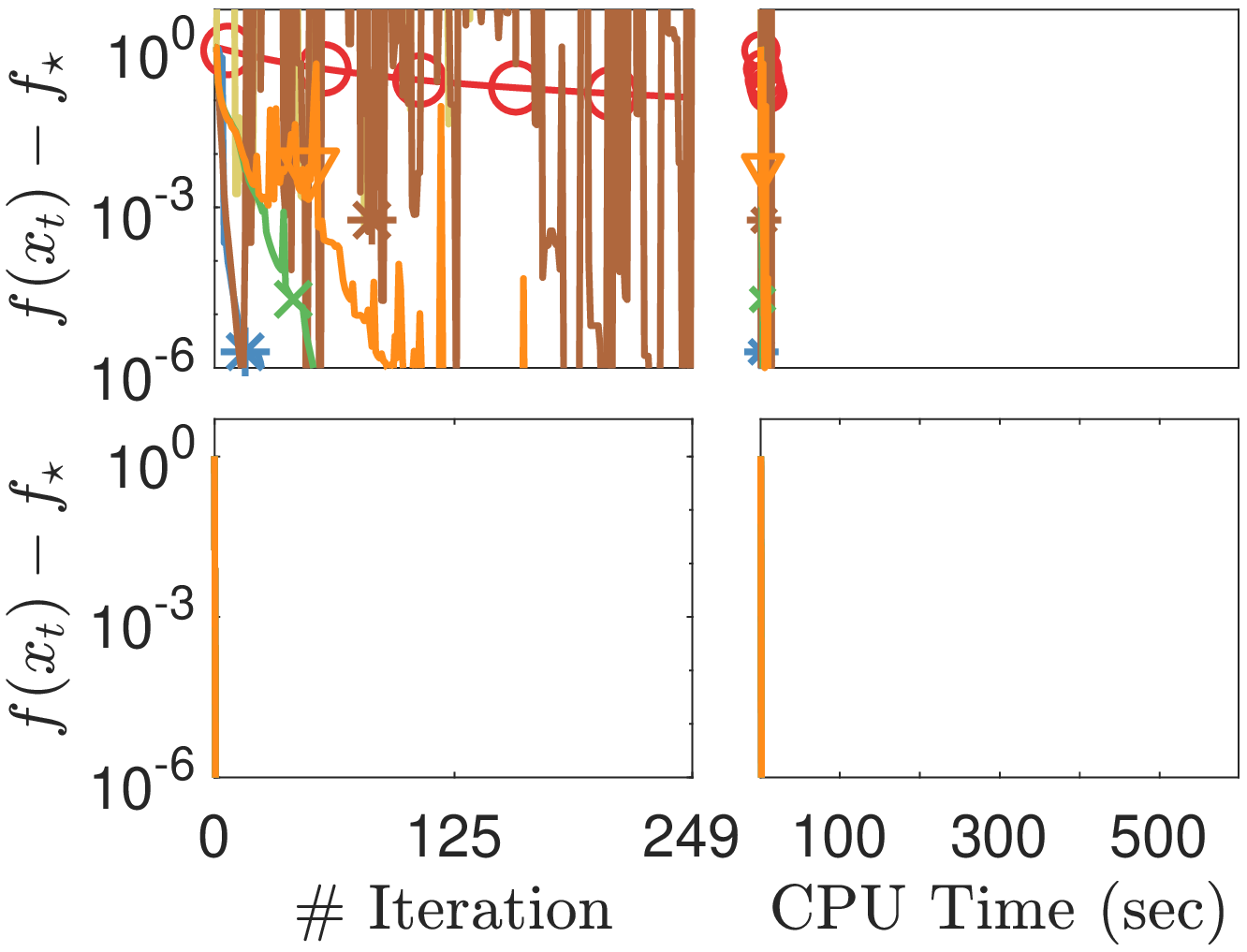}
    \includegraphics[width=0.48\textwidth]{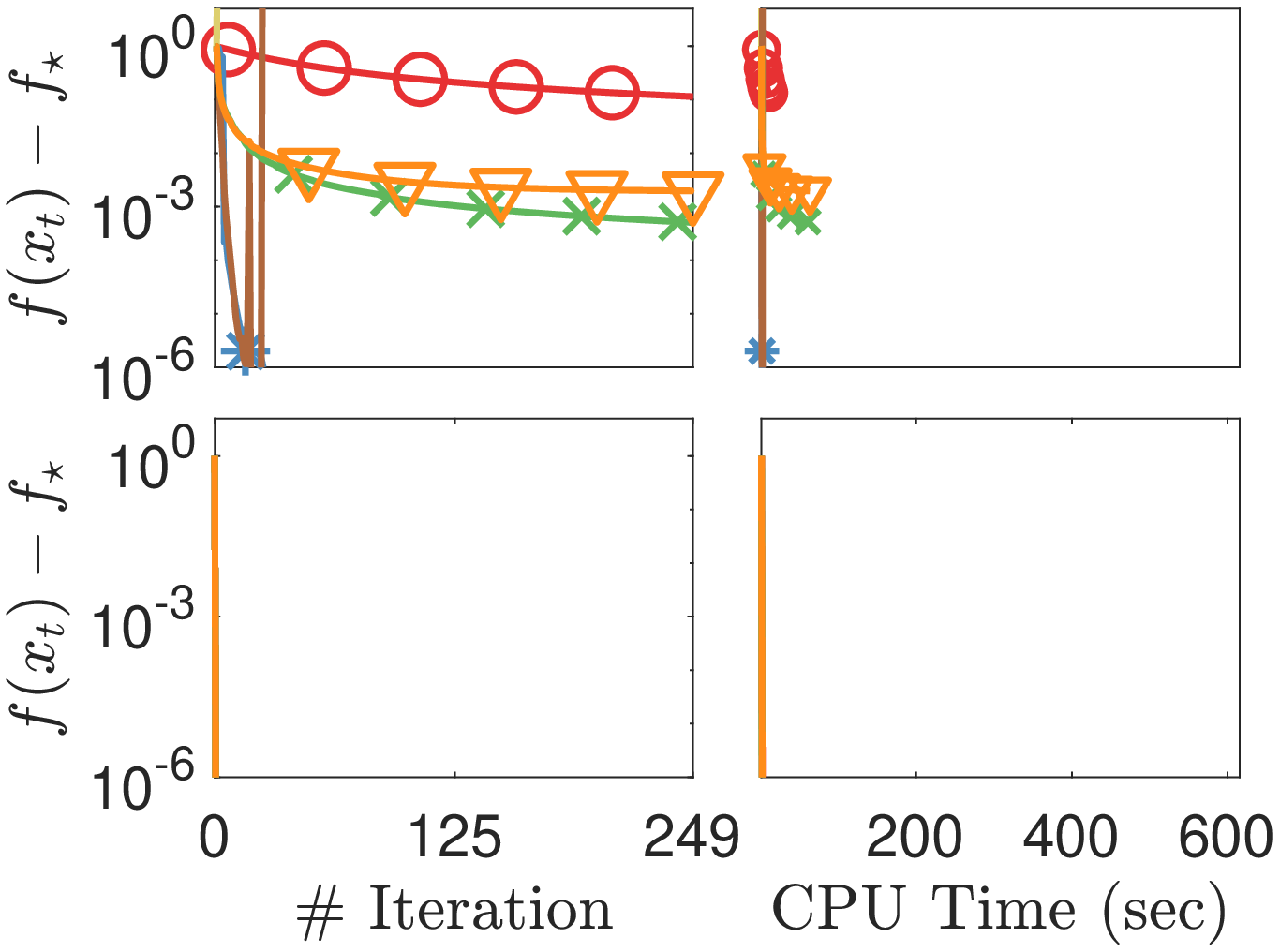}\\
    \includegraphics[width=0.5\textwidth]{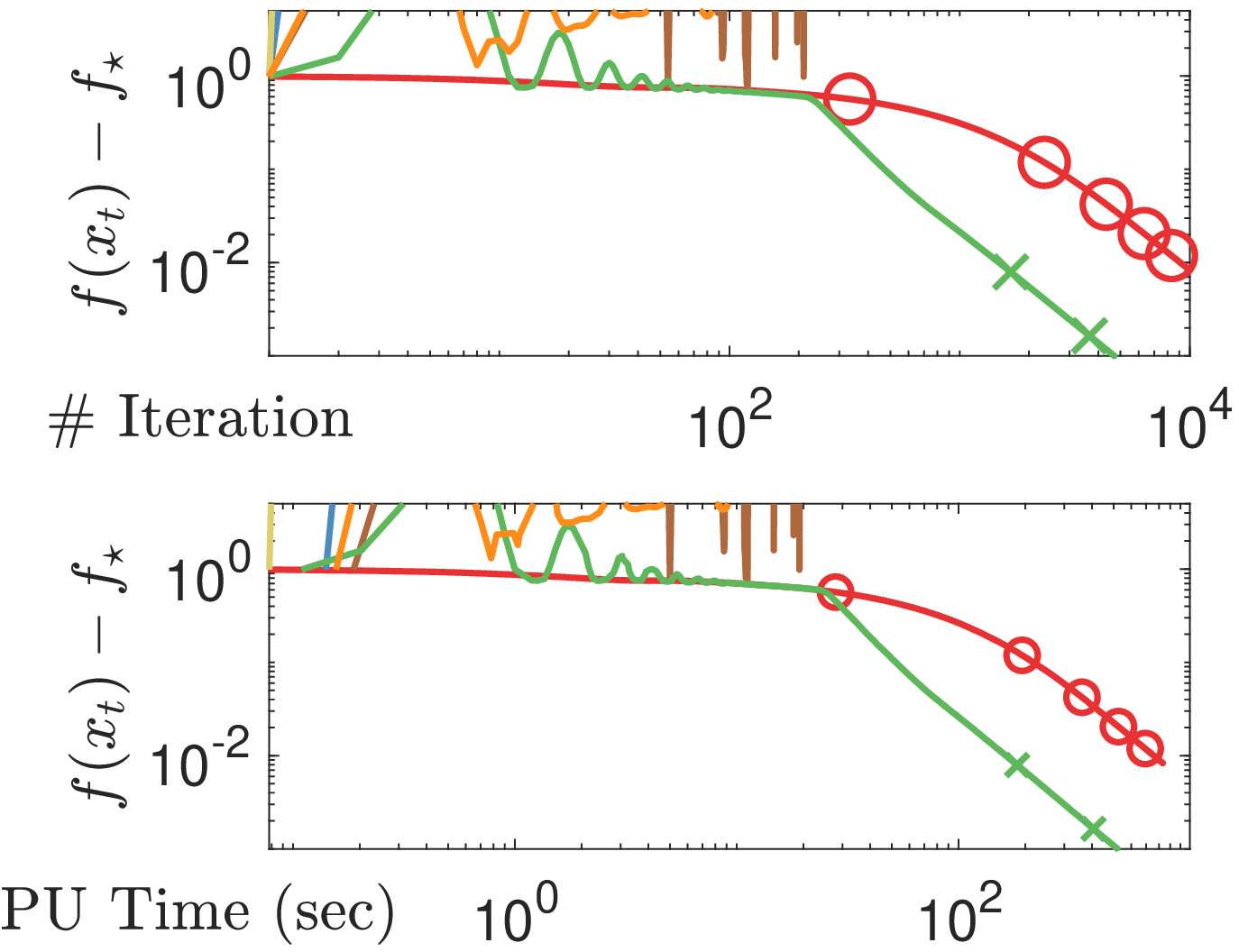}
\end{figure}